\documentclass[10 pt,a4paper]{article}

\pdfoutput=1                      

\usepackage[accsupp]{axessibility}    

\usepackage{latexsym,amsfonts,amsmath,amsthm,amssymb, graphicx, mathrsfs, fancyhdr, makeidx, amscd,  color, mathdesign, fontenc, stix}
\usepackage[all]{xy}
\usepackage[notref,notcite]{showkeys}

\newtheorem{theorem}{Theorem}[section]
\newtheorem{definition}[theorem]{Definition}
\newtheorem{proposition}[theorem]{Proposition}
\newtheorem{lemma}[theorem]{Lemma}
\newtheorem{remark}[theorem]{Remark}
\newtheorem{corollary}[theorem]{Corollary}
\newtheorem{example}[theorem]{Example}
\newtheorem*{question}{Question}
\newtheorem*{notation}{Notational agreements}
\numberwithin{equation}{section}
\numberwithin{theorem}{section}
\newtheorem*{convention}{Convention}

\begin{document}


\newcommand{\riem}{(M^m, \langle \, , \, \rangle)}
\newcommand{\Hess}{\mathrm{Hess}\, }
\newcommand{\hess}{\mathrm{hess}\, }
\newcommand{\cut}{\mathrm{cut}}
\newcommand{\ind}{\mathrm{ind}}
\newcommand{\ess}{\mathrm{ess}}
\newcommand{\longra}{\longrightarrow}
\newcommand{\eps}{\varepsilon}
\newcommand{\ra}{\rightarrow}
\newcommand{\vol}{\mathrm{vol}}
\newcommand{\di}{\mathrm{d}}
\newcommand{\R}{\mathbb R}
\newcommand{\C}{\mathbb C}
\newcommand{\Z}{\mathbb Z}
\newcommand{\N}{\mathbb N}
\newcommand{\HH}{\mathbb H}
\newcommand{\esse}{\mathbb S}
\newcommand{\bull}{\rule{2.5mm}{2.5mm}\vskip 0.5 truecm}
\newcommand{\binomio}[2]{\genfrac{}{}{0pt}{}{#1}{#2}} 
\newcommand{\metric}{\langle \, , \, \rangle}
\newcommand{\lip}{\mathrm{Lip}}
\newcommand{\loc}{\mathrm{loc}}
\newcommand{\diver}{\mathrm{div}}
\newcommand{\disp}{\displaystyle}
\newcommand{\rad}{\mathrm{rad}}
\newcommand{\mmetric}{\langle\langle \, , \, \rangle\rangle}
\newcommand{\sn}{\mathrm{sn}}
\newcommand{\cn}{\mathrm{cn}}
\newcommand{\ink}{\mathrm{in}}
\newcommand{\hol}{\mathrm{H\ddot{o}l}}
\newcommand{\capac}{\mathrm{cap}}
\newcommand{\bmo}{\{b <0\}}
\newcommand{\bmuo}{\{b \le 0\}}
\newcommand{\Fk}{\mathcal{F}_k}
\newcommand{\dist}{\mathrm{dist}}
\newcommand{\gr}{\mathcal{G}}
\newcommand{\grh}{\mathcal{G}^{h}}
\newcommand{\sgrh}{\mathscr{G}^h}
\newcommand{\sgrhe}{\mathscr{G}^{h_\eps}}
\newcommand{\sgr}{\mathscr{G}}
\newcommand{\Ricc}{\mathrm{Ric}}
\newcommand{\Sect}{\mathrm{Sec}}
\newcommand{\roy}{\mathsf{Ro}}
\newcommand{\roypsi}{\mathsf{Ro}_\psi}
\newcommand{\LL}{\mathcal{L}}
\newcommand{\Do}{\mathscr{D}_o}
\newcommand{\Dom}{\mathscr{D}_\Omega}
\newcommand{\BBB}{\mathscr{B}}

\newcommand{\vp}{\varphi}
\renewcommand{\div}[1]{{\mathop{\mathrm div}}\left(#1\right)}
\newcommand{\divphi}[1]{{\mathop{\mathrm div}}\bigl(\vert \nabla #1
\vert^{-1} \varphi(\vert \nabla #1 \vert)\nabla #1   \bigr)}
\newcommand{\nablaphi}[1]{\vert \nabla #1\vert^{-1}
\varphi(\vert \nabla #1 \vert)\nabla #1}
\newcommand{\modnabla}[1]{\vert \nabla #1\vert }
\newcommand{\modnablaphi}[1]{\varphi\bigl(\vert \nabla #1 \vert\bigr)
\vert \nabla #1\vert }
\newcommand{\ds}{\displaystyle}
\newcommand{\cL}{\mathcal{L}}
\newcommand{\essem}{\mathds{S}^m}
\newcommand{\erre}{\mathds{R}}
\newcommand{\errem}{\mathds{R}^m}
\newcommand{\enne}{\mathds{N}}
\newcommand{\acca}{\mathds{H}}

\newcommand{\cvett}{\Gamma(TM)}
\newcommand{\cinf}{C^{\infty}(M)}
\newcommand{\sptg}[1]{T_{#1}M}
\newcommand{\partder}[1]{\frac{\partial}{\partial {#1}}}
\newcommand{\partderf}[2]{\frac{\partial {#1}}{\partial {#2}}}
\newcommand{\ctloc}{(\mathcal{U}, \varphi)}
\newcommand{\fcoord}{x^1, \ldots, x^n}
\newcommand{\ddk}[2]{\delta_{#2}^{#1}}
\newcommand{\christ}{\Gamma_{ij}^k}
\newcommand{\ricc}{\operatorname{Ricc}}
\newcommand{\supp}{\operatorname{supp}}
\newcommand{\sgn}{\operatorname{sgn}}
\newcommand{\rg}{\operatorname{rg}}
\newcommand{\inv}[1]{{#1}^{-1}}
\newcommand{\id}{\operatorname{id}}
\newcommand{\jacobi}[3]{\sq{\sq{#1,#2},#3}+\sq{\sq{#2,#3},#1}+\sq{\sq{#3,#1},#2}=0}
\newcommand{\lie}{\mathfrak{g}}
\newcommand{\rp}{\erre\mathds{P}}
\newcommand{\II}{\operatorname{II}}
\newcommand{\gradh}[1]{\nabla_{H^m}{#1}}
\newcommand{\absh}[1]{{\left|#1\right|_{H^m}}}
\newcommand{\mob}{\mathrm{M\ddot{o}b}}
\newcommand{\mab}{\mathfrak{m\ddot{o}b}}
\newcommand{\foc}{\mathrm{foc}}
\newcommand{\F}{\mathcal{F}}
\newcommand{\Cf}{\mathcal{C}_f}
\newcommand{\cutf}{\mathrm{cut}_{f}}
\newcommand{\Cn}{\mathcal{C}_n}
\newcommand{\cutn}{\mathrm{cut}_{n}}
\newcommand{\Ca}{\mathcal{C}_a}
\newcommand{\cuta}{\mathrm{cut}_{a}}
\newcommand{\cutc}{\mathrm{cut}_c}
\newcommand{\cutcf}{\mathrm{cut}_{cf}}
\newcommand{\rk}{\mathrm{rk}}
\newcommand{\crit}{\mathrm{crit}}
\newcommand{\diam}{\mathrm{diam}}
\newcommand{\haus}{\mathscr{H}}
\newcommand{\po}{\mathrm{po}}
\newcommand{\grg}{\mathcal{G}^{(g)}}
\newcommand{\grj}{\mathcal{G}^{(j)}}
\newcommand{\Sph}{\mathbb{S}}
\newcommand{\BB}{\mathbb{B}}
\newcommand{\EE}{\mathbb{E}}
\newcommand{\Gia}{G^{(a)}}
\newcommand{\Gib}{G^{(b)}}
\newcommand{\tr}{\mathrm{Tr}}
\newcommand{\dou}{\mathrm{(VD)}}
\newcommand{\dous}{\mathrm{(VD)}}
\newcommand{\neuqp}{(\mathrm{NP}_{q,p})}
\newcommand{\neup}{(\mathrm{NP}_{p,p})}
\newcommand{\neuq}{(\mathrm{NP}_{q,q})}
\newcommand{\neuuno}{(\mathrm{NP}_{1,1})}
\newcommand{\neuunop}{(\mathrm{NP}_{1,p})}

\newcommand{\uniho}{\mathrm{(UO)}}
\newcommand{\vcom}{\mathrm{(VC)}}
\newcommand{\pvcom}{\mathrm{(PVC)}}
\newcommand{\tcr}{\textcolor{red}}
\newcommand{\tcb}{\textcolor{blue}}
\newcommand{\tcg}{\textcolor{green}}
\newcommand{\ol}{\overline}
\newcommand{\ul}{\underline}
\newcommand{\Hp}{\mathscr{H}_p}
\newcommand{\vut}{\{\varrho=t\}}
\newcommand{\vmt}{\{\varrho < t\}}
\newcommand{\vutp}{\{\varrho_p=t\}}
\newcommand{\vmtp}{\{\varrho_p \le t\}}
\newcommand{\vus}{\{\varrho=s\}}
\newcommand{\vms}{\{\varrho \le s\}}
\newcommand{\Ar}{\mathscr{A}}
\newcommand{\VV}{\mathscr{V}}
\newcommand{\metricN}{( \, , \, )}
\newcommand{\Po}{\mathscr{P}}
\newcommand{\Ha}{\mathscr{H}}
\newcommand{\So}{\mathscr{S}}
\newcommand{\Dou}{\mathscr{D}}
\newcommand{\Vo}{\mathscr{V}}
\newcommand{\RD}{\mathscr{R}}
\newcommand{\DM}{\mathcal{DM}}

\author{Luciano Mari \and Marco Rigoli \and Alberto G. Setti}
\title{\textbf{On the $1/H$-flow by $p$-Laplace approximation: new estimates via fake distances under Ricci lower bounds}}
\date{}
\maketitle

\scriptsize \begin{center} Dipartimento di Matematica, Universit\`a degl Studi di Torino,\\
Via Carlo Alberto 10, 10123 Torino (Italy)\\
E-mail: luciano.mari@unito.it
\end{center}

\scriptsize \begin{center} Dipartimento di Matematica,
Universit\`a
degli Studi di Milano,\\
Via Saldini 50, I-20133 Milano (Italy)\\
E-mail: marco.rigoli55@gmail.com
\end{center}

\scriptsize \begin{center} Dipartimento di Scienza e Alta Tecnologia,
Universit\`a degli Studi dell'Insubria,\\
Via Valleggio 11, I-22100 Como (Italy)\\
E-mail: alberto.setti@uninsubria.it
\end{center}

\normalsize

\begin{abstract}
In this paper we show the existence of weak solutions $w : M \ra \R$ to the inverse mean curvature flow starting from a relatively compact set (possibly, a point) on a large class of manifolds satisfying Ricci lower bounds. Under natural assumptions, we obtain sharp estimates for the growth of $w$ and for the mean curvature of its level sets, which are well behaved with respect to Gromov-Hausdorff convergence. The construction follows R. Moser's approximation procedure via the $p$-Laplace equation, and relies on new gradient and decay estimates for $p$-harmonic capacity potentials, notably for the kernel $\gr_p$ of $\Delta_p$. These bounds, stable as $p \ra 1$, are achieved by studying fake distances associated to capacity potentials and Green kernels. We conclude by investigating some basic isoperimetric properties of the level sets of $w$.
\end{abstract}

\tableofcontents

\section{Introduction}

\begin{notation}
\emph{We set $\R^+ \doteq (0,\infty)$, $\R^+_0 \doteq [0, \infty)$. Given two positive functions $f,g : U \ra \R$, we say that $f \asymp g$ on $U$ if $C^{-1}f \le g \le Cf$ on $U$ for some constant $C>0$, while we use $\sim$ to denote the usual asymptotic relation. Given sets $U, \Omega$, we write $U \Subset \Omega$ to denote that $U$ has compact closure in $\Omega$. If $(M,g)$ is a Riemannian manifold and $f: M \to \R$, we write $\Ricc \ge f$ to mean $\Ricc \ge fg$ in the sense of quadratic forms.
}
\end{notation}

The inverse mean curvature flow (IMCF) is an effective tool to study the geometry of manifolds $M$ whose behaviour at infinity is controlled in a precise way. This is the case, remarkably, of asymptotically flat and hyperbolic manifolds in general relativity. The purpose of the present paper is to show the existence of weak solutions to  the IMCF, with sharp estimates, on complete manifolds $(M^m, \metric)$ only satisfying mild conditions at infinity, making the tool amenable to study the geometry in the large of manifolds with a Ricci lower bound. In fact, we only impose a lower bound on the Ricci curvature, together with the validity of (weighted) Sobolev inequalities or some lower bound on the volume of balls centered at a fixed origin. In particular, no control on the sectional curvature is needed, a feature that makes our techniques  robust enough, for instance, to pass to limits with respect to pointed Gromov-Hausdorff convergence and produce an IMCF on Ricci limit spaces.\par

Classically, a family of two sided hypersurfaces $F : [0,T] \times \Sigma^{m-1} \ra M^m$ evolves by IMCF provided that
\begin{equation}\label{def_IMCF}
\frac{\partial F}{\partial t} = \frac{\nu_t}{\mathcal{H}_t},
\end{equation}
where $\nu_t$ is a choice of unit normal to $\Sigma_t = F(t, \Sigma)$ and $\mathcal{H}_t = \mathrm{tr}_{\Sigma_t}(\nabla \nu_t)$ is the mean curvature in the direction $-\nu_t$, which is assumed to be positive at the initial time. The possible formation of singularities when $\inf_{\Sigma_t} \mathcal{H}_t \ra 0$ prompted G. Huisken and T. Ilmanen \cite{huiskenilmanen} to introduce the notion of weak solutions to \eqref{def_IMCF}.
The strategy is to look for a proper function $w :  M \ra \R$ solving
\begin{equation}\label{def_1Lapla}
\Delta_1 w \doteq  \diver \left( \frac{\nabla w}{|\nabla w|} \right) = |\nabla w|
\end{equation}
in a suitable weak sense, and to consider the sets $\Sigma_t = \partial \{w<t\}$. We recall that $w : M \ra \R$ is said to be proper if its sublevel sets $\{w \le t\}$ are compact for each $t \in \R$. In particular, $\Sigma_t$ is also compact. If $w$ is smooth and $|\nabla w| \neq 0$ on $\Sigma_{t_0}$, then locally around $t_0$ the family $\Sigma_{t}$ is the unique smooth solution to  \eqref{def_IMCF}, and the mean curvature $\mathcal{H}_t = |\nabla w|$ points in the direction of $-\nabla w$ (hence, every $\Sigma_t$ has positive mean curvature). In \cite{huiskenilmanen}, the authors defined a weak solution to  the IMCF starting from a given relatively compact open subset $\Omega$ to be a function $w : M \ra \R$ satisfying
\begin{itemize}
\item[-] $w \in \lip_\loc(M)$, $\ \Omega = \{w<0\}$;
\item[-] for each $\phi \in \lip_\loc(M \backslash \Omega)$ with $\mathrm{supp}(\phi-w) \Subset M \backslash \overline{\Omega}$, and for each compact $K$ containing the support of $\phi-w$,
\begin{equation}\label{minimization_1Lapla}
\int_K |\nabla w| + w |\nabla w| \le \int_K |\nabla \phi| + \phi |\nabla w|.
\end{equation}
\end{itemize}
Subsolutions and supersolutions for \eqref{def_1Lapla} are defined accordingly, by requiring that \eqref{minimization_1Lapla} holds only for competitors $\phi$ satisfying $\phi \le w$, respectively $\phi \ge w$. They correspond to weak solutions to  $\Delta_1 w \ge |\nabla w|$, respectively $\Delta_1 w \le |\nabla w|$. The following important existence result holds.

\begin{theorem}[\cite{huiskenilmanen}, Thm. 3.1]\label{teo_HI}
Let $M^m$ be a complete manifold and let $\Omega \Subset M$ be a relatively compact, open set with $\partial \Omega \in C^1$. If there exists a proper, $\lip_\loc$ weak subsolution $\bar w$ for \eqref{def_1Lapla} defined outside of a relatively compact set, then there exists a proper solution $w$ to the IMCF with initial condition $\Omega$. This solution is unique in $M \backslash \Omega$.
\end{theorem}

The higher regularity of the flow was studied in \cite{huiskenilmanen_2}. As explained in \cite[Thm 2.2]{huiskenilmanen}, the properness of $w$ guarantees its uniqueness. Theorem \ref{teo_HI} was applied with remarkable success to study the Riemannian Penrose inequality in the setting of asymptotically flat manifolds \cite{huiskenilmanen}. In this case, as well as for asymptotically conical or hyperbolic manifolds, the existence of a proper subsolution $\bar w$ is easy to establish. However, on more general manifolds, barriers like $\bar w$ are much harder to find: {for instance, to produce a function $\bar w$ of the form $\bar {\rm  w}(r)$, with $r$ the distance from a smooth compact set $K$, the properness requirement forces the set $\mathscr{W} = \{ x : \bar {\rm  w}'(r(x)) >0\}$ to be non-empty. On $\mathscr{W}$, the inequality $\Delta_1 \bar w \geq |\nabla \bar w|$ becomes $\Delta r \ge \bar {\rm  w}'(r)$ which entails a bound for $\Delta r$ \emph{from below}.
Comparison theorems then call for an upper bound for the sectional curvature, together with the fact that the normal exponential map from $K$ be a diffeomorphism.} Conditions of this type are available, for suitable $K$, in the above relevant classes of manifolds, while they are definitely too restrictive in the setting that we are going to consider.

A further reason that makes their result hard to use in our setting regards the convergence procedure to construct $w$, that is based on local $C^1$ estimates, independent of $\eps>0$, for solutions $w_\eps$ of the approximating problems
\begin{equation}\label{approx_meancurv}
\left\{ \begin{array}{l}
\disp \diver \left( \frac{\nabla w_\eps}{\sqrt{\eps^2 + |\nabla w_\eps|^2}}\right) = \sqrt{ \eps^2 + |\nabla w_\eps|^2} \qquad \text{on } \, \{ \bar w < L\} \\[0.5cm]
w_\eps = 0 \quad \text{on } \, \partial \Omega, \qquad w_\eps = L-2 \quad \text{on } \, \partial \{\bar w<L\},
\end{array}\right.
\end{equation}
where $L \in \R^+$ and $\eps$ is suitably small. The gradient estimate proved by the authors, that is,
\begin{equation}\label{HI_gradesti}
|\nabla w(x)| \le \sup_{\partial \Omega \cap B_r(x)} \mathcal{H}_+ + \frac{C(m)}{r} \qquad \text{for a.e. } \, x \in M \backslash \Omega,
\end{equation}
is restricted to radii $r$ for which there exists $\rho \in C^2(B_r(x))$ touching from above, at $x$, the squared distance function $r_x^2$ from $x$ and satisfying $\nabla^2 \rho \le 3 \metric$ on $B_r(x)$. To guarantee the existence of $\rho$ and the validity of the above Hessian bound, assumptions stronger than a control on Ricci from below seem unavoidable. For instance, if $\rho = r_x^2$, the Hessian Comparison Theorem would require a sectional curvature lower bound. On the other hand, the arguments in \cite[Prop. 1.3]{rimovero} and \cite{imperimovero} could be used to construct $\rho$ on manifolds with positive injectivity radius and satisfying a two-sided control on Ricci. To the best of our knowledge, local gradient estimates for mean curvature type equations like \eqref{approx_meancurv} assuming only a Ricci lower bound are still unknown, and the problem seems challenging.

For these reasons, we adopt a different strategy and follow the beautiful idea described by R. Moser in \cite{moser, moser_2, moser_3} to approximate the solution to \eqref{def_1Lapla} via solutions to the $p$-Laplace equation. Namely, given $\Omega \Subset M$ and $p>1$, let $u_p$ be the $p$-capacity potential of $\Omega$, that is, the minimal positive solution to
\begin{equation}\label{p-potential}
\left\{ \begin{array}{l}
\Delta_p u_p \doteq \diver \left( |\nabla u_p|^{p-2} \nabla u_p \right) =0 \qquad \text{on } \, M \backslash \overline\Omega \\[0.4cm]
u_p = 1 \quad \text{on } \, \partial \Omega.
\end{array}\right.
\end{equation}
Under suitable conditions at infinity on $M$, $u_p \not \equiv 1$ and therefore $0< u_p < 1$ on $M \backslash \overline{\Omega}$ by the maximum principle. Changing variables according to $w_p = (1-p) \log u_p$, we obtain
\begin{equation}\label{eq_up_intro}
\left\{ \begin{array}{l}
\Delta_p w_p = |\nabla w_p|^p \qquad \text{on } \, M \backslash \overline\Omega, \\[0.4cm]
w_p = 0 \quad \text{on } \, \partial \Omega, \qquad w_p >0 \quad \text{on } \, M \backslash \overline\Omega.
\end{array}\right.
\end{equation}
The analogy between \eqref{def_1Lapla} and \eqref{eq_up_intro} suggests to construct $w$ as a locally uniform limit of $w_p$ as $p \ra 1$. In this case, passing to the limit in the weak formulation of \eqref{eq_up_intro} one easily deduces that $w$ is a weak solution to  the IMCF, provided that it never vanishes. In \cite{moser}, set in Euclidean space, the properness of $w$ is achieved by means of suitable barriers. Moser's approach to existence was later extended by B. Kotschwar and L. Ni \cite{kotschwarni} on manifolds with a lower sectional curvature bound, notably on manifolds with asymptotically non-negative sectional curvature. Since here no barrier is available, the properness of $w$ becomes a subtle issue, addressed via deep results by P. Li and L.F. Tam \cite{litam_harm} and I. Holopainen \cite{holopainen2}. The convergence to the solution as $p\to 1$ is based on a sharp local estimate for $|\nabla w_p|$, obtained in \cite{kotschwarni} by refining the Cheng-Yau's technique (see \cite{chengyau,yau,liwang}) to apply to $p$-harmonic functions. However, if $p\neq 2$, it seems difficult to modify their argument, as well as those in \cite{moser}, to avoid the requirement of a lower bound on the sectional curvature. Although some progress was recently made by different approaches, see \cite{wangzhang, sungwang}, this is still not enough for our purposes.
%
\subsection*{Our setting, the fake distance, and gradient estimates}
In the present work, we investigate the existence of the IMCF on complete manifolds $(M^m, \metric)$ of dimension $m \ge 2$ satisfying the Ricci lower bound
\begin{equation}\label{Ricci_lower_intro}
\Ricc \ge -(m-1) H(r) \qquad \text{on } \, M,
\end{equation}
where $r$ is the distance from some fixed origin $o \in M$ and $H \in C(\R^+_0)$ is such that
\begin{equation}\label{ipo_h_intro}
H \ge 0 \quad \text{and is non-increasing on $\R^+$.}
\end{equation}
We study the IMCF starting from a relatively compact, open set $\Omega$ with $C^2$ boundary, as well as the one starting from the origin $o$. The latter reveals to be particularly subtle in view of the singular nature of the initial data. A main technical tool is a new, sharp estimate for $|\nabla \log u|$, where $u$ is a positive $p$-harmonic function defined on an open subset of $M$, with special attention to the the case where the domain is the complement of $o$. To describe our results, we consider the model manifold $M_h$, which is diffeomorphic to $\R^m$ with polar coordinates $(t,\theta) \in \R^+ \times \mathbb{S}^{m-1}$ outside of the origin and with the radially symmetric metric
$$
\di t^2 + h(t)^2 \di \theta^2,
$$
where $\di \theta^2$ is the round metric of curvature $1$ on $\mathbb{S}^{m-1}$ and $h$ solves
$$
\left\{ \begin{array}{ll}
h'' = Hh & \quad \text{on } \, \R^+, \\[0.2cm]
h(0)=0, \quad h'(0)=1.
\end{array}\right.
$$
For instance, if $H(t) = 0$ then $h(t)=t$ and $M_h$ is Euclidean space $\R^m$, while if $H(t) = \kappa^2$ for a positive constant $\kappa$, then $M_h$ is Hyperbolic space of curvature $-\kappa^2$.
Denote by
$$
v_h(t) = \omega_{m-1}h(t)^{m-1}
$$
the volume of the sphere of radius $t$ centered at the origin of $M_h$, $\omega_{m-1}$ being the volume of $\mathbb{S}^{m-1}$. Condition \eqref{Ricci_lower_intro} allows us to relate $M$ to $M_h$ by means of standard comparison results, which will be extensively used in this paper.\par
Focusing on the singular case $\overline{\Omega} = \{o\}$ for ease of presentation, the role of the $p$-capacity potential in \eqref{p-potential} is played by the minimal positive Green kernel $\gr$ of $\Delta_p$ with pole at $o$, which solves $\Delta_p \gr = -\delta_o$ on $M$, where $\delta_o$ is the Dirac-delta measure at $o$. The existence of $\gr$ is one of the equivalent characterizations of the fact that $\Delta_p$ is non-parabolic on $M$, a property customarily introduced in terms of the $p$-capacity of compact sets and recalled in Section \ref{sec_preliminaries} below. By \eqref{Ricci_lower_intro} and  comparison (see Proposition \ref{prop_parabnonparab} below), $\Delta_p$ is non-parabolic on $M_h$ as well, equivalently
$$
\sgrh(t) \doteq \int_t^\infty v_h(s)^{-\frac{1}{p-1}} \di s < \infty.
$$
Indeed, $\sgrh(t)$ is the minimal positive Green kernel of $\Delta_p$ on $M_h$ with pole at the origin. To bound $|\nabla \log \gr|$, inspired by \cite{colding}, we reparametrize the level sets of $\gr$ in terms of a function that mimics the distance from $o$: since we restrict to $p \le m$, by classical work of Serrin \cite{Serrin_1, Serrin_2}, $\gr(x) \asymp \sgrh(r(x))$ in a neighbourhood of $o$.
In particular, both $\gr$ and $\sgrh(r)$ diverge as $r(x) \ra 0$, and therefore we can define implicitly  $\varrho : M \backslash \{o\} \ra \R^+$ by the formula
\begin{equation}\label{def_fake_intro}
\gr(x) = \sgrh\big( \varrho(x)\big).
\end{equation}
Note that $\varrho$ is proper if and only if $\gr(x) \ra 0$ as $x$ diverges. Clearly, $\varrho = r$ when $M = M_h$, and for this reason we call $\varrho$ a \emph{fake distance (from the point $o$)}. When needed, we will write $\gr_p,\varrho_p$ to emphasize their dependence on $p$. Since $w_p$ in \eqref{eq_up_intro} corresponds to $(1-p)\log \gr_p$, and $\varrho_p$ is a reparametrization of $\gr_p$, the estimates required to produce a solution to  the IMCF are equivalent to a local $C^1$ bound and a local lower bound on $\varrho_p$, both uniform for $p$ close to $1$. The latter serves to guarantee that $\varrho_p$ does not vanish identically in the limit $p \ra 1$.\par
In the literature, when $p=2$ fake distances have been used very successfully to study the geometry in the large of manifolds with non-negative Ricci curvature, see for instance \cite{cheegercolding,coldingmini,colding} and  references therein. Also, they have been independently considered in \cite{bmr5} to study the Yamabe problem on non-compact manifolds with nontrivial topology. Indeed, a key observation is the following identity, valid for each $\psi \in C^2(\R)$ with never vanishing first derivative:
\begin{equation}\label{deltapvarrho}
\Delta_p \psi(\varrho) = \left[ (p-1)\psi'' + \frac{v_h'(\varrho)}{v_h(\varrho)}\psi'\right] |\nabla \varrho|^p.
\end{equation}
Since the expression in square brackets is the $p$-Laplacian of $\psi(t)$ on the model $M_h$, \eqref{deltapvarrho} allows to radialize with respect to $\varrho$ in cases where an analogous procedure with respect to the distance $r$ would require the use of comparison theorems from below, hence binding topological assumptions. This is effective, for instance, when studying Yamabe type equations (cf. \cite{bmr5}, for $p=2$) or the validity of the compact support principle (cf. \cite[Sec. 7]{bmpr}). We refer to \cite[Sect. 2]{bmpr} for further information.\par
Finding global gradient estimates for $\varrho$ is needed both to produce the IMCF and to be able to exploit \eqref{deltapvarrho}. One of the main achievements of the present paper is the following sharp gradient estimate, see Theorem \ref{teo_good} below.

\begin{theorem}\label{teo_good_intro}
Assume that $M^m$ satisfies \eqref{Ricci_lower_intro} and \eqref{ipo_h_intro}, that $p \le m$ and that $\Delta_p$ is non-parabolic on $M$. Then, having defined $\varrho$ as in \eqref{def_fake_intro},
\begin{itemize}
\item[(i)] $|\nabla \varrho| \le 1$ on $M \backslash \{o\}$;
\item[(ii)] equality $|\nabla \varrho(x)|=1$ holds for some $x \in M \backslash \{o\}$ if and only if $\varrho = r$ and $M$ is the radially symmetric model $M_h$.
\end{itemize}
\end{theorem}

The estimate is inspired by \cite[Thm. 3.1]{colding}, which deals with the case $p=2$ and $\Ricc \ge 0$, see also \cite{coldingmini_cv} for improvements. The proof in \cite{colding} exploits the linearity of $\Delta$ and the properness of $\varrho$, so it is not extendable to our setting. Nevertheless, as in \cite{colding}, our theorem relies on a new (and, somehow, surprising) Bochner formula for some singular operator associated to $\varrho$, see Proposition \ref{prop_miracolo}.  The key maximum principle used to conclude that $|\nabla \varrho| \le 1$ is Lemma \ref{lem_key} below. It is related to \cite[Chapter 4]{prsmemoirs}, and it is quite versatile: for instance, it directly implies a sharp estimate for $|\nabla \log u|$ when $u$ is a $p$-harmonic function defined on an open set $\Omega$ and having a controlled growth, see Theorem \ref{teo_bellagradiente} below. This complements a result in \cite{sungwang}.\par
%
\subsection*{Main existence results}
Suppose that $\Delta_p$ is non-parabolic on $M$ for every $p$ sufficiently close to $1$, and let $\varrho_p$ be the fake distance associated to $\Delta_p$ with pole at $o$. In Section \ref{sec_convergence}, we study a sequential limit $\varrho_1 = \lim_{p\ra 1} \varrho_p$, searching for conditions that guarantee the positivity and properness of $\varrho_1$. In view of \eqref{def_fake_intro}, this is achieved by finding decay estimates and Harnack inequalities for the kernel $\gr_p$ of $\Delta_p$ that are well behaved as $p \ra 1$. The problem is addressed in Section \ref{sec_proper}, whose main results are sharp  estimates for $\gr_p$ (Theorems \ref{teo_sobolev} and \ref{theorem-decayGreen2}), suited to guarantee the properness of $\varrho_1$, and a sharp Harnack inequality (Theorem \ref{teo_harnack}), robust enough to ensure that $\varrho_1>0$ on $M \backslash \{o\}$. These results are of independent interest.\par

We now describe our main existence theorems for the IMCF. The first one, Theorem \ref{teo_main_L1sobolev}, holds on  manifolds supporting an isoperimetric inequality, a family that encompasses the relevant Examples \ref{ex_minimal} to \ref{ex_roughisometric} below.

\begin{theorem}\label{teo_main_L1sobolev_intro}
Let $M^m$ be connected, complete, non-compact and satisfying \eqref{Ricci_lower_intro}, \eqref{ipo_h_intro} together with the $L^1$ Sobolev inequality
\begin{equation}\label{sob_intro}
\left( \int |\psi|^{ \frac{m}{m-1}}\right)^{\frac{m-1}{m}} \le \So_{1} \int |\nabla \psi| \qquad \forall \, \psi \in \lip_c(M).
\end{equation}
Then, the function $\varrho_1$ is positive, $1$-Lipschitz and proper on $M$, and for each $x \in M \backslash \{o\}$
$$
v_h^{-1} \left( \frac{r(x)^{m-1}}{\So_1^m 2^{m^2-1}}\right) \le \varrho_1(x) \le r(x).
$$
Moreover, $w(x) = (m-1)\log h(\varrho_1)$ is a weak solution to  the IMCF on $M \backslash \{o\}$ satisfying the mean curvature estimate
\begin{equation}\label{esti_nablau_intro}
|\nabla w| \le (m-1)e^{-\frac{w}{m-1}} h' \left( h^{-1} \left( e^{\frac{w}{m-1}} \right) \right) \qquad \text{a.e. on } \, M \backslash \{o\}.
\end{equation}
\end{theorem}

In fact, \eqref{esti_nablau_intro} is equivalent to $|\nabla \varrho_1|\le 1$, and if $\Ricc \ge -(m-1) \kappa^2$ for some $\kappa \in \R^+_0$ it takes the simple form
\begin{equation}\label{esti_H_con}
|\nabla w| \le (m-1)e^{-\frac{w}{m-1}}\sqrt{ \kappa^2 e^{\frac{2w}{m-1}} + 1} \qquad \text{a.e. on } \, M\backslash\{o\}.
\end{equation}
The bound is sharp, and attained with equality by the flow of concentric spheres starting from a point in Euclidean and Hyperbolic space. If the IMCF is smooth, one can deduce \eqref{esti_H_con}, but only for a  nonsingular initial condition, as a consequence of the parabolic maximum principle applied to the equation for $\Delta \mathcal{H}$. However, we found no such estimates available for weak solutions. Theorem \ref{teo_main_L1sobolev_intro} can also be applied to manifolds with asymptotically nonnegative Ricci curvature, namely, those satisfying \eqref{Ricci_lower_intro} and \eqref{ipo_h_intro} with
$$
\int_0^\infty t H(t) \di t < \infty.
$$
In this case, under the validity of \eqref{sob_intro} the fake distance $\varrho_1$ is of the order of $r$, see Remark \ref{rem_asiricci}.\par

Our second result focuses on manifolds with non-negative Ricci curvature and, more generally, manifolds supporting global doubling and weak $(1,1)$-Poincar\'e inequalities, see Theorem \ref{teo_main_riccimagzero} below.
\begin{theorem}\label{teo_main_riccimagzero_intro}
Let $M^m$ be a connected, complete non-compact manifold with $\Ricc \ge 0$. Assume that there exist $C_{\RD}$ and $b \in (1,m]$ such that
\begin{equation}\label{eq_reversevol_222_intro}
\forall t \ge s >0, \qquad \frac{|B_t|}{|B_s|} \ge C_{\RD} \left( \frac{t}{s}\right)^b,
\end{equation}
where balls are centered at a fixed origin $o$. Then the fake distance $\varrho_1$ is positive and proper on $M \backslash \{o\}$. Moreover, there exist constants $C, \bar C$ depending on $C_{\RD}, b,m$, with $\bar C$ also depending on a lower bound for the volume $|B_1|$, such that
$$
\left\{ \begin{array}{ll}
\disp C \left[\inf_{t \in (1, r(x))} \frac{|B_t|}{t^m} \right]^{\frac{1}{m-1}} r(x) \le \varrho_1(x) \le r(x) & \quad \text{on } \, M \backslash B_1, \\[0.5cm]
\disp \bar C r(x) \le \varrho_1(x) \le r(x) & \quad \text{on } \, B_1,\\[0.3cm]
|\nabla \varrho_1| \le 1 & \quad \text{a.e. on } \, M.
\end{array}\right.
$$
Furthermore, $w = (m-1)\log \varrho_1$ is a solution to  the IMCF issuing from $o$ and satisfying
$$
|\nabla w| \le (m-1)e^{-\frac{w}{m-1}} \qquad \text{a.e. on } \, M \backslash \{o\}.
$$
\end{theorem}

\begin{remark}
\emph{The estimates in Theorem \ref{teo_main_riccimagzero_intro} pass to limits with respect to pointed Gromov-Hausdorff convergence $(M_k^m, \metric_k, o_k) \ra (X, \di_X, o)$ whenever \eqref{eq_reversevol_222_intro} holds uniformly in $k$ and the sequence satisfies the noncollapsing condition $|B_1(o_k) | \ge \upsilon$ for each $k$, for some constant $\upsilon >0$.
}
\end{remark}

\begin{remark}
\emph{Condition \eqref{eq_reversevol_222_intro} is implied by the inequality
$$
\tilde C^{-1} t^b \le |B_t| \le \tilde C t^b \qquad \forall \, t \ge 1,
$$
for some constant $\tilde C > 1$. Indeed, $C_{\RD}$ turns out to depend only on $\tilde C, m$ and on a lower bound on Ricci on $B_6$, see Remark \ref{rem_reversevol} below.
}
\end{remark}

Our last result considers the IMCF starting from a relatively compact domain.

\begin{theorem}\label{teo_main_relcompact}
Let the assumptions of either Theorem \ref{teo_main_L1sobolev_intro} or Theorem \ref{teo_main_riccimagzero_intro} be satisfied. Fix $\Omega \Subset M$ with $C^2$ boundary and containing the origin $o$, and define the fake inner and outer radii
$$
R_i = \sup \Big\{ t : \{ \varrho_1 < t\} \subset \Omega \Big\}, \qquad R_o = \inf \Big\{ t : \Omega \subset \{\varrho_1 < t\} \Big\},
$$
with $\varrho_1= \lim_{p \ra 1} \varrho_p$ the fake distance issuing from $o$. Then, there exists a unique, proper solution $w : M \ra \R$ to the IMCF starting from $\Omega$, satisfying
$$
\left\{ \begin{array}{ll}
(i) & \disp \log v_h\big(\varrho_1(x)\big) - \log v_h\big( R_o\big) \le w(x) \le \log v_h\big(\varrho_1(x)\big) - \log v_h\big( R_i\big) \\[0.4cm]
(ii) & \disp |\nabla w| \le \max \left\{ (m-1)\sqrt{H(R_i)}, \ \max_{\partial \Omega} \mathcal{H}_+ \right\},
\end{array}\right.
$$
with $\mathcal{H}_+(x) = \max\{ \mathcal{H}(x),0\}$ the positive part of the mean curvature of $\partial \Omega$ in the inward direction.
\end{theorem}
We stress that the bounds in $(i)$ and $(ii)$ above are explicit, because Theorems \ref{teo_main_L1sobolev_intro} or \ref{teo_main_riccimagzero_intro} allow to effectively estimate $\varrho_1$, hence $R_o$ and $R_i$, in terms of the distance from $o$. The inequality in $(ii)$ should be compared to \eqref{HI_gradesti}, and in fact it could be strengthened to include a decay of $|\nabla w|$ in terms of $\varrho_1$. The interested reader is referred to Remark \ref{rem_concluding} below, where computations for a quadratically decaying lower bound on the Ricci tensor are worked out in detail.\par
In the final Section \ref{sec_isoperimetry}, we study some basic properties of the foliation by $\{\varrho_1 < t\}$. In particular, we prove that the isoperimetric profile of $\{\varrho_1 < t\}$ is, as one might expect, below that of geodesic balls centered at the origin in the model $M_h$, see Theorem \ref{teo_inzero_rho2}.\\
\par
The present paper is meant to be the first step of a broader project. The original motivation for this work was our desire to understand possible links between the recent monotonicity formulas found by  Colding and Colding-Minicozzi in \cite{colding, coldingmini_cv}, for manifolds with non-negative Ricci curvature, and the monotonicity of Hawking-type masses in General Relativity. It is tempting to ask whether one could, somehow, ``bridge" the two via the use of the $p$-Laplace equation, and see whether the new formulas could provide further insight into the geometry of manifolds with Ricci lower bounds or of spacelike slices in General Relativity. In this respect,
foliations by level sets of solutions to  $p$-Laplace equations {had  already been considered} by J. Jeziersky and J. Kijovski to prove special cases of the Riemannian Penrose inequality on asymptotically flat spaces, see \cite{jezierski, jeziekijo, jeziekijo2} and the works of P. Chru\'sciel \cite{chrusciel, chrusciel2}. Furthermore, quite recently, monotonicity formulas similar to those in \cite{colding, coldingmini_cv}, obtained with a different approach, {have been used to find new geometric inequalities} on Euclidean space (\cite{agofogamazzie_2,fogamazziepina}, using the $p$-Laplacian), on static manifolds \cite{agomazzie, borghinimazzieri}, and on manifolds with $\Ricc \ge 0$ (see \cite{agofogamazzie}).

\vspace{0.4cm}

\noindent \textbf{Acknowledgements.} The authors thank Felix Schulze, Virginia Agostiniani, Mattia Fogagnolo and Lorenzo Mazzieri for various discussions about the IMCF and the results in \cite{agomazzie, fogamazziepina, agofogamazzie}. They also express their gratitude to Tobias Colding, Shouhei Honda, Frank Morgan, Andrea Pinamonti, Yehuda Pinchover, Pekka Koskela and Ilkka Holopainen for useful comments. The first author thanks the Math Departments of the Universidade Federal do Cear\'a and of the Scuola Normale Superiore, where most of this work was carried out, for the wonderful environment.\\[0.2cm]
\noindent \textbf{Financial support: } SNS\_RB\_MARI and SNS17\_B\_MARI of the SNS, PRONEX 2015 ``N\'ucleo de An\'alise Geom\'etrica e Aplicac\~oes" (Proc. No. PR2-0054-00009.01.00/11), PRIN 2015 2015KB9WPT\_001, GNAMPA 2017 ``Equazioni differenziali non lineari" (L. Mari); PRIN 2015 ``Real and Complex Manifolds:
Geometry, Topology and Harmonic Analysis'' (M. Rigoli, A.G. Setti). GNAMPA group ``Equazioni differenziali e
sistemi dinamici'' (A.G. Setti).


\section{Preliminaries: capacitors and the Green kernel}\label{sec_preliminaries}

Let $(M^m, \metric)$ be complete, fix an origin $o \in M$ and let $r$ be the distance from $o$. Hereafter, a geodesic ball $B_r$ will always be considered to be centered at $o$, unless otherwise specified. Let $p \in (1,\infty)$, and consider the $p$-Laplace operator $\Delta_p$ on an open set $\Omega$, possibly the entire $M$.


It is convenient to briefly recall some terminology and basic results (we refer the reader to \cite{HKM,holopainen, holopainen3, troyanov, troyanov2, pigolasettitroyanov} for a thorough discussion). Given a pair of open sets $K \Subset \Omega$, the $p$-capacity of the capacitor $(K, \Omega)$ is by definition
$$
\capac_p(K,\Omega) = \inf \left\{ \int_\Omega |\nabla \psi|^p \ \ : \ \ \psi \in \lip_c(\Omega), \ \psi\ge 1 \ \text{ on } \, K\right\}.
$$
If $K$ and $\Omega$ have smooth enough boundary (locally Lipschitz suffices) and are relatively compact, the infimum coincides with the energy $\|\nabla u\|^p_p$, where $u$ is the unique solution to 
$$
\left\{ \begin{array}{l}
\Delta_pu =0 \qquad \text{on  } \, \Omega \backslash K, \\[0.2cm]
u=0 \quad \text{on } \, \partial \Omega, \qquad u=1 \quad \text{on } \, \partial K,
\end{array}\right.
$$
extended with $u\equiv 1$ on $K$, called the $p$-capacity potential of $(K, \Omega)$. If $\Omega$ has non-compact closure, or if it has irregular boundary, by exhausting $\Omega$ with a family of smooth open sets $\Omega_j$ satisfying
\begin{equation}\label{def_exhaustion}
\disp K \Subset \Omega_j \Subset \Omega_{j+1} \Subset \Omega \quad \text{for each } \, j \ge 1, \qquad \bigcup_{j=1}^{\infty} \Omega_j = \Omega,
\end{equation}
the sequence $\{u_j\}$ of the $p$-capacity potentials of $(K, \Omega_j)$ converges to a limit $u : \Omega \ra (0,1]$ which is independent of the chosen exhaustion, is equal to $1$ on $K$,  satisfies $\Delta_p v=0$ on $\Omega \backslash K$, and is still called the $p$-capacity potential of $(K, \Omega)$. Furthermore, $\capac_p(K, \Omega) = \|\nabla u\|_p^p$  (cf. \cite{pigolasettitroyanov}). We say that $\Delta_p$ is non-parabolic on $\Omega$ if $\capac_p(K, \Omega)>0$ for some (equivalently, every) $K \Subset \Omega$, that is, the $p$-capacity potential $u$ of $(K, \Omega)$ is not identically $1$. From \cite{holopainen,holopainen3,troyanov,troyanov2}, this is equivalent to the existence, for each fixed $o \in \Omega$, of a positive Green kernel $\gr$ with pole at $o$, namely, of a positive distributional solution to  $\Delta_p \gr = -\delta_o$ on $\Omega$:
\begin{equation}\label{weakdef}
\int_\Omega |\nabla \gr|^{p-2} \langle \nabla \gr, \nabla \psi \rangle = \psi(o) \qquad \forall \,  \psi \in \lip_c(\Omega).
\end{equation}
A kernel $\gr$ was constructed in \cite{holopainen, holopainen3} starting with an increasing exhaustion $\{\Omega_j\}$ of smooth domains of $\Omega$ and related Green kernels $\gr_{j}$ with pole at $o$ and Dirichlet boundary conditions on $\partial \Omega_j$. We call such a $\gr$ a Dirichlet Green kernel. The existence of each $\gr_{j}$ was shown in \cite[Thm. 3.19]{holopainen} for $p \in (1,m]$, and in \cite{holopainen3} for $p>m$.
The convergence of $\gr_j$ to a finite limit and its equivalence to the non-parabolicity of $\Delta_p$ on $\Omega$ can be found in \cite[Thm. 3.27]{holopainen}. We shall prove that a comparison theorem holds for Green kernels (cf. Corollary \ref{teo_confronto_nuclei} below), and therefore that the Dirichlet Green kernel of each open set $\Omega$, constructed by exhaustion as above, is unique and minimal among positive solutions to \eqref{weakdef}.

\begin{remark}
\emph{Interestingly, the construction in \cite[Thm. 3.25]{holopainen} in fact produces an increasing sequence $\{\gr_j\}$ even without appealing to a comparison result for Green kernels. 
}
\end{remark}


Assume that $\partial \Omega$ is locally Lipschitz, let $\psi \in C^1(\Omega)$, $\ell > 0$ and consider \eqref{weakdef} with test function $\psi \eta(\gr)$, where
$$
\eta \equiv 0 \ \text{on } \, [0,\ell-\eps], \quad \eta \equiv 1 \ \text{on } \, [\ell, \infty), \quad \eta(s) = \eps^{-1}(s-\ell+\eps) \ \text{on } \, (\ell-\eps, \ell).
$$
The regularity of $\partial \Omega$ guarantees that $\gr = 0$ there continuously, thus $\psi \eta(\gr) \in \lip_c(\Omega)$ is admissible as a test function. Letting $\eps \ra 0$ and using the coarea's formula we deduce that
\begin{equation}\label{ide_Gr}
\begin{array}{lcl}
\psi(o) & = & \disp \int_{\{\gr > \ell\}} |\nabla \gr|^{p-2} \langle \nabla \gr, \nabla \psi \rangle + \int_{\{\gr = \ell\}} |\nabla \gr|^{p-1} \psi
\end{array}
\end{equation}
holds for almost every $\ell \in \R^+$. Similarly, the identity
\begin{equation}\label{ide_Gr_2}
\begin{array}{lcl}
0 & = & \disp \int_{\{\gr < \ell\}} |\nabla \gr|^{p-2} \langle \nabla \gr, \nabla \psi \rangle - \int_{\{\gr = \ell\}} |\nabla \gr|^{p-1} \psi
\end{array}
\end{equation}
holds for every $\psi \in C^1(\Omega \backslash \{o\})$ and a.e. $\ell$.

%

Hereafter, we set
\begin{equation}\label{def_mup}
\disp \mu(r) = \left\{ \begin{array}{ll}
\disp \omega_{m-1}^{- \frac{1}{p-1}}\left(\frac{p-1}{m-p}\right) r^{- \frac{m-p}{p-1}} & \qquad \text{if } \, p < m \\[0.5cm]
\disp \omega_{m-1}^{- \frac{1}{m-1}}(-\log r) & \qquad \text{if } \, p=m.
\end{array}\right.
\end{equation}
When $p<m$, note that $\mu(|x|)$ is the Dirichlet Green kernel of $\Delta_p$ on $\R^m$.

\subsection*{Basic comparison theory for Green kernels}

Let $o \in M$ be a fixed point, let $r$ be the distance from $o$ and denote by $\Do$ the maximal domain of normal coordinates centered at $o$. We define the radial sectional curvature of $M$ as the function
	\[
	\begin{array}{l}
	\disp \Sect_\rad \ : \ \mathscr{D}_o \backslash\{o\} \ra \R, \\[0.2cm]
	\Sect_\rad(x) = \max \Big\{ \Sect(X \wedge \nabla r) \ : \ X \in \nabla r(x)^\perp, \ |X|=1 \Big\}.
	\end{array}
	\]
For $R_\infty \in (0, \infty]$, let
$$
h \in C^2([0,R_\infty)), \qquad h>0 \ \text{ on } \, (0,R_\infty), \qquad h(0)=0, \qquad h'(0)=1.
$$
The model manifold $M_h$ is, by definition, $B_{R_\infty}(0)\subset \R^m$ endowed with the metric which in polar coordinates $(t,\theta) \in \R^+ \times \Sph^{m-1}$ centered at the origin is given by
$$
\metric_h = \di t^2 + h(t)^2\di \theta^2,
$$
where $\di\theta^2$  is the round metric on the unit sphere. The radial sectional curvature of $M_h$ is given by $H(t) \doteq - h''(t)/h(t)$. Alternatively, a model can be equivalently described by specifying $H \in C(\R^+_0)$, recovering $h$ as the unique solution to 
\begin{equation}\label{eq_h_uguale}
\left\{ \begin{array}{lcl}
h'' - Hh = 0 \qquad \text{on } \, \R^+ \\[0.2cm]
h(0)=0, \quad h'(0)=1,
\end{array}\right.
\end{equation}
and letting
$$
R_\infty = \sup\{ t \ : h>0 \ \text{on } (0,t) \} \le \infty.
$$
The model is (metrically) complete if and only if $R_\infty = \infty$. Given $M_h$, we denote by
$$
v_h(t) = \omega_{m-1}h(t)^{m-1}, \quad V_h(t) = \int_0^t v_h(s)\di s \qquad \text{ for }t \in [0,R_\infty)
$$
the volume of the geodesic spheres $\partial B_t^{h}$ and balls $B_t^{h}$ of radius $t$ centered at the origin, respectively. In the particular case where $H$ is a constant, the corresponding solution to  \eqref{eq_h_uguale} is
$$
h(s) = \left\{ \begin{array}{lll}
\disp \frac{\sin(\kappa t)}{\kappa} & \quad \text{if } \, H = -\kappa^2 < 0, & \quad \text{with } \, R_\infty = \pi/\kappa \\[0.3cm]
\disp t & \quad \text{if } \, H=0, & \quad \text{with } \, R_\infty = \infty\\[0.3cm]
\disp \frac{\sinh(\kappa t)}{\kappa} & \quad \text{if } \, H = \kappa^2 > 0, & \quad \text{with } \, R_\infty = \infty.
\end{array}\right.
$$
We shall always be concerned with models with $H \ge 0$, and, if $H=\kappa^2$, with a slight abuse of notation, we simply denote the volume of geodesic spheres, respectively balls, by $v_\kappa(t)$ and $V_\kappa(t)$.

\begin{definition}\label{def_modelabovebelow}
Let $M$ be a complete Riemannian manifold of dimension $m \ge 2$, fix $o \in M$ and let $B_R^*(o)=B_R(o)\backslash\{o\}$.
\begin{itemize}
\item[-] We say that $B_R^{h} \subset M_h$ is a model from below for $M$ if
\begin{equation}\label{assu_ricciperdefvarrho}
\Ricc(\nabla r, \nabla r) \ge -(m-1)H(r) \qquad \text{on } \, \Do \cap B^*_R(o),
\end{equation}
\item[-] We say that $B_R^{h} \subset M_h$ is a model from above for $M$ if
\begin{equation}\label{assu_sectioperdefvarrho}
B_R^*(o) \subset \Do, \qquad \Sect_\rad \le -H(r) \qquad \text{on } \, B_R^*(o).
\end{equation}
\end{itemize}
\end{definition}
The Dirichlet Green kernel of $-\Delta_p$ on $B^h_R \subset M_h$ with singularity at the origin is
\begin{equation}\label{def_greenmodel_R}
\sgrh_{R}(t) = \int_{t}^{R} v_h(s)^{-\frac{1}{p-1}} \di s
\end{equation}
and, if $R = R_\infty = \infty$, we simply write
\begin{equation}\label{def_greenmodel}
\sgrh(t) = \int_{t}^{\infty} v_h(s)^{-\frac{1}{p-1}}\di s.
\end{equation}
The finiteness of $\sgrh(t)$ is equivalent to the non-parabolicity of $\Delta_p$ on $M_h$. Note also that $\sgrh(0^+)$ is finite if and only if $p>m$. As a consequence of the comparison theory for the distance function (cf. \cite[Chapter 2]{prs}), one obtains the following result which will be repeatedly used:

\begin{proposition}\label{prop_compagreenmodel}
Let $M$ be a complete manifold, $o \in M$ and $r(x) = \dist(x,o)$.
\begin{itemize}
\item[$(i)$] If $B_R^h \subset M_h$ is a model from below for $M$, then the transplanted function
\begin{equation}\label{def_grh_transplanted}
\grh_R(x) = \sgrh_R\big(r(x)\big)
\end{equation}
satisfies $\Delta_p \grh_R \ge -\delta_o$ on $B_R(o)$.
\item[$(ii)$] If $B_R^h \subset M_h$ is a model from above for $M$, then $\grh_R$ in \eqref{def_grh_transplanted} satisfies $\Delta_p \grh_R \le -\delta_o$ on $B_R(o)$.
\end{itemize}
\end{proposition}

To compare a Green kernel $\gr$ with $\sgrh_R(r)$, we need a precise description of the behaviour of $\gr$ near its singularity. For $p \le m$, in our needed generality J. Serrin in \cite[Thm 12]{Serrin_1} showed that
\begin{equation}\label{boundserrin}
\gr(x,o) \asymp \mu\big(r(x)\big)
\end{equation}
as $r(x)= \mathrm{dist}(x,o) \ra 0$, with $\mu$ as in \eqref{def_mup}. However, for our purposes we need to know both the asymptotic behaviour of $\gr$ and $\nabla \gr$ near the singularity. In Euclidean setting, the problem was considered by S. Kichenassamy and L. Veron in \cite{kichenveron} and in \cite[pp. 243-251]{veron}, and their technique, based on a blow-up procedure, can be adapted to manifolds. A complete proof of the following result can be found in \cite{mrs}.


\begin{theorem}\label{teo_localsingular}
For $p \le m$, let $\gr$ be a Green kernel for $\Delta_p$ on an open set $\Omega \subset M^m$ containing $o$. Then, $\gr$ is smooth in a punctured neighbourhood of $o$ and, as $x \ra o$,
\begin{equation}
\begin{array}{ll}
(1) & \quad \gr \sim \mu(r), \\[0.2cm]
(2) & \quad | \nabla \gr - \mu'(r) \nabla r | = o\big( \mu'(r)\big),\\[0.3cm]
(3) & \quad \text{if } \, p < m, \quad \big| \nabla^2 \gr - \mu''(r) \di r \otimes \di r - \frac{\mu'(r)}{r} \big( \metric - \di r \otimes \di r\big) \big| = o\big( \mu''(r)\big).
\end{array}
\end{equation}
\end{theorem}

\begin{remark}
\emph{The above theorem does not contain the full strength of Kichenassamy-Veron's result, in particular we do not claim that $\gr - \mu(r) \in L^\infty$ near the origin. Indeed, even for the kernel $\sgrh$ of a model with curvature $H(0) \neq 0$, a direct computation shows that $\sgrh - \mu \not\in L^\infty$ when $3p \le m+2$.
}
\end{remark}

As a direct consequence, we obtain the next comparison theorem which can be found in \cite[Thm. 2.1]{kichenveron} when $M = \R^m$.

\begin{corollary}\label{teo_confronto_nuclei}
Fix $p \in (1,\infty)$, let $\Omega_1 \subset \Omega_2 \Subset M$ be open domains containing $o$, and let $\gr_j$ be a Dirichlet Green kernel for $\Delta_p$ on $\Omega_j$, $j \in \{1,2\}$. Then, $\gr_1 \le \gr_2$. In particular, the Dirichlet Green kernel of an open set, if it exists, is unique.
\end{corollary}

\begin{proof}
We prove that the Dirichlet Green kernel $\gr'$ of a smooth domain $\Omega' \Subset \Omega_1$ satisfies $\gr' \le \gr_2$, and the thesis follows by letting $\gr' \uparrow \gr_1$. If $p > m$, it is enough to apply standard comparison (cf. \cite[Thm. 3.4.1]{pucciserrin}), since it is known that both $\gr', \gr_2 \in W^{1,p}(\Omega')$ (cf. \cite{mrs} for a quick proof). If $p \le m$, for every $\eps>0$ we compare $\gr'$ and $(1+\eps) \gr_2$ on the set $\Omega'_\eps = \Omega' \cap \{ \gr' > (1+\eps) \gr_2\}$, which we assume to be non-empty. Theorem \ref{teo_localsingular} guarantees that $o \not \in \overline{\Omega'_\eps}$, so $\gr', \gr_2 \in W^{1,p}(\Omega'_\eps)$ and $\gr' \le (1+\eps)\gr_2$ by comparison, contradicting the definition of $\Omega_\eps'$. Hence, $\Omega'_\eps = \emptyset$, and we let $\eps \ra 0$ to conclude.
\end{proof}

\begin{remark}
\emph{The result admits a generalization to $p$-Laplace operators with a potential: in Euclidean setting, see \cite[Thm. 5.4]{pinchovertintarev} for $p \le m$ and \cite[Cor. 1.1]{pinchofraas} for $p>m$.
}
\end{remark}


With the same technique, we can also compare $\gr$ to kernels of models from above and below, improving on \cite{kura}.

\begin{corollary}[\textbf{Comparison}]\label{teo_confronto_conmodel}
Let $(M, \metric)$ be a complete manifold, fix $p \in (1,\infty)$ and let $\gr$ be the Green kernel for $\Delta_p$ on an open domain $\Omega$ containing $o$.
\begin{itemize}
\item[$(i)$] Suppose that $B_R^{h} \subset M_h$ is a model from below for $M$ and $B_R(o) \subset \Omega$. Then,
$$
\gr(x) \ge \disp \sgrh_R\big(r(x)\big) \qquad \forall \, x \in B_R(o).
$$
\item[$(ii)$] Suppose that $B_R^{h} \subset M_h$ is a model from above for $M$ and $B_R(o) \subset \Omega$. Then,
$$
\gr(x) \le \disp \sgrh_R\big(r(x)\big) + \| \gr\|_{L^\infty(\partial B_R(o))} \qquad \forall \, x \in B_R(o).
$$
\end{itemize}
\end{corollary}

We conclude with

\begin{proposition}\label{prop_parabnonparab}
Let $(M, \metric)$ be a complete, non-compact manifold, and fix $p \in (1,\infty)$.
\begin{itemize}
\item[$(i)$] If $\Delta_p$ is non-parabolic on $M$, then every model from below $M_h$ satisfies $R_\infty = \infty$ and $\Delta_p$ is non-parabolic on $M_h$, namely,
	\begin{equation}\label{mod_nonp}
v_h(s)^{-\frac{1}{p-1}} \in L^1(\infty).
	\end{equation}
\item[(ii)] If $M$ admits a model from above $M_h$ satisfying $R_\infty = \infty$ and whose $p$-Laplacian $\Delta_p$ is non-parabolic, then $\Delta_p$ is non-parabolic on $M$.
\end{itemize}
\end{proposition}

\begin{proof}
$(i)$. If $R_\infty < \infty$, the Laplacian comparison theorem (\cite[Thm. 2.4]{prs}) would imply that $M$ is compact with $\mathscr{D}_o \subset B_{R_\infty}(o)$, a contradiction. By $(i)$ in Corollary \ref{teo_confronto_conmodel}, for each $R$ it holds $\sgrh_R(r(x)) \le \gr(x)$, and letting $R \ra \infty$ we get \eqref{mod_nonp}. Item $(ii)$ similarly follows by applying $(ii)$ in Corollary \ref{teo_confronto_conmodel} to the kernel of $B_R(o)$ and letting $R \to \infty$.
\end{proof}

\begin{convention}
Hereafter, we will say for short that $\gr$ is the Green kernel of $\Delta_p$ on $\Omega$ if it is the Dirichlet Green kernel (the indication of the pole is omitted when no confusion arises).
\end{convention}

\subsection{The fake distance}\label{subsec_fake}

We shall assume the following:

\begin{itemize}
\item[$(\Hp)$] $(M^m, \metric)$ is complete, non-compact and, fixed $o \in M$ and writing $r(x)=\dist(x,o)$,
\begin{equation}\label{assu_ricciperconfronto}
\Ricc \ge -(m-1) H(r) \qquad \text{on } \, M,
\end{equation}
for some $0 \le H \in C(\R^+_0)$. Moreover, $\Delta_p$ is non-parabolic on $M$.
\end{itemize}

If $M$ satisfies $(\Hp)$, then by Proposition \ref{prop_parabnonparab} the solution $h \in C^2(\R^+_0)$ of \eqref{eq_h_uguale} is positive on $\R^+$ and $\Delta_p$ is non-parabolic on $M_h$, that is,
\begin{equation}\label{mod_nonparab}
v_h^{- \frac{1}{p-1}}\in L^1(\infty).
\end{equation}
Also, note that $h$, and therefore 
$v_h$, are monotone increasing and diverging as $t \ra \infty$. In what follows, we shall be interested in non-increasing $H$, in which case the following two useful properties hold. 

\begin{lemma}\label{lem_ODE}
Let $0 \le H \in C(\R^+_0)$ be non-increasing. Then, $v_h'/v_h$ and $|(\log \sgrh)'|$ have negative derivatives on $\R^+$ (the latter, for each $p>1$).
\end{lemma}
\begin{proof}
The behaviour of $v_h$ at zero guarantees that $\{t : (v_h'/v_h)'(t) < 0 \} \neq \emptyset$. Assume by contradiction that $(v_h'/v_h)'(t_0)=0$ for some $t_0 \in \R^+$. Then, the Riccati equation
\begin{equation}\label{eq_riccati}
\left( \frac{h'}{h}\right)' + \left( \frac{h'}{h}\right)^2 = H
\end{equation}
implies the equality $(v_h'/v_h)(t_0) = (m-1)\kappa$, where we have set $\kappa \doteq \sqrt{H(t_0)}$. Because $H$ is non-increasing, by Sturm comparison on $(0,t_0]$ with the model of curvature $-\kappa^2$ and volume $v_\kappa$ we deduce that $v_h'/v_h \ge v_\kappa'/v_\kappa$ on $(0,t_0]$. However, 
	\[
	\forall \, t \in \R^+, \ \ \frac{v_\kappa'(t)}{v_\kappa(t)} = \left\{ \begin{array}{ll}
	(m-1)\kappa \coth(\kappa t) > (m-1)\kappa & \quad \text{if } \, \kappa >0, \\[0.2cm]
	(m-1)t^{-1} > 0 & \quad \text{if } \, \kappa = 0,
	\end{array}\right.
	\]
contradiction. To show the second part of the statement, set for convenience $\chi(t) = |(\log \sgrh)'(t)|$, and note that $\chi > 0$ on $\R^+$. Differentiating,
	\begin{equation}\label{eq_chi}
	\chi' = \chi \left[ \chi - \frac{1}{p-1} \frac{v_h'}{v_h} \right].	
	\end{equation}
Suppose that $\chi'(t_0) \ge 0$ for some $t_0 \in \R^+$; since $v_h'/v_h$ has negative derivative on $\R^+$,  inspection of the ODE shows that $\chi'>0$ on $(t_0,\infty)$. Therefore, having fixed $t_1>t_0$ there exists $\eps > 0$ such that the term in brackets in \eqref{eq_chi} is greater than $\eps \chi$ on $[t_1,\infty)$. By comparison, $\chi$ lies above the solution $\bar \chi$ to $\bar \chi' = \eps \bar\chi^2$ on $[t_1, \infty)$. However, $\bar \chi$ explodes in finite time, contradiction.
\end{proof}

We are ready to define the fake distance.

\begin{definition}\label{def_fake}

Let $M$ satisfy $(\Hp)$ for some $p \in (1,m]$ and origin $o \in M$, and let $\gr$ be the Green kernel with pole at $o$. The \emph{fake distance} $\varrho : M\backslash\{o\} \ra \R^+_0$ is implicitly defined as $\gr(x) = \sgrh\big( \varrho(x)\big)$, that is,
\begin{equation}\label{def_bxy}
\gr(x) = \int_{\varrho(x)}^{\infty} v_h(s)^{-\frac{1}{p-1}} \di s \qquad\mbox{on } \, M \backslash \{o\}.
\end{equation}
\end{definition}
Observe that, because of \eqref{boundserrin} and $\sgrh(0^+) = +\infty$, $\varrho$ is well defined and positive on $M \backslash \{o\}$, and can be extended by continuity with $\varrho(o)=0$. Furthermore, by \cite{tolksdorff}, $\varrho$ is locally in $C^{1,\alpha}$ on $M \backslash \{o\}$. Corollary \ref{teo_confronto_conmodel} easily implies the following
%
%
\begin{proposition}\label{prop_basiccomp}
Let $M$ satisfy $(\Hp)$ for some $p \in (1,m]$ and origin $o$, and let $\varrho$ be the fake distance associated to the kernel $\gr$ of $\Delta_p$. Then, $\varrho \le r$ on $M$.
\end{proposition}

\begin{proof}
Corollary \ref{teo_confronto_conmodel} and the definition of $\varrho$ imply
$$
\sgrh\big( \varrho(x)\big) = \gr(x) \ge \sgrh\big(r(x)\big) \qquad \text{on } \, M \backslash \{o\},
$$
and the conclusion follows since $\sgrh$ is decreasing.
\end{proof}

Differentiating shows that $\varrho$ satisfies the following identities:

\begin{equation}\label{identi_varrho}
\begin{array}{ll}
\disp \nabla \varrho = -v_h(\varrho)^{\frac{1}{p-1}} \nabla \gr & \qquad \text{on } \, M \backslash \{o\}, \\[0.3cm]
\disp \Delta_p \varrho = \frac{v_h'(\varrho)}{v_h(\varrho)}|\nabla \varrho|^p & \qquad \text{weakly on } \, M\backslash\{o\}
\end{array}
\end{equation}

and, therefore, for each $\psi \in C^2(\R)$ with $\psi' \neq 0$,

\begin{equation}\label{bella!!!}
\disp \Delta_p \big[\psi(\varrho)\big] = \left[v_h^{-1}\left( v_h |\psi'|^{p-2}\psi'\right)'\right](\varrho) |\nabla \varrho|^p.
\end{equation}

As remarked in the Introduction, $v_h^{-1}\left( v_h |\psi'|^{p-2}\psi'\right)'$ is the expression of the $p$-Laplacian of the radial function $\psi$ in the model $M_h$, making it possible to radialize with respect to $\varrho$.

%
%
%
%
%
%

\subsection{Gradient estimates}

\begin{proposition}[\textbf{Near the singularity}]\label{prop_nearminm}
Assume $(\Hp)$ for some $p \in (1, m)$, and define $\varrho$ as in \eqref{def_bxy}. Then, $\varrho$ is smooth in a punctured neighbourhood of $o$ and
\begin{equation}\label{asinto_varrho}
\begin{array}{c}
\disp \varrho(x) \sim r(x), \quad |\nabla \varrho(x)- \nabla r(x)| \ra 0 \qquad \text{as } \, x \ra o, \\[0.3cm]
\disp \varrho \nabla^2 \varrho - \metric + \di r \otimes \di r \ra 0 \qquad \text{as a quadratic form, as } \, x \ra o.
\end{array}
\end{equation}
\end{proposition}

\begin{proof}
By $(2)$ in Theorem \ref{teo_localsingular}, $|\nabla \gr|>0$ in a punctured neighbourhood of $o$, so $\gr$ (hence $\varrho$) is smooth there. Using $(1)$ in Theorem \ref{teo_localsingular} we deduce that $\varrho(x) \sim r(x)$ as $x \ra o$. According to the first identity in  \eqref{identi_varrho},
\begin{equation}\label{gradrho}
\nabla \varrho = -v_h(\varrho)^{\frac{1}{p-1}} \nabla \gr(x).
\end{equation}
By $(2)$ in Theorem \ref{teo_localsingular}
$$
\begin{array}{lcl}
o\big( |\mu'(r)|\big) & = & \disp \big| \nabla \gr - \mu'(r) \nabla r\big| =
\left| v_h(\varrho)^{-\frac{1}{p-1}}\left(\nabla r -\nabla \varrho \right) -
\left(v_h(\varrho)^{-\frac1{p-1}} + \mu'(r)\right) \nabla r \right| \\
&\ge & v_h(\varrho)^{-\frac{1}{p-1}}|\nabla \varrho - \nabla r| -
|\mu'(r)| \left|v_h(\varrho)^{-\frac1{p-1}}\mu'(r)^{-1} + 1\right|,
\end{array}
$$
so, dividing  through by $|\mu'(r)|$ and rearranging we deduce that
$$
 \frac{v_h(\varrho)^{-\frac1{p-1}}}{|\mu'(r)|}|\nabla \varrho - \nabla r|\leq \frac{o(|\mu'(r)|)}{|\mu'(r)|} +
 \left|v_h(\varrho)^{-\frac1{p-1}}\mu'(r)^{-1} + 1\right|
$$
and, using $\mu'(r) \sim - v_h(r)^{-\frac{1}{p-1}} \sim - v_h(\varrho)^{-\frac{1}{p-1}}$,
we conclude that  $|\nabla \varrho - \nabla r| \ra 0$.

To show the Hessian estimates, we differentiate \eqref{gradrho} to deduce
$$
\nabla^2 \varrho = \frac{1}{p-1} \frac{v_h'(\varrho)}{v_h(\varrho)} \di \varrho \otimes \di \varrho - v_h(\varrho)^{\frac{1}{p-1}} \nabla^2 \gr
$$
which, together with $3)$ in Theorem \ref{teo_localsingular}, gives
$$
\begin{array}{l}
\disp \frac{1}{|\mu'(r)|} \left[ \mu''(r) \di r \otimes \di r + \frac{\mu'(r)}{r} \big( \metric -\di r \otimes \di r \big) \right] = \frac{\nabla^2 \gr}{|\mu'(r)|} + o \left( \left|\frac{\mu''(r)}{\mu'(r)}\right| \right) \\[0.4cm]
\qquad = \disp \big(1+ o(1)\big) \left[ \frac{1}{p-1}\frac{v_h'(\varrho)}{v_h(\varrho)} \di \varrho \otimes \di \varrho - \nabla^2 \varrho \right] + o\left( \frac{1}{r} \right).
\end{array}
$$
The third formula in \eqref{asinto_varrho} follows multiplying by $\varrho$ and using that $v_h'(t)/v_h(t) \sim (m-1)/t$ as $t \ra 0$.

\end{proof}

As a consequence of the above proposition, the singularity of $\varrho$ at the origin is mild enough to guarantee that the second identity in \eqref{identi_varrho} holds weakly on the entire $M$:
$$
\Delta_p \varrho = \frac{v_h'(\varrho)}{v_h(\varrho)}|\nabla \varrho|^p \qquad \text{weakly on } \, M.
$$

We next search for global gradient estimates for $\varrho$. For $X \in TM$, $X \neq 0$ define the linearization of the $p$-Laplacian $A(X) : TM \ra TM$ as
$$
A(X) = |X|^{p-2} \left( \mathrm{Id} + (p-2) \left\langle \cdot, \frac{X}{|X|} \right\rangle \frac{X}{|X|}\right).
$$
The eigenvalues of $A(X)$ are $(p-1)|X|^{p-2}$ in the direction of $X$, and $|X|^{p-2}$ in the orthogonal complement. Define also $\metric_B$ as the $(2,0)$-version of $A(X)^{-1/2}$, and note that $\metric_B$ is a metric for each $X \neq 0$. Norms and traces with respect to $\metric_B$ will be denoted with $|\cdot |_B, \mathrm{Tr}_B$. Setting $\nu = X/|X|$ and considering an orthonormal frame $\{e_i, \nu\}$, $1 \le i \le m-1$ for $\metric$ with dual coframe $\{\theta^j, \theta^\nu\}$, for every covariant $2$-tensor $C$ we can write
\begin{equation}\label{ide_basicheimpo}
\begin{array}{lcl}
\metric_B & = & \disp |X|^{-\frac{p-2}{2}} \Big\{ (p-1)^{-1/2} \theta^\nu \otimes \theta^\nu + \sum_j \theta^j \otimes \theta^j \Big\} \\[0.4cm]
\mathrm{Tr}_B C & = & \disp |X|^{\frac{p-2}{2}} \Big\{ \sqrt{p-1}C_{\nu\nu} + \sum_j C_{jj} \Big\} \\[0.4cm]
|C|^2_B & = & \disp \disp |X|^{p-2}\Big\{ (p-1)C^2_{\nu\nu} + p \sum_j C_{\nu j}^2 + \sum_{i,j} C_{ij}^2\Big\}.
\end{array}
\end{equation}

The following Bochner formula is basically a rewriting, in a form more suitable for our application, of \cite[Lem. 2.1]{kotschwarni}, see also \cite[Prop. 7]{nabervalto}. We provide a quick proof for the sake of completeness.

\begin{proposition}\label{bochner_basic}
Let $p \in (1,\infty)$, $U \subset M$ be an open set and let $F \in C^3(U)$ with $|\nabla F| >0$ on $U$. Then,
\begin{equation}\label{bochner}
\begin{array}{l}
\disp \frac{1}{2} \diver \left( A(\nabla F) \nabla |\nabla F|^2 \right) = \\[0.5cm]
\qquad \qquad = \big| \nabla^2 F\big|^2_B + \Ricc(\nabla F, \nabla F)|\nabla F|^{p-2} + \langle \nabla \Delta_p F, \nabla F \rangle
\end{array}
\end{equation}
on $U$, where $B = B(\nabla F)$.
\end{proposition}

\begin{proof}
Let $\{e_i, \nu\}$ be an adapted orthonormal frame with $\nu = \nabla F/|\nabla F|$. We first compute
\begin{equation}\label{www}
\begin{array}{lcl}
\disp \langle \nabla \Delta_p F, \nabla F \rangle & = & \disp (p-2)^2 |\nabla F|^{p-2} \langle \nabla |\nabla F|, \nu \rangle^2 + (p-2)|\nabla F|^{p-2} \langle \nabla \langle \nabla |\nabla F|, \nu \rangle, \nabla F \rangle  \\[0.3cm]
& & \disp + (p-2) \langle \nabla|\nabla F|, \nu \rangle |\nabla F|^{p-2}\Delta F + |\nabla F|^{p-2} \langle \nabla \Delta F, \nabla F \rangle \\[0.3cm]
& = & (p-2) \langle \nabla |\nabla F|, \nu \rangle  \Delta_p F  + (p-2)|\nabla F|^{p-2} \langle \nabla \langle \nabla |\nabla F|, \nu \rangle, \nabla F \rangle \\[0.3cm]
& & \disp + |\nabla F|^{p-2} \langle \nabla \Delta F, \nabla F \rangle.
\end{array}
\end{equation}
On the other hand
$$
\begin{array}{l}
\disp \frac{1}{2} \diver \left( A(\nabla F) \nabla |\nabla F|^2 \right) = \frac{1}{2} \diver \Big( |\nabla F|^{p-2} \nabla |\nabla F|^2 + 2(p-2)|\nabla F|^{p-2} \langle \nabla |\nabla F|, \nu \rangle \nabla F  \Big) \\[0.3cm]
\quad = \disp (p-2)|\nabla F|^{p-2} \big|\nabla |\nabla F|\big|^2 + \frac{1}{2} |\nabla F|^{p-2} \Delta |\nabla F|^2 + (p-2) |\nabla F|^{p-2} \langle \nabla \langle \nabla |\nabla F|, \nu \rangle, \nabla F \rangle \\[0.3cm]
\qquad \disp + (p-2) \langle \nabla |\nabla F|, \nu \rangle \Delta_p F.
\end{array}
$$
Replacing the last two terms by means of \eqref{www}, using the standard Bochner formula for the Laplacian, the identity $\nabla |\nabla F| = F_{\nu j} e_j + F_{\nu\nu} \nu$, and \eqref{ide_basicheimpo} we infer
$$
\begin{array}{l}
\disp \frac{1}{2} \diver \left( A(\nabla F) \nabla |\nabla F|^2 \right) \\[0.3cm]
\qquad = \disp (p-2)|\nabla F|^{p-2} \big|\nabla |\nabla F|\big|^2 + |\nabla F|^{p-2} \Big[ \frac{\Delta |\nabla F|^2}{2} - \langle \nabla \Delta F, \nabla F \rangle \Big] + \langle \nabla \Delta_p F, \nabla F \rangle \\[0.3cm]
\qquad = \disp (p-2)|\nabla F|^{p-2} \big|\nabla |\nabla F|\big|^2 + |\nabla F|^{p-2} \Big[ |\nabla^2 F|^2  + \Ricc(\nabla F, \nabla F) \Big] + \langle \nabla \Delta_p F, \nabla F \rangle \\[0.3cm]
\qquad = \disp |\nabla F|^{p-2} \Big[ \sum_{i,j = 1}^{m-1} F_{ij}^2 + p \sum_{j=1}^{m-1} F_{\nu j}^2 + (p-1)F_{\nu\nu}^2\Big] + |\nabla F|^{p-2}\Ricc(\nabla F, \nabla F) + \langle \nabla \Delta_p F, \nabla F \rangle \\[0.3cm]
\qquad = \disp \big| \nabla^2 F\big|^2_B + |\nabla F|^{p-2}\Ricc(\nabla F, \nabla F) + \langle \nabla \Delta_p F, \nabla F \rangle,
\end{array}
$$
as claimed.
\end{proof}

The next is the main, new Bochner formula.

\begin{proposition}\label{prop_miracolo}
For $p \in (1,\infty)$, let $u$ be a positive solution to  $\Delta_p u =0$ in an open set $\Omega \subset M^m$. Fix a model $M_h$ with radial curvature $-H(r)$ and such that $\Delta_p$ is non-parabolic on $M_h$, and define $\varrho$ according to
$$
u(x) = \sgrh\big(\varrho(x)\big) = \int_{\varrho(x)}^\infty v_h(s)^{-\frac{1}{p-1}} \di s.
$$
If $p>m$, also assume that $u < \sgrh(0)$ on $\Omega$. Set
\begin{equation}\label{def_mu}
\mu = -\frac{mp -3p+2}{p-1}, \qquad F(t) = \int_0^t h(s)^{\frac{1}{\sqrt{p-1}}} \di s.
\end{equation}
Then, on $\big\{ |\nabla \varrho|>0\big\}$ and denoting with $\nu = \nabla \varrho/|\nabla \varrho|$,
\begin{equation}\label{bochner_miracolo}
\begin{array}{l}
\disp \frac 12 \disp h^{-\mu} \diver \left( h^{\mu} A(\nabla \varrho) \nabla |\nabla \varrho|^2 \right) = \\[0.5cm]
\disp \qquad \qquad (F')^{-p} \left| \nabla^2 F - \frac{\tr_B \nabla^2 F}{m}\metric_B \right|^2_B \\[0.5cm]
\qquad \qquad \disp + \frac{1}{m} \left[ \sqrt{p-1} - (p-1) \right]^2|\nabla \varrho|^{p-2} \big[\nabla^2 \varrho(\nu,\nu)^2\big] \\[0.5cm]
\qquad \qquad \disp + |\nabla \varrho|^p\Big[\Ricc(\nu, \nu) + (m-1)H|\nabla \varrho|^2\Big]
\end{array}
\end{equation}
where,  with a slight abuse of notation, $F = F(\varrho)$, $B = B(\nabla F)$ and $H,h$ are evaluated at $\varrho$.
\end{proposition}

\begin{remark}
\emph{It is interesting to compare our formula with the integral identities in \cite{agofogamazzie_2} and in  \cite[Thm. 3.4]{fogamazziepina} in Euclidean setting: the latter follows from a different viewpoint, nevertheless still inspired by the use of "fake distance" type functions. In the linear case, similar identities were obtained in \cite{coldingmini_cv} and in \cite{agofogamazzie, borghinimazzieri}.
}
\end{remark}

\begin{proof}
Let $\{e_i,\nu\}$ be an orthonormal frame with $\nu = \nabla \varrho/|\nabla \varrho|$. Let $\{F_{\nu\nu}, F_{\nu j}, F_{ij}\}$ be the components of $\nabla^2 F$ in the basis $\{e_i,\nu\}$.
From
$$
\nabla F = F' \nabla \varrho, \qquad \nabla^2 F = F'' \di \varrho \otimes \di \varrho + F' \nabla^2 \varrho
$$
We get
$$
F_{\nu\nu} = F''|\nabla \varrho|^2 + F' \varrho_{\nu\nu}, \qquad F_{ij} = F'\varrho_{ij}, \qquad F_{\nu j} = F' \varrho_{\nu j}.
$$
Moreover, the definition of $F$ implies
\begin{equation}\label{eq_tracciaF}
\begin{array}{l}
\disp |\nabla F|^{-\frac{p-2}{2}} \tr_B(\nabla^2F) = \disp \sqrt{p-1}F_{\nu\nu} + \sum_j F_{jj} \\[0.4cm]
\qquad  \disp = \sqrt{p-1}F''|\nabla \varrho|^2 + F'\left[\sqrt{p-1}\varrho_{\nu\nu} + \sum_j \varrho_{jj}\right] \\[0.4cm]
\qquad \disp = \disp F' \left\{ \frac{h'}{h}|\nabla \varrho|^2 + \left[ (p-1)\varrho_{\nu\nu} + \sum_j \varrho_{jj}\right] + \left[\sqrt{p-1}-(p-1)\right] \varrho_{\nu\nu}\right\}.
\end{array}
\end{equation}
Expanding the expression for $\Delta_p \varrho$ and using \eqref{identi_varrho}, we get
$$
(p-1)\varrho_{\nu\nu} + \sum_j \varrho_{jj} = |\nabla \varrho|^{2-p}\Delta_p \varrho = \frac{v_h'}{v_h}|\nabla \varrho|^2 = (m-1)\frac{h'}{h}|\nabla \varrho|^2
$$
and from \eqref{eq_tracciaF} we deduce
\begin{equation}\label{eq_trBnabla2F}
\tr_B(\nabla^2 F) = (F')^{\frac{p}{2}}|\nabla \varrho|^{\frac{p-2}{2}}\left\{ m \frac{h'}{h}|\nabla \varrho|^2 + \left[\sqrt{p-1}-(p-1)\right] \varrho_{\nu\nu}\right\}.
\end{equation}
We next examine $|\nabla^2 F|^2_B$. By Bochner's formula \eqref{bochner_basic},
$$
|\nabla^2F|^2_B = \disp  \disp \frac{1}{2} \diver \left( A(\nabla F) \nabla |\nabla F|^2 \right) - (F')^p \Ricc(\nabla \varrho, \nabla \varrho)|\nabla \varrho|^{p-2} - \langle \nabla \Delta_p F, \nabla F \rangle,
$$
hence using the identities
\begin{equation}\label{ide_use}
\begin{array}{lcl}
\disp A(\nabla F)\nabla|\nabla F|^2 & = & \disp (F')^p A(\nabla \varrho)\nabla |\nabla \varrho|^2 + 2(p-1)(F')^{p-1}F''|\nabla \varrho|^p \nabla \varrho \\[0.5cm]
& = & \disp (F')^p \left\{ A(\nabla \varrho)\nabla |\nabla \varrho|^2 + 2\sqrt{p-1} \frac{h'}{h}|\nabla \varrho|^2 \left( |\nabla \varrho|^{p-2}\nabla \varrho\right) \right\} \\[0.5cm]
\disp \langle A(\nabla \varrho) \nabla |\nabla \varrho|^2, \nabla \varrho \rangle  & = & \disp 2(p-1)|\nabla \varrho|^p \varrho_{\nu\nu}
\end{array}
\end{equation}
we get
\[
\begin{array}{lcl}
|\nabla^2F|^2_B & = & \disp  \disp \frac{1}{2} \diver \left( (F')^p \left\{ A(\nabla \varrho)\nabla |\nabla \varrho|^2 + 2\sqrt{p-1} \frac{h'}{h}|\nabla \varrho|^2 \left( |\nabla \varrho|^{p-2}\nabla \varrho\right) \right\}\right) \\[0.5cm]
& & \disp - (F')^p \Ricc(\nabla \varrho, \nabla \varrho)|\nabla \varrho|^{p-2} - \langle \nabla \Delta_p F, \nabla F \rangle \\[0.5cm]
 & = & \disp  \disp \frac{1}{2}(F')^p \diver \left( A(\nabla \varrho)\nabla |\nabla \varrho|^2\right) \\[0.5cm]
& & \disp + \frac{p}{2}(F')^{p-1}F'' \langle A(\nabla \varrho) \nabla |\nabla \varrho|^2, \nabla \varrho \rangle + \sqrt{p-1} (F')^p \frac{h'}{h} |\nabla \varrho|^2 \Delta_p \varrho \\[0.5cm]
& & \disp + \sqrt{p-1}(F')^p\frac{h'}{h}|\nabla \varrho|^{p-2} \langle \nabla |\nabla \varrho|^2, \nabla \varrho \rangle + \sqrt{p-1}(F')^p \left(\frac{h'}{h}\right)' |\nabla \varrho|^{p+2} \\[0.5cm]
& & \disp + \sqrt{p-1}p(F')^{p-1}F''|\nabla \varrho|^{p+2} \frac{h'}{h} \\[0.5cm]
& & \disp - (F')^p \Ricc(\nu,\nu)|\nabla \varrho|^{p} - \langle \nabla \Delta_p F, \nabla F \rangle,
\end{array}
\]

that is,

\begin{equation}\label{eq_nabla2F}
\begin{array}{lcl}
\disp |\nabla^2F|^2_B & = & \disp  \disp \frac{1}{2}(F')^p \diver \left( A(\nabla \varrho)\nabla |\nabla \varrho|^2\right) \\[0.5cm]
& & \disp + p(F')^{p}\sqrt{p-1}\frac{h'}{h}|\nabla \varrho|^{p} \varrho_{\nu\nu} + \sqrt{p-1}(m-1) (F')^p \left(\frac{h'}{h}\right)^2|\nabla \varrho|^{p+2} \\[0.5cm]
& & \disp + 2\sqrt{p-1}(F')^p\frac{h'}{h}|\nabla \varrho|^{p} \varrho_{\nu\nu} + \sqrt{p-1}(F')^p \left(\frac{h'}{h}\right)' |\nabla \varrho|^{p+2} \\[0.5cm]
& & \disp + p(F')^{p}\left(\frac{h'}{h}\right)^2|\nabla \varrho|^{p+2} - (F')^p \Ricc(\nu,\nu)|\nabla \varrho|^{p} - \langle \nabla \Delta_p F, \nabla F \rangle. \\[0.5cm]
\end{array}
\end{equation}

Next, we compute
\begin{equation}\label{eq_deltapF}
\Delta_p F
 = \disp \left[(m-1) + \sqrt{p-1}\right](F')^{p-1}\frac{h'}{h}|\nabla \varrho|^p,
\end{equation}
hence differentiating and using $\langle \nabla |\nabla \varrho|, \nabla \varrho \rangle = |\nabla \varrho| \varrho_{\nu\nu}$ we get
\begin{equation}\label{eq_nabladeltapF}
\begin{array}{lcl}
\langle \nabla \Delta_p F, \nabla F\rangle  & = & \disp \left[(m-1) + \sqrt{p-1}\right]\left\{ (F')^p \frac{h'}{h} \langle \nabla |\nabla \varrho|^p, \nabla \varrho \rangle \right.\\[0.4cm]
& & \disp \left. + (F')^p\left( \frac{h'}{h}\right)' |\nabla \varrho|^{p+2} + (p-1)(F')^{p-1} F'' \frac{h'}{h}|\nabla \varrho|^{p+2} \right\} \\[0.4cm]
 & = & \disp \left[(m-1) + \sqrt{p-1}\right](F')^p\left\{ \frac{h'}{h} p|\nabla \varrho|^p \varrho_{\nu\nu} \right. \\[0.4cm]
& & \disp + \left. \left[ \left( \frac{h'}{h}\right)' + \sqrt{p-1}\left(\frac{h'}{h}\right)^2\right]|\nabla \varrho|^{p+2} \right\}. \\[0.4cm]
\end{array}
\end{equation}
Putting together \eqref{eq_nabladeltapF}, \eqref{eq_trBnabla2F} and \eqref{eq_nabla2F}, we obtain
\begin{equation}\label{good!}
\begin{array}{l}
\disp (F')^{-p}\left|\nabla^2F - \frac{\tr_B(\nabla^2 F)}{m}\metric_B\right|^2_B \\[0.5cm]
\qquad \disp = (F')^{-p} \left\{|\nabla^2 F|_B^2 - \frac{[\tr_B(\nabla^2 F)]^2}{m}\right\} \\[0.5cm]
\qquad \disp = \frac{1}{2}\diver \left( A(\nabla \varrho)\nabla |\nabla \varrho|^2\right) + p\sqrt{p-1}\frac{h'}{h}|\nabla \varrho|^{p} \varrho_{\nu\nu} + \sqrt{p-1}(m-1) \left(\frac{h'}{h}\right)^2|\nabla \varrho|^{p+2} \\[0.5cm]
\qquad \disp + 2\sqrt{p-1}\frac{h'}{h}|\nabla \varrho|^{p} \varrho_{\nu\nu} + \sqrt{p-1}\left(\frac{h'}{h}\right)' |\nabla \varrho|^{p+2} + p\left(\frac{h'}{h}\right)^2|\nabla \varrho|^{p+2} - \Ricc(\nu,\nu)|\nabla \varrho|^{p} \\[0.5cm]
\qquad \disp - \left[(m-1) + \sqrt{p-1}\right]\left\{ \frac{h'}{h} p|\nabla \varrho|^p \varrho_{\nu\nu} + \left[ \left( \frac{h'}{h}\right)' + \sqrt{p-1}\left(\frac{h'}{h}\right)^2\right]|\nabla \varrho|^{p+2} \right\} \\[0.5cm]
\qquad \disp - \frac{1}{m} |\nabla \varrho|^{p-2}\left\{ m \frac{h'}{h}|\nabla \varrho|^2 + \left[\sqrt{p-1}-(p-1)\right] \varrho_{\nu\nu}\right\}^2.
\end{array}
\end{equation}
Simplifying, we deduce
\begin{equation}\label{verygood}
\begin{array}{l}
\disp (F')^{-p}\left|\nabla^2F - \frac{\tr_B(\nabla^2 F)}{m}\metric_B\right|^2_B \\[0.5cm]
\qquad \disp = \frac{1}{2}\diver \left( A(\nabla \varrho)\nabla |\nabla \varrho|^2\right) + (-mp+3p-2)\frac{h'}{h}|\nabla \varrho|^p \varrho_{\nu\nu} - \Ricc(\nu,\nu)|\nabla \varrho|^{p} \\[0.4cm]
\qquad \disp - |\nabla \varrho|^{p+2}(m-1) \left[ \left( \frac{h'}{h}\right)' + \left( \frac{h'}{h}\right)^2 \right]  - \frac{1}{m}\left[\sqrt{p-1}-(p-1)\right]^2 \varrho_{\nu\nu}^2.
\end{array}
\end{equation}
Inserting the Riccati equation \eqref{eq_riccati} and the identity
\begin{equation}\label{sec}
\begin{array}{lcl}
\disp h^{-\mu} \diver\big( h^\mu A(\nabla \varrho)\nabla |\nabla \varrho|^2\big) & = & \disp \diver\big(A(\nabla \varrho)\nabla |\nabla \varrho|^2\big) + \mu \frac{h'}{h} \langle A(\nabla \varrho)\nabla|\nabla \varrho|^2, \nabla \varrho \rangle \\[0.5cm]
& = & \disp \diver\big(A(\nabla \varrho)\nabla |\nabla \varrho|^2\big) + 2\mu(p-1) \frac{h'}{h}|\nabla \varrho|^p \varrho_{\nu\nu}
\end{array}
\end{equation}
into \eqref{verygood}, and recalling the definition of $\mu$ in \eqref{def_mu}, we obtain the desired \eqref{bochner_miracolo}.
\end{proof}

\begin{lemma}[\textbf{Key Lemma}]\label{lem_key}
Let $M^m$ be complete. Let $p \in (1,\infty)$, and let $0 \le H \in C(\R^+_0)$ be non-increasing. Consider a model $M_h$ with radial curvature $-H(r)$, and assume that $\Delta_p$ is non-parabolic on $M_h$. Let $u$ be a positive solution to  $\Delta_p u =0$ in an open set $\Omega \subset M$, possibly the entire $M$, and define $\varrho$ according to
$$
u(x) = \sgrh\big(\varrho(x)\big) = \int_{\varrho(x)}^\infty v_h(s)^{-\frac{1}{p-1}}\di s.
$$
When $p>m$, also suppose that $u < \sgrh(0)$ on $\Omega$. Assume that
\begin{equation}\label{eq_lowerricci_conrho}
(i) \quad \inf_M \Ricc > -\infty, \qquad (ii) \quad  \Ricc \ge -(m-1)H(\varrho) \qquad \text{on } \, \Omega
\end{equation}
and that either 
\begin{itemize}
\item[(a)] $m = 2 \le p$, or
\item[(b)] $(a)$ fails and one of the following conditions is satisfied for some sequence $R_j \to \infty$:
\[
\begin{array}{ll}
\log \|u\|_{L^\infty(\Omega \cap B_{R_j})} = o(R_j^2) & \quad \text{if } \, p < m, \\[0.2cm]
\|u\|_{L^\infty(\Omega \cap B_{R_j})} = o(R_j^2) & \quad \text{if } \, p = m, \\[0.2cm]
\|u\|_{L^\infty(\Omega)} < \sgrh(0) & \quad \text{if } \, p > m.
\end{array}
\]
\end{itemize}
Then,
\begin{equation}\label{limsup}
\sup_\Omega |\nabla \varrho| \le \max \Big\{ 1, \limsup_{x \ra \partial \Omega}|\nabla \varrho(x)|\Big\},
\end{equation}
where we set
\[
\limsup_{x \ra \partial \Omega}|\nabla \varrho(x)| \doteq \inf\Big\{ \sup_{\Omega \backslash \overline{V}} |\nabla \varrho| \ : \ V \ \text{open whose closure in $M$ satisfies } \, \overline{V} \subset \Omega \Big\}.
\]
In particular, if $\partial \Omega = \emptyset$ then $|\nabla \varrho| \le 1$.
\end{lemma}

\begin{remark}\label{rem_importante}
\emph{Bound \eqref{eq_lowerricci_conrho} automatically holds in the following relevant cases:
\begin{itemize}
\item[1)] $\inf_M \Ricc > -\infty$ and $\Ricc \ge -(m-1) \kappa^2$ on $\Omega$, for some constant $\kappa \ge 0$, choosing $H(t) = \kappa^2$. If $\kappa = 0$, we further assume that $p < m$ in order for $\Delta_p$ to be non-parabolic on $M_h$;
\item[2)] $M$ satisfies $(\Hp)$ for some $p \in (1,m]$ and non-increasing $H$, $\Omega \subset M \backslash \{o\}$ and $u$ is the restriction to $\Omega$ of the Green kernel of $M$ with pole at $o$. Indeed, by Proposition \ref{prop_basiccomp}, the fake distance $\varrho$ associated to $u$ satisfies $\varrho \le r$ and therefore
$$
\Ricc \ge -(m-1)H(r) \ge -(m-1)H(\varrho) \qquad \text{on } \, M.
$$
\end{itemize}
}
\end{remark}

\begin{proof}
%
%
Assume that
	\begin{equation}\label{eq_contrad_nabla}
	\sup_\Omega |\nabla \varrho|^2 > \max \Big\{ 1, \limsup_{y \to \partial \Omega}|\nabla \varrho(y)|^2 \Big\}, \qquad \Big( \sup_\Omega |\nabla \varrho|^2 > 1 \ \ \  \text{if \, } \Omega = M\Big). 	
	\end{equation}
Then, we can pick $\delta_0>0$ such that, for each $\delta \in [\delta_0, \sup_\Omega |\nabla \varrho|^2-1)$, the set
\begin{equation}\label{def_Udelta}
U_\delta = \Big\{|\nabla \varrho|^2 >1+\delta\Big\}
\end{equation}
is non-empty and $\overline{U}_\delta\cap \partial \Omega=\emptyset$. We remark that $|\nabla \varrho| \in C^\infty(\overline{U}_\delta)$, by the regularity of $u$ on the complementary of its stationary points. Suppose first that $H_* \doteq \inf H > 0$. Inserting \eqref{eq_lowerricci_conrho} into \eqref{bochner_miracolo} shows that the following inequality holds on $U_\delta$:
\begin{equation}\label{bochner_miracolo_upper_2}
\begin{array}{lcl}
\disp \frac{1}{2} h^{-\mu} \diver\big( h^{\mu} A(\nabla \varrho) \nabla |\nabla \varrho|^2 \big) & \ge & \disp |\nabla \varrho|^p(m-1)H(\varrho)\Big[|\nabla \varrho|^2-1\Big] \\[0.4cm]
& \ge & (m-1)H_*|\nabla \varrho|^{p+2}\left[\frac{\delta}{\delta+1}\right] \\[0.4cm]
& \ge & \disp c_0|\nabla \varrho|^{p+2},
\end{array}
\end{equation}
where $c_0 = (m-1)\delta_0/(1+\delta_0)H_*$. 
%
For $R \ge 1$, pick $\psi \in C^2_c(B_{2R}(o))$ and $\lambda \in C^1(\R)$
satisfying
$$
\begin{array}{ll}
\disp 0 \le \psi \le 1 \ \ \text{ on } \, M, \qquad \psi \equiv 1 \ \  \text{ on } \, B_R(o), \qquad |\nabla \psi| \le \frac{8}{R} \psi^{1/2} \\[0.3cm]
0 \le \lambda \le 1 \ \  \text{ on } \, \R, \qquad \mathrm{supp}(\lambda) \subset (1+\delta, \infty), \qquad  \lambda'\ge 0 \ \ \text{ on } \, \R.
\end{array}
$$
For $\eta, \alpha \ge 1$ to be chosen later, we use the test function
$$
\varphi = \lambda(|\nabla \varrho|^2) \psi(x)^\eta |\nabla \varrho|^\alpha \in C^1_c(U_\delta).
$$
in the weak definition of \eqref{bochner_miracolo_upper_2}. Writing $A = A(\nabla \varrho)$, $\lambda = \lambda(|\nabla \varrho|^2)$
we get
\begin{equation}\label{array_basic}
\begin{array}{l}
\disp \frac{\alpha}{2} \int \langle A \nabla |\nabla \varrho|^2, \lambda |\nabla \varrho|^{\alpha-1} \psi^\eta \nabla |\nabla \varrho| \rangle h(\varrho)^{\mu} + c_0 \int \lambda \psi^\eta |\nabla \varrho|^{\alpha+p+2} h(\varrho)^{\mu}\\[0.5cm]
\qquad \le \quad \disp - \frac{\eta}{2} \int \psi^{\eta-1} \lambda |\nabla \varrho|^\alpha \langle A\nabla|\nabla \varrho|^2, \nabla \psi \rangle h(\varrho)^{\mu} - \frac{1}{2} \int \lambda' \psi^\eta \langle A\nabla |\nabla \varrho|^2, \nabla |\nabla \varrho|^2 \rangle h(\varrho)^{\mu}\\[0.5cm]
\qquad \le \quad \disp - \frac{\eta}{2} \int \psi^{\eta-1} \lambda |\nabla \varrho|^\alpha \langle A\nabla|\nabla \varrho|^2, \nabla \psi \rangle h(\varrho)^{\mu},
\end{array}
\end{equation}
where, in the last inequality, we used $\lambda' \ge 0$ and the non-negativity of $A$. From the expression of the eigenvalues of $A$,
$$
\langle A \nabla |\nabla \varrho|^2, \nabla |\nabla \varrho| \rangle = \disp\frac{\langle A \nabla |\nabla \varrho|^2, \nabla |\nabla \varrho|^2 \rangle}{2|\nabla \varrho|} \ge  \disp \frac{\min\{1,p-1\}}{2} |\nabla \varrho|^{p-3} \big| \nabla |\nabla \varrho|^2 \big|^2
$$
while, by Cauchy-Schwarz inequality,
$$
\begin{array}{lcl}
\disp \langle A \nabla |\nabla \varrho|^2, \nabla \psi \rangle & \le & \disp \Big\{\langle A \nabla |\nabla \varrho|^2, \nabla |\nabla \varrho|^2 \rangle \Big\}^{1/2} \Big\{ \disp \langle A \nabla \psi, \nabla \psi\rangle\Big\}^{1/2} \\[0.5cm]
& \le & \disp \max\{1,p-1\} |\nabla \varrho|^{p-2} \big| \nabla |\nabla \varrho|^2\big| |\nabla \psi| \\[0.5cm]
& \le & \disp \frac{8\max\{1,p-1\}}{R} |\nabla \varrho|^{p-2} \big| \nabla |\nabla \varrho|^2\big| \psi^{1/2}
\end{array}
$$
Substituting into \eqref{array_basic} we obtain
\begin{equation}\label{array_basic_2}
\begin{array}{l}
\disp \frac{\alpha}{4} \min\{1,p-1\} \int \psi^\eta \lambda |\nabla \varrho|^{p+\alpha -4} \big| \nabla |\nabla \varrho|^2 \big|^2 h(\varrho)^{\mu}+ c_0 \int \lambda \psi^\eta |\nabla \varrho|^{\alpha+p+2} h(\varrho)^{\mu} \\[0.5cm]
\qquad \le \quad \disp \frac{4\eta \max\{1,p-1\}}{R} \int \psi^{\eta-\frac{1}{2}} \lambda |\nabla \varrho|^{\alpha +p-2} \big| \nabla |\nabla \varrho|^2 \big|h(\varrho)^{\mu}
\end{array}
\end{equation}
By Young's inequality,
$$
2\psi^{\eta-\frac{1}{2}} \lambda |\nabla \varrho|^{\alpha +p-2} \big| \nabla |\nabla \varrho|^2 \big| \le \tau \psi^\eta \lambda|\nabla \varrho|^{\alpha + p-4} \big| \nabla |\nabla \varrho|^2\big|^2 + \frac{1}{\tau} \psi^{\eta-1} \lambda|\nabla \varrho|^{\alpha+p},
$$
whence, choosing
$$
\tau = \frac{\alpha R \min\{1,p-1\}}{8\eta \max\{1,p-1\}}
$$
and inserting into \eqref{array_basic_2}, we deduce the existence of a constant $c_p$ which depends only on $p$ such that
\begin{equation}\label{array_basic_3}
c_0 \int \lambda \psi^\eta |\nabla \varrho|^{\alpha+p+2} h(\varrho)^{\mu}\le c_p \frac{\eta^2}{\alpha R^2} \int \lambda \psi^{\eta-1}|\nabla \varrho|^{\alpha+p}h(\varrho)^{\mu}.
\end{equation}
We next apply Young's inequality again with exponents
$$
q = \frac{p+\alpha+2}{p+\alpha}, \qquad q' = \frac{p+\alpha+2}{2}
$$
and a free parameter $\bar \tau$ to obtain
$$
\psi^{\eta-1} |\nabla \varrho|^{p+\alpha} \le \frac{\bar \tau^q}{q} \psi^\eta |\nabla \varrho|^{p+\alpha+2} + \frac{1}{q'\bar\tau^{q'}} \psi^{\eta - q'}.
$$
We choose $\eta = 2q' = p+\alpha+2$ and $\bar \tau$ such that
$$
c_p \frac{\eta^2}{\alpha R^2} \frac{\bar \tau^q}{q} = \frac{c_0}{2},
$$
so that, inserting into \eqref{array_basic_3} and rearranging, we deduce that there exists a constant $c_1= c_1(c_0, c_p)$ such that
\begin{equation}\label{array_basic_4}
\begin{array}{lcl}
\disp \int \lambda \psi^\eta |\nabla \varrho|^{\alpha+p+2}h(\varrho)^{\mu} & \le & \disp \frac{c_0}{p+\alpha} \left[ \frac{2 c_p}{c_0} \frac{\eta^2}{\alpha R^2} \frac{p+\alpha}{p+\alpha+2} \right]^{\frac{p+\alpha+2}{2}} \int \lambda \psi^{\eta/2} h(\varrho)^{\mu}\\[0.5cm]
& \le & \disp \left[ \frac{c_1(p+\alpha+2)}{R^2}\right]^{\frac{p+\alpha+2}{2}} \int \lambda \psi^{\eta/2}h(\varrho)^{\mu}.
\end{array}
\end{equation}
Set
$$
I(R) = \int_{B_R} \lambda h(\varrho)^{\mu}.
$$
Choose $\lambda$ close enough to the indicator function of $(1+\delta,\infty)$ to guarantee that $I(R_0)>0$ for some $R_0>0$. Taking into account the definition of $\psi$ and the fact that $|\nabla \varrho|^2 \ge 1+\delta$ on the support of $\lambda \psi$, \eqref{array_basic_4} yields
$$
I(R) \le \left[\frac{c_1(p+\alpha+2)}{R^2(1+\delta)}\right]^{\frac{p+\alpha+2}{2}} I(2R).
$$
Choosing $\alpha$ to satisfy
$$
\frac{c_1(p+\alpha+2)}{R^2(1+\delta)} = \frac{1}{e}, \qquad \text{so that }\qquad \frac{p+\alpha+2}{2} = \frac{R^2(1+\delta)}{2c_1 e},
$$
we get
$$
I(R) \le e^{- \frac{p+\alpha+2}{2}} I(2R) = e^{-\frac{R^2(1+\delta)}{2 c_1 e}} I(2R).
$$
Iterating and taking logarithms as in \cite[Lem. 4.7]{prsmemoirs} shows that there exists $S>0$ independent of $R, \delta$ such that for each $R > 2R_0$,
\begin{equation}\label{limit_eq}
\frac{\log I(R)}{R^2} \ge \frac{\log I(R_0)}{R^2} + S \frac{(1+\delta)}{c_1}.
\end{equation}
To conclude our desired contradiction, we estimate $I(R)$. First, observe that if we can prove that  
\begin{equation}\label{eq_growth_hvarrho}
\liminf_{R \to \infty} \frac{\log \|h(\varrho)^\mu\|_{L^\infty(\Omega \cap B_R)}}{R^2} = 0, 
\end{equation}
then we reach a contradiction as follows: by \eqref{eq_lowerricci_conrho}, $(i)$ and Bishop-Gromov comparison theorem, recalling that $0 \le \lambda \le 1$ we get 
	\[
	\log I(R) \le \log |B_R| + \log \|h(\varrho)^\mu\|_{L^\infty(\Omega \cap B_R)} \le C_1 + C_2 R + \log \|h(\varrho)^\mu\|_{L^\infty(\Omega \cap B_R)}
	\]
for suitable constants $C_j$, so we can let $R \to \infty$ in \eqref{limit_eq} along a sequence realizing the liminf and conclude $0 \ge S(1+\delta)/c_1$, which is absurd. \\[0.2cm]
\noindent \textbf{Case (a).} Assumption $m=2, p \ge 2$ is equivalent to $\mu \ge 0$ in \eqref{def_mu}. If $\mu = 0$ then \eqref{eq_growth_hvarrho} trivially holds and the thesis follows. If $\mu >0$, $h(\varrho)^\mu$ is unbounded when $\varrho(x) \to \infty$, that is, when $u(x) \to 0$. The conclusion will follow from \eqref{eq_growth_hvarrho} via an approximation argument. We let $c>0$ and consider the function $\varrho_c$ defined by the identity
	\[
	u(x) + c = \int_{\varrho_c(x)}^\infty v_h(s)^{-\frac{1}{p-1}}\di s \qquad \text{on } \, \Omega_c \doteq \big\{ x \in \Omega : u(x) + c < \sgrh(0)\big\}.
	\]	
Notice that $\Omega_c \uparrow \Omega$ and $\varrho_c \uparrow \varrho$ pointwise in $\Omega$ as $c \downarrow 0$, whence 
	\[
	\Ricc \ge - (m-1)H(\varrho) \ge -(m-1)H(\varrho_c) \qquad \text{on } \, \Omega_c,
	\]
and moreover	
	\begin{equation}\label{eq_goodmono}
	|\nabla \varrho_c| = |\nabla u| v_h(\varrho_c)^{\frac{1}{p-1}} \le |\nabla u| v_h(\varrho)^{\frac{1}{p-1}} = |\nabla \varrho|.
	\end{equation}
Also, $\varrho_c \to \varrho$ locally in $C^1(\Omega)$, as can be argued by differentiating the very definitions of $\varrho_c$ and $\varrho$. By construction, $\varrho_c$ is bounded, whence \eqref{eq_growth_hvarrho} holds for $\varrho = \varrho_c$ and we can therefore deduce from the integral estimates leading to \eqref{limit_eq} with $\varrho_c$ replacing $\varrho$ the inequality
	\begin{equation}\label{eq_good_c}
	\sup_{\Omega_c} |\nabla \varrho_c| \le \max \Big\{ 1, \limsup_{\Omega_c \ni y \ra \partial \Omega_c}|\nabla \varrho_c(y)| \Big\}.
	\end{equation}
However, if $y \in \partial \Omega_c \cap \Omega$ then $\varrho_c(y) = 0$, and therefore $\nabla \varrho_c(y)=0$ by the first equality in \eqref{eq_goodmono}. Again by \eqref{eq_goodmono} we conclude  
	\[
	\limsup_{\Omega_c \ni y \ra \partial \Omega_c}|\nabla \varrho_c(y)| = \left\{ \begin{array}{ll}
	\disp \limsup_{\Omega_c \ni y \to \partial \Omega}|\nabla \varrho_c(y)| \le \limsup_{y \to \partial \Omega} |\nabla \varrho(y)| & \quad \text{if } \, \partial \Omega \neq \emptyset \\[0.3cm]
	0 & \quad \text{if } \, \Omega = M.
	\end{array}\right.
	\]	
Inserting the latter into \eqref{eq_good_c}, observing that $\varrho_c \to \varrho$ in $C^1_\loc(\Omega)$ guarantees that $|\nabla \varrho_c(x)| \to |\nabla \varrho(x)|$ for each $x \in \Omega$, and letting $c \to 0$ we get the desired \eqref{limsup} (with $|\nabla \varrho| \le 1$ if $\Omega = M$).\\[0.2cm]
\noindent \textbf{Case (b).} Since $(a)$ fails, in this case $\mu < 0$. Hence, to estimate $h(\varrho)^\mu$ we need to consider its behaviour as $\varrho(x) \to 0$, that is, as $u(x) \to +\infty$. If $p>m$, our $L^\infty$ condition on $u$ implies that $\varrho$ is bounded below by a positive constant, whence $h(\varrho)^\mu \in L^\infty(\Omega)$ and \eqref{eq_growth_hvarrho} holds. If $p \le m$, we know that $h(t) \sim t$ and $\sgrh(t) \sim \mu(t)$ as $t \to 0$, where $\mu(t)$ is as in \eqref{def_mup}. Hence, here exist constants $C_j= C_j(m,p,H)$ such that
	\[
	\|h(\varrho)^\mu\|_{L^\infty(\Omega \cap B_R)} \le \left\{ \begin{array}{ll}
	C_1 + C_2\|u\|_{L^\infty(\Omega \cap B_R)}^{\frac{|\mu|(p-1)}{m-p}} & \quad \text{if } \, p < m \\[0.4cm]
	C_1 + \exp \left\{ C_2 \|u\|_{L^\infty(\Omega \cap B_R)} \right\} & \quad \text{if } \, p = m,
	\end{array}\right.
	\]
and \eqref{eq_growth_hvarrho} follows by the growth conditions in $(b)$.\\[0.2cm]
Having proved the lemma for $H_*>0$, it remains to examine the case $H_* = 0$. Fix a small $c >0$, consider a model from below of curvature $-H(t)-c$, let $v_{h,c}$ be the volume of its geodesic spheres and let $\sgrh_c$ be its Green kernel. Note that
	\[
	\sgrh_c \uparrow \sgrh \qquad \text{in } \, C^1_\loc(\R^+) \quad \text{as } \, c \downarrow 0.
	\]
Define $\varrho_c$ as the fake distance associated to $u$ and $\sgrh_c$, the definition being meaningful on
	\[
	\Omega_c \doteq \Big\{ x \in \Omega \ : \ u(x) < \sgrh_c(0) \Big\} \subset \Omega.
	\]
Note that $\Omega_c \equiv \Omega$ if $p \le m$, since $\sgrh_c(0) = \infty$, while $\Omega_c \uparrow \Omega$ if $p>m$. Moreover, $\varrho_c \ra \varrho$ locally uniformly in $C^1(\Omega)$ and monotonically from below as $c \downarrow 0$, which implies
	\[
	|\nabla \varrho(x)| = \lim_{c \ra 0} |\nabla \varrho_c(x)|	\qquad \forall \, x \in \Omega.
	\]
We claim that
	\begin{equation}\label{eq_compvhtau}
	v_{h,c}(\varrho_c) \le v_h(\varrho) \qquad \text{on } \, \Omega_c.
	\end{equation}
We postpone for a moment its proof, and conclude the argument. Observe that 
	\[
	\Ricc \ge - (m-1)H(\varrho) \ge -(m-1)[H(\varrho_c) +c] \qquad \text{on } \, \Omega_c,
	\]
and that, under assumption $(b)$, when $p>m$ we can guarantee that $\|u\|_\infty < \sgrh_c(0)$ if $c$ is small enough (in this case, note that $\Omega_c \equiv \Omega$). Hence, we can apply the first part of the proof to $\varrho_c$ and use that $\Omega_c \uparrow \Omega$ to deduce, for each $x \in \Omega$,
	\begin{equation}\label{eq_lower_c}
	|\nabla \varrho(x)| = \lim_{c \ra 0} |\nabla \varrho_c(x)| \le \liminf_{c \ra 0} \max \Big\{ 1, \limsup_{\Omega_c \ni y \to \partial \Omega_c}|\nabla \varrho_c(y)|\Big\}.
	\end{equation}
Next, \eqref{eq_compvhtau} implies
	\[
	|\nabla \varrho_c| = |\nabla u| v_{h,c}(\varrho_c)^{\frac{1}{p-1}} \le |\nabla u| v_{h}(\varrho)^{\frac{1}{p-1}} = |\nabla \varrho| \qquad \text{on } \, \Omega_c.
	\]
Under assumption $(b)$, or if $(a)$ holds with $p = 2$, we have $\Omega_c \equiv \Omega$ and thus
	\[
	\limsup_{\Omega_c \ni y \ra \partial \Omega_c}|\nabla \varrho_c(y)| = \limsup_{y \ra \partial \Omega}|\nabla \varrho_c(y)| \le \limsup_{y \ra \partial \Omega}|\nabla \varrho(y)|,
	\]
which, together with \eqref{eq_lower_c}, implies \eqref{limsup}. If $(a)$ holds with $p > 2$, $\varrho_c$ vanishes on $\partial \Omega_c \cap \Omega$ and therefore
$$
|\nabla \varrho_c(x)| = |\nabla u(x)| v_{h,c}(\varrho_c(x))^{\frac{1}{p-1}} \ra 0 \qquad \text{as } \, x \ra \partial \Omega_c \cap \Omega,
$$
hence
	\[
	\limsup_{\Omega_c \ni y \ra \partial \Omega_c}|\nabla \varrho_c(y)| = \limsup_{\Omega_c \ni y \ra \partial \Omega}|\nabla \varrho_c(y)| \le \limsup_{y \ra \partial \Omega}|\nabla \varrho(y)|,
	\]
again proving \eqref{limsup}.\par
	To show \eqref{eq_compvhtau}, by Sturm comparison $v_{h,c}/v_h$ is increasing on $\R^+$, thus \cite[Prop. 4.12]{bmr2} (see also \cite[Lem. 4.11]{bmpr}) implies the inequality $|(\log \sgrh_c)'| \ge |(\log \sgrh)'|$ on $\R^+$, that is,
\[
\frac{v_{h,c}^{-\frac{1}{p-1}}}{\sgrh_c}(t) \ge \frac{v_{h}^{-\frac{1}{p-1}}}{\sgrh}(t) \qquad \forall \, t \in \R^+.
\]
We evaluate at $t = \varrho_c$ and use $\varrho_c \le \varrho$ together with the monotonicity of $|(\log \sgrh)'(t)|$ that follows from Lemma \ref{lem_ODE}, to deduce
\[
\frac{v_{h,c}(\varrho_c)^{-\frac{1}{p-1}}}{\sgrh_c(\varrho_c)} \ge \frac{v_{h}(\varrho_c)^{-\frac{1}{p-1}}}{\sgrh(\varrho_c)} \ge \frac{v_{h}(\varrho)^{-\frac{1}{p-1}}}{\sgrh(\varrho)}.
\]
Inequality \eqref{eq_compvhtau} follows since $\sgrh_c(\varrho_c)= \sgrh(\varrho) = u$, concluding the proof.
\end{proof}


We are ready to prove Theorem \ref{teo_good_intro} in the Introduction.
%
%
%
%
%

\begin{theorem}\label{teo_good}
Suppose that $M^m$ satisfies $(\Hp)$ for some $p \in (1,m]$ and
$$
H(t) \ge 0, \qquad H(t) \ \text{ non-increasing on $\R^+$.}
$$
Then, having defined $\varrho$ as in Definition \ref{def_fake},
\begin{itemize}
\item[(i)] $|\nabla \varrho| \le 1$ on $M \backslash \{o\}$.
\item[(ii)] Equality $|\nabla \varrho(x)|=1$ holds for some $x \in M \backslash \{o\}$ if and only if $\varrho = r$ and $M$ is the radially symmetric model $M_h$.
\end{itemize}
\end{theorem}

\begin{proof}
$(i)$. Inequality $\varrho \le r$ holds because of Proposition \ref{prop_basiccomp}, hence $(2)$ in Theorem  \ref{teo_localsingular} guarantees
	\[
	|\nabla \varrho(x)| = v_h(\varrho(x))^{\frac{1}{p-1}} |\nabla \gr(x)| \le v_h(r(x))^{\frac{1}{p-1}} |\nabla \gr(x)| \ra 1	
	\]
as $x \ra o$. Thus, $\limsup_{x \ra o} |\nabla \varrho(x)| \le 1$. For $\delta>0$, fix a small ball $B$ around $o$ so that $|\nabla \varrho| \le 1+ \delta$ on $B \backslash \{o\}$. In view of Remark \ref{rem_importante}, and since $\gr$ is bounded on $M \backslash \overline{B}$ by its very construction and the maximum principle, we are in the position to apply Lemma \ref{lem_key} to conclude $|\nabla \varrho| \le 1+\delta$ on $M \backslash B$. Item $(i)$ follows by letting $\delta \to 0$ and shrinking $B$ accordingly. To show $(ii)$, we observe that because of \eqref{bochner_miracolo} and the monotonicity of $H$, the function $u = 1-|\nabla \varrho|^2 \ge 0$ solves
$$
\begin{array}{lcl}
\frac{1}{2} h^{-\mu} \diver\big( h^\mu A(\nabla \varrho)\nabla u\big) & \le & \disp -|\nabla \varrho|^{p} \Big[ \Ricc(\nu,\nu) + (m-1)H(\varrho)|\nabla \varrho|^2\Big] \\[0.2cm]
& \le & -(m-1)|\nabla \varrho|^{p} \Big[-H(r) + H(\varrho)|\nabla \varrho|^2\Big] \\[0.2cm]
& \le & \disp (m-1)H(r)|\nabla \varrho|^pu.
\end{array}
$$
If $u$ vanishes at some point, by the strong minimum principle $u \equiv 0$ on $M$, that is, $|\nabla \varrho| \equiv 1$. In this case, again by \eqref{bochner_miracolo} we deduce
\begin{equation}\label{lahessiana}
\nabla^2 F = \frac{1}{m} \mathrm{Tr}_B(\nabla^2 F) \metric_B \qquad \text{on } \, M,
\end{equation}
with $F$ as in \eqref{def_mu}. Since $|\nabla \varrho|=1$, the integral curves of the flow $\Phi_t$ of $\nabla \varrho$ are unit speed geodesics. For $x \in M\backslash\{o\}$, because of the completeness of $M$ the geodesic $\Phi_t(x)$ is defined on the maximal interval $(-\varrho(x), \infty)$, and $\lim_{t \ra - \varrho(x)}\Phi_t(x) = o$, being $o$ the unique zero of $\varrho$. Hence, $\Phi_t$ is a unit speed geodesic issuing from $o$ to $x$. If $x \in M\backslash \cut(o)$, it therefore holds $\varrho(x) = r(x)$, and by continuity $\varrho = r$ on $M$. The function $r$ is thus $C^{1}$ outside of $o$, and this implies $\cut(o)= \emptyset$, that is, $o$ is a pole of $M$. Indeed, the distance function $r$ is not differentiable at any point $y \in \cut(o)$ joined to $o$ by at least two minimizing geodesics. The set of such points is dense in $\cut(o)$ by \cite{bishop, wolter}, so $\cut(o) = \emptyset$ whenever $r$ is everywhere differentiable outside of $o$. Rewriting \eqref{lahessiana} in terms of $\nabla^2 \varrho = \nabla^2 r$ we get
\begin{equation}\label{eq_hessianr}
\nabla^2 r = \frac{h'(r)}{h(r)} \Big( \metric - \di r \otimes \di r \Big) \qquad \text{on } \, M \backslash \{o\}.
\end{equation}
Integrating along geodesics we deduce that $M$ is isometric to $M_h$.
\end{proof}

\begin{remark}[\textbf{Hardy weights}]
\emph{When rephrased in terms of $\gr$, the bound $|\nabla \varrho| \le 1$ becomes
\begin{equation}\label{esti_gradientG}
|\nabla \log \gr| \le |( \log \sgrh)'|\big(\varrho(x)\big).
\end{equation}
We mention that the function $|\nabla \log \gr|$ naturally appears as a weight in the Hardy inequality
$$
\left( \frac{p-1}{p}\right)^p \int |\nabla \log \gr|^p |\psi|^p \le \int |\nabla \psi|^p \qquad \forall \, \psi \in \lip_c(M),
$$
which holds on every manifold where $\Delta_p$ is non-parabolic (see \cite[Prop. 4.4]{bmr5}). Hence,  \eqref{esti_gradientG} can be somehow seen as a comparison theorem for Hardy weights. In this respect, it is worth to notice that the weight of $M_h$ transplanted to $M$, that is, the function
$$
\left( \frac{p-1}{p}\right)^p |( \log \sgrh)'|\big(r(x)\big),
$$
is a Hardy weight on $M$ provided that $M_h$ is a model from \emph{above} for $M$, see Section 5 in \cite{bmr5}. For a systematic study of Hardy weights and their role in geometric problems in the linear case $p=2$ we refer the reader to \cite{bmr2, bmr3, bmr4}.
}
\end{remark}

%
%
%
%

\subsection{The sharp gradient estimate for $p$-harmonics: another proof}
To illustrate the versatility of the key Lemma \ref{lem_key}, we consider the case where $p \in (1,\infty)$ and $u>0$ solves $\Delta_pu =0$ on the entire $M$, and we suppose that
$$
\Ricc \ge -(m-1)\kappa^2 \qquad \text{on } \, M,
$$
for some constant $\kappa \ge 0$. By a recent result in \cite{sungwang},
\begin{equation}\label{eq_uppergradient}
|\nabla \log u| \le \frac{(m-1)\kappa}{p-1} \qquad \text{on } \, M.
\end{equation}
The upper bound is sharp in view of warped product manifold $M = \R \times N^{m-1}$, for a compact $(N,\di s^2_N)$ with non-negative Ricci curvature, endowed with the metric $\metric = \di t^2 + e^{-2t}\di s^2_N$. A direct computation shows that $\Ricc \ge -(m-1)$ and
$$
u(x) = \exp\left\{ \frac{m-1}{p-1}t\right\} \qquad \text{is $p$-harmonic on $M$ with } \, |\nabla \log u| = \frac{m-1}{p-1}.
$$
To obtain their sharp global estimate, in \cite{sungwang} the authors rely on the next local gradient bound for $p$-harmonic functions in \cite{wangzhang}, which was proved via a subtle Moser iteration procedure:
\begin{equation}\label{esti_gad_chengyau}
|\nabla \log u| \le C_{m,p} \frac{1 + \kappa R}{R} \qquad \text{on } \, B_R(x),
\end{equation}
whenever $u$ is defined on $B_{4R}(x)$.

\begin{remark}
\emph{We underline that the constant $C_{m,p}$ in \eqref{esti_gad_chengyau} satisfies $(p-1)C_{m,p} \ra \infty$ as $p \ra 1$, which makes \eqref{esti_gad_chengyau} unsuitable for the limit procedures described in the next sections.
}
\end{remark}

As a consequence of Lemma \ref{lem_key}, we can give a direct proof of a version of \eqref{eq_uppergradient} which is valid for sets $\Omega \subset M$ with possibly non-empty boundary, under suitable growth assumptions on $u$ or if $m=2 \le p$. It should be stressed that adapting the proof of \cite{sungwang} to sets with boundary seems to be nontrivial since the upper bound in \eqref{esti_gad_chengyau} blows up as $R \ra 0$, while the global boundedness of $|\nabla \log u|$ which follows from \eqref{esti_gad_chengyau} for entire solutions plays a crucial role in the derivation of \eqref{eq_uppergradient}.
\begin{theorem}\label{teo_bellagradiente}
Let $M^m$ be a complete manifold, let $\Omega \subset M$ be an open set (possibly the entire $M$) and suppose that
\begin{equation}\label{ricci_per_Liouville}
\inf_M \Ricc > -\infty, \qquad \Ricc \ge -(m-1)\kappa^2 \qquad \text{on } \, \Omega,
\end{equation}
for some constant $\kappa \in \R^+_0$. Let $u>0$ solve $\Delta_pu =0$ on $\Omega$ for some $p \in (1,\infty)$, and assume that either $m = 2 \le p$ or that one of the following conditions holds for some sequence $R_j \to \infty$:
\[
\begin{array}{ll}
\log \|u\|_{L^\infty(\Omega \cap B_{R_j})} = o(R_j^2) & \quad \text{if } \, p < m, \\[0.2cm]
\|u\|_{L^\infty(\Omega \cap B_{R_j})} = o(R_j^2) & \quad \text{if } \, p = m, \\[0.2cm]
u \in L^\infty(\Omega) & \quad \text{if } \, p > m.
\end{array}
\]
Then,
$$
|\nabla \log u| \le \max\left\{ \frac{m-1}{p-1}\kappa,  \ \limsup_{x \ra \partial \Omega} |\nabla \log u|\right\}.
$$
\end{theorem}

\begin{proof}
The case $\kappa = 0$ can be handled by choosing a sequence $\kappa_j \downarrow 0$ and letting $j \ra \infty$ in the resulting estimate, so we can suppose without loss of generality that $\kappa>0$. For $c>0$ small, we define $\varrho_c$ by the formula
\begin{equation}\label{def_rhoc_grad}
c u(x) = \mathscr{G}^\kappa\big( \varrho_c(x)\big) = \int_{\varrho_c(x)}^\infty v_\kappa(s)^{-\frac{1}{p-1}} \di s \qquad \text{on } \, \Omega_c \doteq \big\{ x \in \Omega : c u(x) < \mathscr{G}^\kappa(0) \big\}.
\end{equation}
If $p \le m$, then $\Omega_c \equiv \Omega$ for each $c>0$ since $\sgrh(0)=+\infty$. If $p>m$ and $u$ is bounded, up to reducing $c$ we can guarantee that $c\|u\|_\infty < \sgrh_c(0)$, so $\Omega_c \equiv \Omega$ for small enough $c$. Summarizing, in our assumptions on $u$,$m$,$p$ we ensure $\Omega_c \equiv \Omega$ for small $c$ unless $m=2< p$ and $u$ is unbounded, a case which will be examined later. By its very definition, $\varrho_c>0$ on $\Omega$. Since we assume a constant lower bound on $\Ricc$, by Remark \ref{rem_importante}, setting $H(t) = \kappa^2$ assumption \eqref{eq_lowerricci_conrho} holds. Therefore, we can apply the key Lemma \ref{lem_key} to $\varrho_c$ and infer
\begin{equation}\label{bound_rhoc_grad}
\sup_\Omega |\nabla \varrho_c| \le \max\Big\{ 1, \limsup_{x \ra \partial \Omega} |\nabla \varrho_c(x)|\Big\}.
\end{equation}
Rephrasing in terms of $u$, we deduce from \eqref{bound_rhoc_grad} that for each $x \in \Omega$
\begin{equation}\label{eq_bella}
\begin{array}{lcl}
\disp |\nabla \log u(x)| & \le & \disp \disp |(\log \mathscr{G}^{\kappa})'|(\varrho_c(x)) \cdot \max \left\{ 1, \limsup_{y \ra \partial \Omega} \frac{|\nabla \log u(y)|}{|(\log \mathscr{G}^{\kappa})'|(\varrho_c(y))} \right\} \\[0.5cm]
& \le & \disp \disp |(\log \mathscr{G}^{\kappa})'|(\varrho_c(x)) \cdot \max \left\{ 1, \frac{p-1}{(m-1)\kappa}\limsup_{y \ra \partial \Omega}|\nabla \log u(y)| \right\}, \\[0.5cm]
\end{array}
\end{equation}
where the last inequality follows since $|(\log \mathscr{G}^\kappa)'|$ is decreasing on $\R^+$ because of Lemma \ref{lem_ODE}, and since
$$
|(\log \mathscr{G}^\kappa)'|(t) \downarrow \frac{m-1}{p-1}\kappa \qquad \text{as } \, t \ra \infty.
$$
By \eqref{def_rhoc_grad}, for each fixed $x \in M$, $\varrho_c(x) \ra \infty$ as $c \ra 0$. Taking limits in \eqref{eq_bella} as $c \ra 0$ we obtain
\begin{equation}\label{eq_bella_2}
\disp |\nabla \log u(x)| \le \disp \disp \frac{m-1}{p-1} \kappa \cdot \max \left\{ 1, \frac{p-1}{(m-1)\kappa}\limsup_{y \ra \partial \Omega} |\nabla \log u(y)| \right\},
\end{equation}
as claimed. If $m=2 <p$ and $u$ is unbounded, Lemma \ref{lem_key}, $(a)$ guarantees that
$$
|\nabla\varrho_c| \le \max \Big\{ 1, \limsup_{\Omega_c \ni y \ra \partial \Omega_c} |\nabla \varrho_c(y)| \Big\} \qquad \text{on } \, \Omega_c.
$$
However, $\varrho_c=0$ on $\partial \Omega_c \cap \Omega$, whence $|\nabla \varrho_c(x)| = c|\nabla u(x)| v_\kappa(\varrho_c(x))^{\frac{1}{p-1}} \to 0$ as $x \ra \partial \Omega_c \cap \Omega$ and we deduce
$$
|\nabla\varrho_c| \le \max \Big\{ 1, \limsup_{\Omega_c \ni y \ra \partial \Omega} |\nabla \varrho_c(y)| \Big\} \qquad \text{on } \, \Omega_c.
$$
The required conclusion then follows as in the case $\Omega_c \equiv \Omega$, once we observe that $\varrho_c(x) \ra \infty$ and $\Omega_c \uparrow \Omega$ as $c \ra 0$.
\end{proof}

\subsection{Capacitors and exterior domains}

Let $K$ be a $C^1$, relatively compact open set, and let $u$ be the $p$-capacity potential of $(K,M)$. Assuming the Ricci curvature bound \eqref{ricci_per_Liouville}, by Theorem \ref{teo_bellagradiente} a global estimate for $|\nabla \log u|$ reduces to an estimate on
$$
\sup_{\partial K} |\nabla \log u|.
$$
Barriers for $\log u$ on $\partial \Omega$ are described in \cite[Sect. 3]{kotschwarni} for $1< p< m$. Here, we extend the argument to every $p$ and slightly shorten the proof. We recall that the Dirichlet kernel for $\Delta_p$ on the ball $B_{R}$ in the model of curvature $-\kappa^2$ is given by
$$
\sgr^\kappa_{R}(t) \doteq \int_t^{R} v_\kappa(s)^{-\frac{1}{p-1}}\di s.
$$
\begin{proposition}\label{prop_boundarygradient}
Let $u$ be the $p$-capacity potential of $(K,M)$. Fix $x \in \partial K$ and define
$$
R_x = \sup \Big\{ r \ : \ B_r \subset \overline{K} \ \ \text{is a ball of radius $r$ with } \, x \in \partial B_r \Big\}.
$$
Suppose that $R_x>0$. Fix $\tau \in (0, \infty]$ and let  $\kappa \in \R^+_0$ such that
	\[
	\Ricc \ge -(m-1)\kappa^2 \qquad \text{on } \ \ B_{\tau}(K)= \{ y : \di(y,K) < \tau\}.
	\]
Then,
\begin{equation}\label{boundarygradient}
|\nabla \log u(x)| \le \big| (\log \sgr^{\kappa}_{R_x+\tau})'(R_x)\big|.
\end{equation}
\end{proposition}

\begin{proof}
By continuity, for $R=R_x$ there is a ball $B_R(y) \subset \overline{K}$ that is tangent to $\partial K$ at $x$. Let $r = \dist(y,\cdot)$ and on $B_{R+\tau}(y)\subset B_\tau(K)$ we use as barrier the rescaled kernel
$$
w(x) = \frac{\sgr^\kappa_{R+\tau}(r(x))}{\sgr^\kappa_{R+\tau}(R)}.
$$
Then, $w$ is radially decreasing and $w(x)=1$, $w \le 1$ on $B_{R+\tau}(y) \backslash K$, $w=0$ on $\partial B_{R+\tau}(y)$. By the Laplacian comparison theorem from above, $\Delta_p w \ge 0$ weakly on the punctured ball $B_{R+\tau}(y)^*$, and extending $w$ with zero outside of $B_{R+\tau}$ we still obtain a subsolution. By comparison, $w \le u \le 1$ on $M \backslash K$, thus $\log w \le \log u \le 0$ in a neighbourhood of $x$. Since equality holds at $x$, evaluating along curves lying in $M\backslash K$ issuing from $x$  and taking derivatives yields
$$
|\nabla \log u(x)| \le |\nabla \log w(x)| = \big|( \log \sgr^\kappa_{R+\tau})'(R)\big|,
$$
as claimed.
\end{proof}

Combining Theorem \ref{teo_bellagradiente} and Proposition \ref{prop_boundarygradient}, we deduce

\begin{theorem}\label{teo_bellagradiente_global}
Let $M^m$ be a complete manifold satisfying
$$
\Ricc \ge -(m-1) \kappa^2,
$$
for some constant $\kappa \ge 0$. For $p \in (1,\infty)$, Let $u>0$ be the $p$-capacity potential of a capacitor $(K,M)$, where $K$ is a relatively compact $C^1$ open set such that the following quantity is positive:
$$
R = \inf_{x \in \partial K} \sup \Big\{ r \ : \ B_r \subset \overline{K} \ \ \text{is a ball with } \, x \in \partial B_r \Big\}.
$$
Then,
$$
|\nabla \log u| \le \max \left\{ \frac{m-1}{p-1}\kappa, \big| (\log \sgr^{\kappa})'(R)\big| \right\} \qquad \text{on } \, M.
$$
\end{theorem}

\begin{remark}\label{rem_compacritical}
\emph{It is not hard to estimate the right hand side of \eqref{boundarygradient} with simpler functions. If $\Ricc \ge 0$ on $M$ (thus, necessarily $p < m$), computing $( \log \sgr^\kappa_{R})'$ for $\kappa=0$ yields
$$
|\nabla \log u| \le \frac{m-p}{p-1} \frac{1}{R} \qquad \text{on } \, M,
$$
which mildly improves the bound in \cite{kotschwarni}. On the other hand, when $\kappa >0$, the integral in $\sgr^\kappa$ is not explicitly computable except when $(m-1)/(p-1)$ is an integer, see Example 5.3 in \cite{bmr5}. Indeed, setting $\alpha = (m-1)/(p-1)$ and rescaling the metric so as to have $\kappa = 1$, computing $\sgr^\kappa$ amounts to integrating $\sinh^{-\alpha} t$. Using parameric hyperbolic coordinates $x = \tanh(t/2)$, this leads to a binomial integral of the type
$$
\int \frac{(1-x^2)^{\alpha-1}}{x^\alpha}
$$
which is computable in terms of elementary functions if and only if $\alpha \in \mathbb{Z}$. However, to obtain an explicit upper estimate for $|(\log \sgr^\kappa)'|$ one can use the following comparison result observed in \cite[Prop. 4.12]{bmr2}:
$$
\text{if } \ \ g/h \ \ \text{ is non-decreasing on $(a,b)$, then } \ \ (\log \sgr^g)' \ge (\log \sgrh)' \ \ \text{ on $(a,b)$.}
$$
If $p<m$, a simpler extimate was already given in \cite{kotschwarni}.
}
\end{remark}

\section{Properness of $\varrho$}\label{sec_proper}

The properness of $\varrho$, that is, the property that $\gr(x) \ra 0$ as $r(x) \ra \infty$, is a nontrivial fact intimately related to the geometry of $M$ at infinity. Conditions for its validity will be given in terms of global Sobolev type inequalities or in terms of volume doubling coupled with weak Poincar\'e inequalities. Since local Sobolev and Poincar\'e constants will often appear in the next sections when taking limits as $p \ra 1$, it is convenient to briefly recall their dependence on the geometry of relatively compact balls of $M$.

\begin{remark}\label{rem_localimportant}
\emph{Let $B_{6R} \Subset M^m$ be a relatively compact geodesic ball, and suppose that
\begin{equation}\label{eq_iporicci_local}
\Ricc \ge -(m-1)\kappa^2 \qquad \text{on } \, B_{6R},
\end{equation}
for some constant $\kappa \ge 0$. As before,  we denote  by $V_\kappa(t)$ the volume of a ball of radius $t$ in the space form of curvature $-\kappa^2$. By the Bishop-Gromov volume comparison and the convexity of $V_\kappa$,
\begin{equation}\label{dou_local}
\frac{|B_{2r}(x)|}{|B_r(x)|} \le \frac{V_\kappa(2r)}{V_\kappa(r)} \le \frac{V_\kappa(2R)}{V_\kappa(R)} \doteq C_{\Dou} \qquad \forall \, B_r(x) \Subset B_{R}.
\end{equation}
Furthermore, by \cite[Thm. 5.6.6]{saloff}, see also \cite[Thm. 1.4.1]{koreschoen}, there exists a constant $c_m$ such that, for each $p \in [1,\infty)$, the following weak $(p,p)$-Poincar\'e inequality holds for balls $B_{2r}(x) \subset B_{2R}$:
\begin{equation}\label{buser}
\disp \left\{ \fint_{B_r(x)} |\psi - \bar \psi_{B_r(x)}|^p \right\}^{\frac{1}{p}} \le \disp r \Po_{p,p} \left\{\fint_{B_{2r}(x)} |\nabla \psi|^p \right\}^{\frac{1}{p}} \qquad \forall \, \psi \in \lip(B_{2R}),
\end{equation}
where $\bar \psi_{B_r(x)}$ is the mean value of $\psi$ on $B_r(x)$ and where we set
\begin{equation}\label{poincare_constant}
\Po_{p,p} = \exp\left\{ \frac{c_m(1+ \kappa R)}{p}\right\}.
\end{equation}
Note that the result in \cite{saloff} is proved for $p=1$, and a minor modification using Jensen's inequality yields \eqref{buser} with the stated constant for every $p$. As a consequence of \eqref{buser} with $p=1$ and of Theorem 3.3.5 in \cite{saloff}, there exists $C_m>0$ such that the following $L^1$ Sobolev inequality with potential holds:
\begin{equation}\label{eq_soboconpote}
\left(\int |\psi|^{\frac{m}{m-1}} \right)^{\frac{m-1}{m}} \le \left[ \frac{C_{m}}{\Upsilon}\right] \left( \Po_{1,1} \int |\nabla \psi| + \frac{1}{R} \int |\psi| \right) \qquad \forall \, \psi \in C^\infty_c(B_R),
\end{equation}
where
$$
\Upsilon = \inf \left\{ \frac{|B_t(x)|}{V_0(t)} \ \ : \ \ t \in (0, R), \  x \in B_R\right\}.
$$
Observe that $\Upsilon$ can be estimated from below in terms of $|B_R|$ by using \eqref{eq_iporicci_local} and volume comparison:
\begin{equation}\label{upsilon_Upsilon}
\forall \, x \in B_R, \ t \in (0,R), \quad \frac{|B_t(x)|}{V_0(t)} \ge \frac{|B_{2R}(x)|}{V_\kappa(2R)}\frac{V_\kappa(t)}{V_0(t)} \ge \frac{|B_R|}{V_\kappa(2R)}.
\end{equation}
We next remove the potential part in \eqref{eq_soboconpote} with a slight variation of an argument in \cite[Cor. 1.1]{lischoen}: pick $y \in \partial B_{2R}$ and, setting $r_y = \dist(y,\cdot)$, consider the function
$$
\zeta(x) = \int^{3R}_{r_y(x)} \frac{V_\kappa(3R)- V_\kappa(t)}{v_\kappa(t)}\di t,
$$
with $v_\kappa(t), V_\kappa(t)$ the volume of geodesic spheres and balls, respectively, in the model of curvature $-\kappa^2$. A computation that uses the Laplacian comparison theorem and \eqref{eq_iporicci_local} gives $\Delta \zeta \ge 1$ weakly on $B_R$, thus for every $\psi \in \lip_c(B_R)$ we obtain
$$
\disp \int |\psi| \le \disp \int |\psi|\Delta \zeta = \int \langle \nabla \zeta, \nabla |\psi| \rangle \le \sup_{B_R}|\nabla \zeta| \int |\nabla \psi| \leq \disp \frac{V_\kappa(3R)}{v_\kappa(R)} \int |\nabla \psi|.
$$
Inserting into \eqref{eq_soboconpote} and using \eqref{upsilon_Upsilon} we infer the local $L^1$ Sobolev inequality
\begin{equation}\label{eq_sobosempote}
\begin{array}{lcl}
\disp \left(\int |\psi|^{\frac{m}{m-1}} \right)^{\frac{m-1}{m}} & \le & \disp \frac{C_{m}V_\kappa(2R)}{|B_R|} \left( \Po_{1,1} + \frac{V_\kappa(3R)}{R v_\kappa(R)} \right) \int |\nabla \psi| \\[0.5cm]
& = & \disp \So_{1,m}(R) \int |\nabla \psi| \qquad \forall \, \psi \in \lip_c(B_R)
\end{array}
\end{equation}
%
%
For each $p \in (1,m)$, inserting as a test function $|\psi|^{p\frac{m-1}{m-p}}$ and using H\"older inequality one readily deduces the $L^p$ Sobolev inequality
\begin{equation}\label{SobolevL1_local}
\left( \int |\psi|^{\frac{mp}{m-p}} \right)^{\frac{m-p}{m}} \le \So_{p,m}(R) \int |\nabla \psi|^p \qquad \forall \, \psi \in \lip_c(B_{R})
\end{equation}
where
$$
\So_{p,m}(R) = \left[ \frac{\So_{1,m}(R) p(m-1)}{m-p}\right]^p
$$
converges to $\So_{1,m}(R)$ as $p \ra 1$.
}
\end{remark}

\subsection{Properness under a Sobolev inequality}

We first examine the vanishing of $\gr$ at infinity under the validity of a global Sobolev inequality of the type
$$
\left( \int |\psi|^{ \frac{\nu p}{\nu-p}}\right)^{\frac{\nu-p}{\nu}} \le \So_{p,\nu} \int |\nabla \psi|^p \qquad \forall \, \psi \in \lip_c(M).
$$

To do so, we shall employ the De Giorgi-Nash-Moser iteration technique to obtain a uniform upper bound for $\gr$ on the entire $M$. For applications to the IMCF, it is important to keep track of the dependence on $p$ of the half-Harnack inequalities for positive subsolutions-supersolutions to $\Delta_p u =0$. This is done in \cite{moser_3}, and we present here a slightly different approach that yields a more explicit dependence of the constants on the geometry. As in \cite{moser_3}, we need to tweak the iteration to achieve bounds that behave nicely as $p \ra 1$. We begin with the following standard Caccioppoli Lemma, which can be found in \cite{rigolisalvatorivignati}.

\begin{lemma}\label{lem_caccio}
Let $A_0 \subset M$ be an open set, fix $p \in (1, \infty)$ and let $u \in C(A_0) \cap W^{1,p}_\loc(A_0)$ be non-negative on $A_0$. Fix $\bar q \in \R$. If either
\begin{itemize}
\item[(i)] $\Delta_p u \ge 0$ on $A_0$ and $\bar q > p-1$, or
\item[(ii)] $\Delta_p u \le 0$ on $A_0$ and $\bar q < p-1$,
\end{itemize}
then, for every $0 \le \phi \in C(A_0) \cap W^{1,p}_0(A_0)$,
\begin{equation}\label{caccioppoli}
\int \phi^p u^{\bar q-p}|\nabla u|^p \le \left| \frac{p}{\bar q-p+1}\right|^p \int u^{\bar q} |\nabla \phi|^p.
\end{equation}
\end{lemma}

We next consider the half-Harnack inequalities. We point out that Claim 1 below is what allows to obtain constants which are controlled as $p \ra 1$.

\begin{lemma}\label{lem_moser}
Let $A_\infty \Subset M$ and fix $T>0$ in such a way that $A_0 = B_T(A_\infty) \Subset M$. Suppose that the following Sobolev inequality holds:
\begin{equation}\label{sobolev_permoser}
\left( \int |\psi|^{ \frac{\nu p}{\nu-p}}\right)^{\frac{\nu-p}{\nu}} \le \So_{p,\nu} \int |\nabla \psi|^p \qquad \forall \, \psi \in \lip_c(A_0),
\end{equation}
for some $p \in (1,\infty)$, $\nu>p$ and $\So_{p,\nu}>0$.
\begin{itemize}
\item[(i)] \textbf{Subsolutions}. Fix $q > 0$. Then, there exists a constant $q_0$ with
\begin{equation}\label{ipo_subs}
0 < q_0 \le q < \frac{\nu q_0}{\nu-p}
\end{equation}
such that the following holds: if $u>0$ solves $\Delta_p u = 0$ on $A_0$, then
\begin{equation}\label{halfhar_sub}
\sup_{A_\infty} u \le \disp  (\So_{p,\nu} \bar C_{p,\nu})^{\frac{\nu}{p q_0}} T^{-\frac{\nu}{q_0}} |A_0|^{\frac{1}{q_0}} \left(\fint_{A_0}  u^{q}\right)^{\frac{1}{q}},
\end{equation}
with
\begin{equation}\label{Cpnu_mag0}
 \bar C_{p,\nu} = \left\{ \begin{array}{ll}
2^{\nu}3^p \frac{\nu^\nu}{p^p(\nu-p)^{\nu-p}} & \quad \text{if } \, q \in (0,p), \\[0.5cm]
2^{\nu} \left[ 1+ p\right]^p & \quad \text{if } \, q \ge p.
\end{array} \right.
\end{equation}
If $q \ge p$, then we can choose $q_0=q$ and \eqref{halfhar_sub} holds for weak solutions $0 \le u\in C(A_0) \cap W^{1,p}_\loc(A_0)$ to $\Delta_p u \ge 0$. 

\item[(ii)] \textbf{Supersolutions}. Fix $ q<0  $ and set $q_0=q$. If $0 < u \in C(A_0) \cap W^{1,p}_\loc(A_0)$ solves $\Delta_p u \le 0$, then
\begin{equation}\label{halfhar_super}
\inf_{A_\infty} u \ge \disp  (\So_{p,\nu} \bar C_{p,\nu})^{\frac{\nu}{p q_0}} T^{-\frac{\nu}{q_0}} |A_0|^{\frac{1}{q_0}} \left(\fint_{A_0}  u^{q}\right)^{\frac{1}{q}},
\end{equation}
with
\begin{equation}\label{Cpnu_min0}
 \bar C_{p,\nu} = 2^{p+\nu}.
\end{equation}
\end{itemize}
\end{lemma}


\begin{proof}
Set for convenience
$$
k = \frac{\nu}{\nu-p}.
$$
Let $0 \le \phi \in \lip_c(A_0)$. For a given $\bar q$ for which Lemma \ref{lem_caccio} is applicable (hereafter, an admissible $\bar q$), using \eqref{sobolev_permoser} with $\psi=\phi u^{\frac{\bar q}{p}}$ and \eqref{caccioppoli}, we compute
\begin{equation}\label{firstmoser}
\begin{array}{lcl}
\disp \So_{p, \nu}^{-\frac{1}{p}}\big\|\phi u^{\frac{\bar q}{p}}\big\|_{kp} & \le & \disp  \big\|\nabla \big(\phi u^{\frac{\bar q}{p}}\big)\big\|_p \le \big\|u^{\frac{\bar q}{p}}|\nabla \phi| + |\bar q/p| \phi u^{\frac{\bar q-p}{p}} |\nabla u|\big\|_p \\[0.4cm]
& \le & \disp \big\|u^{\frac{\bar q}{p}}|\nabla \phi|\big\|_p + \left|\frac{\bar q}{p}\right| \big\| \phi u^{\frac{\bar q-p}{p}} |\nabla u|\big\|_p \\[0.4cm]
& \le & \disp \left[1 + \left|\frac{\bar q}{\bar q - p+1}\right|\right] \big\|u^{\frac{\bar q}{p}}|\nabla \phi|\big\|_p.
\end{array}
\end{equation}
%
%
This inequality holds in both cases (i) and (ii).\\[0.2cm]

\noindent \textbf{Claim 1:} the following holds:
\begin{itemize}
\item[(i)] If $q>0$, there exists $q_0 \in \left(\frac{q}{k}, q\right]$ such that
$$
\big|q_0 k^i -p+1\big| \ge \frac{(k-1)q}{2k} \qquad \text{for each } \, i = 0,1,2,\ldots
$$
Moreover, we can choose $q_0=q$ in case $p-1 \le q/k$.
\item[(ii)] If $q < 0$, choosing $q_0=q$ we have $\big|q_0 k^i -p+1\big| \ge |q|$ for each $ \, i = 0,1,2,\ldots$.
\end{itemize}

\begin{proof}[Proof of Claim 1] Case (ii) is obvious, so we focus on case (i). Set for convenience $a = q/k$. Suppose first that $p-1\le a$. Choosing $q_0 = ka$ we deduce
$$
|q_0k^i - p+1| \ge q_0-p+1 \ge ka-a \ge \frac{(k-1)a}{2} \qquad \text{for each } \, i = 0,1,2, \ldots
$$
and the required conclusion is proved. Otherwise, let $j \in \{1,2,\ldots\}$ be such that $p-1 \in I_j = (k^{j-1}a, k^ja]$. We choose $q_0 \in (a, ka]$ in such a way that $q_0k^{j-1}$ is the point in $I_j$ whose distance to $p-1$ is half of the length of $I_j$, so
\begin{equation}
\label{qoded}
| q_0k^{j-1} -p+1| = \frac{k^ja-k^{j-1}a}{2} \ge \frac{(k-1)a}{2}.
\end{equation}
If $p-1$ is strictly bigger that the middle point of $I_j$, then $q_0 k^{j-1}=(p-1)-(k^j-k^{j-1})a/2 < p-1$ and thus, for each $i< j-1$, $|q_0k^i -p+1| \ge |q_0k^{j-1}-p+1| \ge (k-1)a/2$. On the other hand, $q_0k^j > p-1$ and for each $i \ge j$
$$
\begin{array}{lcl}
q_0k^i -p+1 & \ge & \disp q_0k^j - p + 1 = q_0k^j-[q_0k^{j-1}+ (k^j-k^{j-1})a/2] \\[0.4cm]
& = & \disp (k^j-k^{j-1}) \left[ q_0 - \frac{a}{2}\right] \ge (k^j-k^{j-1}) \frac{a}{2} \ge \frac{(k-1)a}{2}.
\end{array}
$$
Similarly, if $p-1$ is not bigger than the middle point of $I_j$, then $q_0 k^{j-1}+(k^j-k^{j-1})a/2 > p-1$ and thus $|q_0k^i -p+1| \ge |q_0k^{j-1}-p+1| \ge (k-1)a/2$ for each $i> j-1$. If $j=1$ we are done, otherwise $q_0k^{j-2} < p-1$ and therefore, taking into account that
	\[
	k^{j-1}q_0\geq (k^j+k^{j-1})a/2 \qquad \text{and} \qquad k^{j-1}(q_0-ka/2)\geq k^{j-1}a/2,
	\]
for each $i \le j-2$ we obtain
$$
\begin{array}{lcl}
p-1-q_0k^i & \ge & \disp p-1- q_0k^{j-2} = [k^{j-1}q_0 -(k^{j}-k^{j-1})a/2] - k^{j-2}q_0 \\[0.4cm]
& = & \disp (k^{j-1}-k^{j-2}) \left[q_0 - \frac{ka}{2}\right] \ge (k^{j-1}-k^{j-2}) \frac{a}{2} \ge \frac{(k-1)a}{2}.
\end{array}
$$
This concludes the proof of the claim.
\end{proof}

Define
$$
r_i = T \left(2-\sum_{j=0}^i2^{-j}\right), \quad A_i = B_{r_i}(A_\infty), \quad \eta_i(t) = \left\{ \begin{array}{ll}
1 & \text{if } \, t \in [0, r_{i+1}) \\[0.2cm]
1 - \frac{2^{i+1}}{T}(t-r_{i+1}) & \text{if } \, t \in [r_{i+1}, r_i) \\[0.2cm]
0 & \text{if } \, t \ge r_{i}
\end{array}\right.
$$
Set $\phi_i = \eta_i(r)$, with $r(x) = \dist(A_\infty, x)$.  Using  $\phi = \phi_i$, $\bar q = q_i = q_0 k^i$ and  $|\nabla \phi_i| \le T^{-1}2^{i+1}$  in \eqref{firstmoser}, and assuming for a moment that each  $q_i$ is admissible, we deduce
\begin{equation}\label{basemoser_sub-0}
\disp \left( \int_{A_{i+1}} u^{q_0k^{i+1}}\right)^{\frac{1}{k}}  \le  \disp \So_{p,\nu} \left[ 1 + \left| \frac{q_0k^i}{q_0k^i-p+1}\right|\right]^p 2^{p(i+1)} T^{-p} \int_{A_i}  u^{q_0 k^i}
\end{equation}
We first consider the case where $0 <q \leq p$. Since $q_i > p-1$ may not hold for small $i$, we shall restrict to solutions to $\Delta_p u = 0$ to ensure that each $q_i$ is admissible. Claim 1 and $|q_0|\le |q|$ imply that
	\begin{equation}\label{bound_fracq1}
1+	\disp \left| \frac{q_0 k^i}{q_0 k^i-p+1} \right| \le \disp 1+  \frac{2q_0 k^{i+1}}{(k-1)q} \le \frac{3k^{i+1}}{(k-1)},
	\end{equation}
which inserted into \eqref{basemoser_sub-0} yields
\begin{equation}\label{basemoser_sub}
\disp \left( \int_{A_{i+1}} u^{q_0k^{i+1}}\right)^{\frac{1}{k}}  \le  \disp \So_{p,\nu} \tilde C_{p,\nu} [2k]^{p(i+1)} T^{-p} \int_{A_i}  u^{q_0 k^i},
\end{equation}
with
\[
\tilde C_{p,\nu} =\left[3\over k-1\right]^p.
\]
On the other hand, if either $q \ge p$ or $q<0$, choosing $q_0 = q$ we can apply \eqref{firstmoser} to solutions of, respectively, $\Delta_p u \ge 0$ (if $q \ge p$) and $\Delta_p u \le 0$ (if $q <0$), since in both cases each $q_i$ is admissible. Noting that $t\to t/(t-p+1)$ is increasing for $ t>p$ we get
	\begin{equation*}
	\left| \frac{q_0 k^i}{q_0 k^i-p+1} \right| \le \left\{ \begin{array}{ll}
	\disp p & \quad \text{if } \, q\ge p, \\[0.4cm]
1 & \quad \text{if } \, q <  0,
	\end{array}\right.
	\end{equation*}
whence \eqref{basemoser_sub-0} gives
\begin{equation}\label{basemoser_sub_2}
	\disp \left( \int_{A_{i+1}} u^{q_0k^{i+1}}\right)^{\frac{1}{k}} \le \disp \So_{p,\nu}\tilde C_p 2^{p(i+1)}  T^{-p} \int_{A_i}  u^{q_0 k^i}
\end{equation}
with
$$
\tilde C_p = \left\{ \begin{array}{ll}
\left[ 1 + p \right]^p & \quad \text{if } \, q \ge p \\[0.3cm]
2^{p} & \quad \text{if } \, q<0.
\end{array}\right.
$$	
If $q \in (0,p)$, taking the $k^i$-th root in \eqref{basemoser_sub}, iterating and explicitly computing the sums, we infer
\begin{equation}\label{itermoser_sub}
\begin{array}{lcl}
\sup_{A_\infty} u^{q_0} & = & \disp \lim_{i \ra \infty} \left(\int_{A_i} u^{q_0k^{i+1}} \right)^{\frac{1}{k^{i+1}}}\\[0.5cm]
&\leq &  \disp (\So_{p,\nu} \tilde C_{p,\nu})^{\sum_{j=0}^\infty k^{-j}}
[2k]^{p \sum_{j=0}^\infty (j+1)k^{-j}} T^{-p \sum_{j=0}^\infty k^{-j}} \int_{A_0}  u^{q_0} \\[0.5cm]
& = & \disp  (\So_{p,\nu}\tilde C_{p,\nu})^{\frac{k}{k-1}}
[2k]^{\frac{pk^2}{(k-1)^2}} T^{-\frac{kp}{k-1}} |A_0| \fint_{A_0}  u^{q_0} \\[0.5cm]
\end{array}
\end{equation}
Taking the $q_0$-th root and applying H\"older inequality with exponents $\frac{q}{q_0}>1$ and $\frac{q}{q-q_0}$  when $q_0<q$, we obtain
\begin{equation}\label{final_sub}
\begin{array}{lcl}
\sup_{A_\infty} u & \le & \disp  (\So_{p,\nu}\tilde C_{p,\nu})^{\frac{k}{(k-1)q_0}}
[2k]^{\frac{pk^2}{(k-1)^2q_0}} T^{-\frac{kp}{(k-1)q_0}} |A_0|^{\frac{1}{q_0}} \left(\fint_{A_0}  u^{q}\right)^{\frac{1}{q}}. \\[0.5cm]
\end{array}
\end{equation}
It is enough to set $\bar C_{p,\nu} = \tilde C_{p,\nu} [2k]^{\frac{pk}{k-1}}$ and use the definition of $k$ to deduce \eqref{halfhar_sub}. The case $q \ge p$ is analogous by using \eqref{basemoser_sub_2}, while, for $q<0$, iterating \eqref{basemoser_sub_2} we get
\begin{equation}
\begin{array}{lcl}
\left(\inf_{A_\infty} u \right)^{-|q_0|} & = & \disp \sup_{A_\infty} u^{q_0} \le \disp  (\So_{p,\nu}\tilde  C_p)^{\frac{k}{k-1}}
2^{\frac{pk^2}{(k-1)^2}} T^{-\frac{kp}{k-1}} |A_0| \fint_{A_0}  u^{q_0} \\[0.5cm]
& \le & \disp (\So_{p,\nu}\tilde C_p)^{\frac{k}{k-1}}
2^{\frac{pk^2}{(k-1)^2}} T^{-\frac{kp}{k-1}} |A_0| \left(\fint_{A_0}  u^{q}\right)^{\frac{q_0}{q}},
\end{array}
\end{equation}
that implies \eqref{halfhar_super} because of our definition of $k$ and $\bar C_{p,\nu}$.
\end{proof}

The above proposition allows to deduce a Harnack inequality with a sharp rate of growth as $p \ra 1$.

\begin{theorem}\label{teo_harnack}
Fix $p \in (1, \infty)$. Let $u$ be a positive solution to  $\Delta_p u=0$ on a ball $B_{6R} = B_{6R}(x_0)$, and suppose that the following weak $(1,p)$-Poincar\'e inequality holds on $B_{4R}$, for some constant $\Po_{1,p}$:
\begin{equation}\label{poinca_p}
\fint_{B_r(y)} |\psi - \bar \psi_{B_r(y)}| \le \Po_{1,p} r \left\{ \fint_{B_{2r}(y)} |\nabla \psi|^p\right\}^{\frac{1}{p}} \qquad \forall \, \psi \in \lip_c(B_{4R}),
\end{equation}
for every ball $B_r(y)\Subset B_{2R}$. Assume the validity of the Sobolev inequality
$$
\left( \int |\psi|^{ \frac{\nu p}{\nu-p}}\right)^{\frac{\nu-p}{\nu}} \le \So_{p,\nu} \int |\nabla \psi|^p \qquad \forall \, \psi \in \lip_c(B_{4R}),
$$
for some $\nu>p$ and $\So_{p,\nu}>0$. Then, having fixed $p_0 \in (p,\nu)$, the following Harnack inequality holds:
$$
\sup_{B_R} u \le \Ha_{p,\nu}^{\frac{1}{p-1}} \inf_{B_R} u,
$$
with constant
\begin{equation}\label{ine_harnack}
\Ha_{p,\nu} = \exp\left\{ c_2 \Po_{1,p} \left[\frac{|B_{6R}|}{|B_{2R}|}\right]^{\frac{1}{p}} Q^{-2} p \right\},
\end{equation}
where $c_2>0$ is a constant depending only on $\nu$ and $p_0$,
$$
Q = \inf_{\tau \in [1, \frac{\nu}{\nu-p}]} \big( \So_{p,\nu} C_{p,\nu} \big)^{- \frac{\nu \tau}{p}} R^{\nu \tau} |B_{2R}|^{-\tau}
$$
and
\begin{equation}\label{Cpnu}
C_{p,\nu} = 2^{\nu} \max \left\{ [1+p]^p, \frac{3^p \nu^\nu}{p^p(\nu-p)^{\nu-p}}\right\}.
\end{equation}
\end{theorem}

\begin{remark}\label{rmk_noname}
\emph{By means of H\"older inequality, the weak $(1,1)$-Poincar\'e inequality \eqref{buser} implies \eqref{poinca_p} with $\Po_{1,p} = \Po_{1,1}$. Therefore, having fixed $\kappa>0$ such that $\Ricc \ge -(m-1)\kappa^2$ on $B_{6R}$, Remark \ref{rem_localimportant} guarantees that $\So_{p,\nu}$ remains bounded as $p \ra 1$, so the Harnack constant in \eqref{ine_harnack} blows up just like $\exp\{c/(p-1)\}$ as $p \ra 1$, which is sharp.
}
\end{remark}

\begin{proof}
We use the abstract version of the John-Nirenberg inequality due to Bombieri and Giusti \cite[Thm. 4]{bombierigiusti}. To this aim, we shall estimate from above
$$
\begin{array}{lcl}
\Lambda & \doteq & \disp \sup_{r \in [R,2R]} \inf_{\lambda >0} \left\{ \fint_{B_r} \left| \log \frac{u}{\lambda}\right| \right\} \le \sup_{r \in [R,2R]} \left\{ \fint_{B_r} |\log u - \overline{\log u}_{B_r}| \right\} \\[0.5cm]
& \le & \disp \sup_{r \in [R,2R]} \Po_{1,p} r \left[ \fint_{B_{2r}} |\nabla \log u|^p \right]^{\frac{1}{p}},
\end{array}
$$
with $\overline{\log u}_{B_r}$ the mean value of $\log u$ on $B_r$. Applying the Caccioppoli inequality \eqref{caccioppoli} with $\bar q = 0$ we deduce that for every $\phi\in C(B_{6R})\cap W_0^{1,p}(B_{6R})$
$$
\int \phi^p |\nabla \log u|^p \le \left(\frac{p}{p-1}\right)^p \int |\nabla \phi|^p,
$$
and in particular, if $\phi$ is a piecewise linear cut-off of $B_{2r}$ inside  $B_{3r}$,
$$
\int_{B_{2r}} |\nabla \log u|^p \le \left(\frac{p}{(p-1)r}\right)^p |B_{3r}|,
$$
which implies
\begin{equation}\label{esti_Lambda}
\Lambda \le \sup_{r \in [R,2R]} \Po_{1,p} r \frac{p}{(p-1)r} \left(\frac{|B_{3r}|}{|B_{2r}|}\right)^{\frac{1}{p}} \le \Po_{1,p} \frac{p}{p-1} \left(\frac{|B_{6R}|}{|B_{2R}|}\right)^{\frac{1}{p}}.
\end{equation}
To agree with the notation in \cite{bombierigiusti}, for $t \in [0,1]$ and $k \in \R\backslash \{0\}$ we define
$$
\BBB_t = B_{(t+1)R}, \qquad |u|_{k, t} = \left\{ \fint_{\BBB_t} u^{k} \right\}^{\frac{1}{k}}, \quad |u|_{+\infty, t} = \sup_{\BBB_t} u, \quad |u|_{-\infty, t} = \inf_{\BBB_t} u.
$$
For $q>0$ and $q_0$, $\nu$ as in Lemma~\ref{lem_moser}, set $\bar \sigma = \nu q/q_0$. Inequalities \eqref{halfhar_sub} and \eqref{halfhar_super} imply that, for each $0 \le s < r \le 1$,
\[
|u|_{\infty, s} \le \disp \left\{ \bar Q(r-s)^{\bar \sigma}\right\}^{-1/q} |u|_{q,r}, \qquad |u|_{-\infty, s} \ge \disp \left\{ \bar Q(r-s)^{\bar \sigma}\right\}^{1/q} |u|_{-q,r}
\]
where
$$
\bar Q = \big( \So_{p,\nu} \bar C_{p,\nu} \big)^{- \frac{\nu q}{pq_0}} R^{\frac{\nu q}{q_0}} |\BBB_r|^{-\frac{q}{q_0}}
$$
and $\bar C_{p,\nu}$ is, according to the value of $q$, any of the constants in \eqref{Cpnu_mag0}, \eqref{Cpnu_min0}. Taking into account that $C_{p, \nu}$ in \eqref{Cpnu} is the maximum of those constants, minimizing $\bar Q$ over all choices of $q_0 \in \left( \frac{\nu-p}{\nu} q, q\right]$ we obtain
$$
\bar Q \ge \inf_{\tau \in [1, \frac{\nu}{\nu-p}]} \big( \So_{p,\nu} C_{p,\nu} \big)^{- \frac{\nu \tau}{p}} R^{\nu \tau} |\BBB_1|^{-\tau} \equiv Q.
$$
Furthermore, since $(r-s) \le 1$, we can replace $\bar \sigma$ with the larger $\sigma = \nu^2/(\nu-p_0)$. Hence, for every $q>0$,
$$
\begin{array}{lcl}
|u|_{\infty, s} & \le & \disp \left\{ Q(r-s)^{\sigma}\right\}^{-1/q} |u|_{q,r} \\[0.3cm]
|u|_{-\infty, s} & \ge & \disp \left\{ Q(r-s)^{\sigma}\right\}^{1/q} |u|_{-q,r}.
\end{array}
$$
We are now in the position to apply Theorem 4 in \cite{bombierigiusti} to deduce the existence of a constant $c_2$ just depending on $\sigma$ (hence, on $\nu,p_0$) such that
$$
\sup_{\BBB_0} u \le \exp \left\{ \frac{c_2 \Lambda}{Q^{2}} \right\} \inf_{\BBB_0} u.
$$
In view of \eqref{esti_Lambda}, this concludes the proof.
\end{proof}

We first apply the half-Harnack inequality to give a sharp upper bound for the Green kernel $\gr$ on an open set $\Omega \subset M$, provided that a weighted Sobolev inequality holds on the entire $\Omega$. Let $\eta$ be a function satisfying
\begin{equation}\label{def_eta_sobolev}
\disp \eta \in C(\R^+_0), \qquad \eta > 0, \qquad \eta(t) \ \ \text{ non-decreasing on } \, \R^+_0, \end{equation}
Let $o \in M$ and set $r(x) = \dist(x,o)$. We will be interested in weighted Sobolev inequalities of the type
\begin{equation}\label{sobolev_weighted}
\left( \int \eta(r)^{-\frac{p}{\nu-p}} |\psi|^{ \frac{\nu p}{\nu-p}}\right)^{\frac{\nu-p}{\nu}} \le \So_{p,\nu} \int |\nabla \psi|^p \qquad \forall \, \psi \in \lip_c(\Omega),
\end{equation}
for sets $\Omega \subset M$. Various classical examples of manifolds satisfying \eqref{sobolev_weighted} in the unweighted case $\eta=1$ will be discussed below. To our knowledge, Sobolev inequalities with a nontrivial weight were first investigated on manifolds with non-negative Ricci curvature by V. Minerbe \cite{minerbe}, see also generalizations in \cite{hein,tewo2}. We will focus on them in Section \ref{subsec_doupoinc}.

\begin{theorem}\label{teo_sobolev}
Let $\Omega \subset M$ be a connected open set, and denote by $r$ the distance from a fixed origin $o \in \Omega$. Assume that $\Omega$ supports the weighted Sobolev inequality \eqref{sobolev_weighted} for some $p \in (1,\nu)$, constant $\So_{p,\nu}>0$ and weight $\eta$ satisfying \eqref{def_eta_sobolev}. Then, $\Delta_p$ is non-parabolic on $\Omega$ and, letting $\gr(x)$ be the Green kernel of $\Delta_p$ on $\Omega$ with pole at $o$,
\begin{equation}\label{upper_consobolev}
\gr(x) \le C_{p,\nu}^{\frac{1}{p-1}} \eta\big(2r(x)\big)^{\frac{1}{p-1}} r(x)^{- \frac{\nu-p}{p-1}}, \qquad \forall \, x \in  \Omega \backslash \{o\},
\end{equation}
where
$$
C_{p,\nu} = \So_{p,\nu}^{\frac{\nu}{p}} \left[ 2^{\nu}p(1+p)^p \left(\frac{p}{p-1}\right)^{p-1}\right]^{\frac{\nu-p}{p}}
$$
is bounded as $p \ra 1$ if so is $\So_{p,\nu}$. In particular, if
\begin{equation}\label{ipo_eta_2}
\eta(t) = o \left( t^{\nu - p} \right) \qquad \text{as } \, t \ra \infty.
\end{equation}
then
 $\gr(x) \ra 0$ as $r(x) \ra \infty$ in $\Omega$.
\end{theorem}

\begin{remark}
\emph{For $p=2$ and $\eta \equiv 1$, Theorem \ref{teo_sobolev} was obtained in \cite{ni} by integrating the corresponding decay estimate for the heat kernel.
}
\end{remark}



\begin{proof}
The non-parabolicity of $\Delta_p$ on $\Omega$ is a standard consequence of \eqref{sobolev_weighted}. Indeed, by contradiction, assume that there exists a relatively compact open set $K \Subset \Omega$ with $\capac_p(K,\Omega) = 0$. Evaluating \eqref{sobolev_weighted} along a sequence $\{\psi_j\} \subset \lip_c(\Omega)$ with $\psi_j = 1$ on $K$ and $\|\nabla \psi_j\|_{L^p(\Omega)} \to 0$, and letting $j \to \infty$, we contradict the positivity of $\eta$ on $K$.
To show \eqref{upper_consobolev}, without loss of generality we can assume that $\gr$ is the Green kernel of a smooth, connected, relatively compact  domain $\Omega' \Subset \Omega$, the thesis then follows by taking an exhaustion $\{\Omega_j\} \uparrow \Omega$ and the limit of the corresponding kernels.
%
%
%
Let $x\in \Omega$ and let $\theta \in (0,1)$ be a constant to be chosen later.
Choose $\ell$ satisfying
\begin{equation}\label{def_lj}
\sup_{\Omega' \cap \partial B_{(1-\theta)r(x)}} \gr < \ell
\end{equation}
and such that \eqref{ide_Gr}, \eqref{ide_Gr_2} hold with, respectively, $\psi \equiv 1$ and $\psi = \gr$. In particular,
\begin{equation}\label{inte_byparts}
\int_{\{\gr \le \ell\}} |\nabla \gr|^p = \ell \int_{\{\gr = \ell\}}|\nabla \gr|^{p-1} = \ell.
\end{equation}
Since $\gr$ is $p$-harmonic on $\Omega'\setminus\{o\}$ and vanishes on $\partial \Omega'$, the maximum principle implies that its super-level sets are connected, thus
	\[
	\{\gr > \ell\} \subset B_{(1-\theta)r(x)}.
	\]
Again by the maximum principle, $\sup_{\partial  B_r\cap \Omega'}\gr$ is a non-increasing function of $r$. Extend $\gr$ with zero on $M \backslash \Omega'$, and observe that $\Delta_p \gr \ge -\delta_o$ on $M$. To apply Lemma \ref{lem_moser} to $\gr$ with $t = r(x)$ and the choices $A_\infty = B_{t(2-\theta)}\backslash B_t$, $T = \theta t$, $q = \nu p/(\nu-p)$, notice that \eqref{sobolev_weighted} restricted to $B_{2t} \backslash B_{(1-\theta)t}$ together with our assumptions \eqref{def_eta_sobolev} on $\eta$ implies the unweighted Sobolev inequality
\begin{equation}\label{sobolev_unweighted}
\left( \int |\psi|^{ \frac{\nu p}{\nu-p}}\right)^{\frac{\nu-p}{\nu}} \le \So_{p,\nu} \eta(2t)^{\frac{p}{\nu}} \int |\nabla \psi|^p \qquad \forall \, \psi \in \lip_c\big(\Omega \cap B_{2t} \backslash B_{(1-\theta)t}\big).
\end{equation}
Although $B_{2t} \backslash B_{(1-\theta)t}$ may not entirely lie in the set where \eqref{sobolev_unweighted} holds, we can still apply Lemma \ref{lem_moser} since $\gr$, and therefore the test functions in the Moser iteration, vanish outside of $\Omega$; hence, we obtain
\begin{equation}\label{final_moser_2}
\begin{array}{lcl}
\disp \|\gr\|_{L^{\infty}(\partial B_t)} & \le & \disp \disp \big(\So_{p,\nu} \bar C_{p, \nu}\big)^{\frac{\nu-p}{p^2}} \eta(2t)^{\frac{\nu-p}{\nu p}} \left( \theta t\right)^{- \frac{\nu-p}{p}} \left( \int_{B_{2t} \backslash B_{(1-\theta)t}} \gr^{\frac{\nu p}{\nu-p}}\right)^{\frac{\nu-p}{\nu p}} \\[0.5cm]
& \le & \disp \disp \big(\So_{p,\nu} \bar C_{p, \nu}\big)^{\frac{\nu-p}{p^2}} \eta(2t)^{\frac{1}{p}} \left( \theta t\right)^{- \frac{\nu-p}{p}} \left( \int_{B_{2t} \backslash B_{(1-\theta)t}} \eta(r)^{-\frac{p}{\nu-p}} \gr^{\frac{\nu p}{\nu-p}}\right)^{\frac{\nu-p}{\nu p}},
\end{array}
\end{equation}
where in the last inequality we used the monotonicity of $\eta$ in \eqref{def_eta_sobolev}, and where
	\[
	\bar C_{p, \nu} = 2^{\nu}[1+p]^p.
	\]
Plugging in the Sobolev inequality \eqref{sobolev_unweighted} the test function $\psi = \min\{\gr, \ell\} \in \lip_c(\Omega)$, and using the fact that $\psi = \gr$ on $\Omega \backslash B_{(1-\theta)t}\subset \{\gr \le \ell\}$ together with \eqref{inte_byparts}, we get
\begin{equation}\label{boundLp}
\begin{array}{lcl}
\disp \left( \int_{B_{2t} \backslash B_{(1-\theta)t}} \eta(r)^{-\frac{p}{\nu-p}} \gr^{ \frac{\nu p}{\nu-p}}\right)^{\frac{\nu-p}{\nu}} & \le & \disp \left( \int \eta(r)^{-\frac{p}{\nu-p}} \psi^{ \frac{\nu p}{\nu-p}}\right)^{\frac{\nu-p}{\nu}} \le \So_{p,\nu} \int |\nabla \psi|^p \\[0.4cm]
& = & \disp \So_{p,\nu}\int_{\{\gr \le \ell\}} |\nabla \gr|^p = \So_{p,\nu} \ell.
\end{array}
\end{equation}
If we denote by $\|\cdot\|_s$ the $L^\infty$ norm on $\partial B_s$, inserting into \eqref{final_moser_2} and letting $\ell \to \|\gr\|_{(1-\theta)t}$ we infer
\begin{equation}\label{sobolev_iteriamola_00}
\|\gr\|_t \le \So_{p,\nu}^{\frac{\nu}{p^2}} \bar C_{p, \nu}^{\frac{\nu-p}{p^2}} \eta\big(2r(x)\big)^{\frac{1}{p}} t^{- \frac{\nu-p}{p}} \theta^{- \frac{\nu-p}{p}} \|\gr\|_{(1-\theta)t}^{\frac{1}{p}}.
\end{equation}
Fix $\xi \in (0,1)$ and consider a sequence $\{\sigma_k\}_{k\ge 0} \subset [1,\infty)$ with the property that
	\begin{equation}\label{mono_pk}
	\sigma_{k+1} > \sigma_k \qquad \text{for } \, k \ge 0,
	\end{equation}
to be specified later, and construct inductively sequences $\{t_k\}, \{\theta_k\}$ for $k \ge 0$ as follows:
	\[
	\begin{array}{c}
	t_0 = t, \qquad \theta_0 = 1 - \xi^{\sigma_1}, \qquad  \theta_k = 1- \xi^{\sigma_{k+1} - \sigma_k} \quad \text{for } \, k \ge 1, \\[0.3cm]
	\quad t_{k+1} = (1-\theta_k)t_k = t \xi^{\sigma_{k+1}}.
	\end{array}
	\]
Set for convenience
	\begin{equation}\label{def_Chat}
	\hat C = \So_{p,\nu}^{\frac{\nu}{p^2}} \bar C_{p, \nu}^{\frac{\nu-p}{p^2}} \eta\big(2r(x)\big)^{\frac{1}{p}}.
	\end{equation}
We iterate \eqref{sobolev_iteriamola_00} $i$-times for the chosen $\theta_k,t_k$ and use that $\eta$ is increasing (so we can use $\hat C$ at every step of the iteration) to deduce
	\begin{equation}\label{sobolev_iteriamola_2}
\begin{array}{lcl}
\disp \|\gr\|_{t_0} & \le & \disp \hat C t_0^{- \frac{\nu-p}{p}} \theta_0^{- \frac{\nu-p}{p}} \|\gr\|_{t_1}^{\frac{1}{p}} \\[0.4cm]
& \le & \disp \hat C^{1 + p^{-1}} \big[t_0\theta_0\big]^{- \frac{\nu-p}{p}} \big[(t_1\theta_1)^{p^{-1}}\big]^{-\frac{\nu-p}{p}} \|\gr\|_{t_2}^{\frac{1}{p^2}} \\[0.4cm]
& \le & \ldots \le \disp \hat C^{\sum_{k=0}^i p^{-k}} \left[ \prod_{k=0}^i (t_k \theta_k)^{p^{-k}}\right]^{-\frac{\nu-p}{p}} \|\gr\|_{t_{i+1}}^{\frac{1}{p^{i+1}}}.
\end{array}
\end{equation}
We shall find a suitable sequence $\{\sigma_k\}$ such that
	\[
	P_1 \doteq \prod_{k=0}^{\infty} (t_k\theta_k)^{p^{-{k}}} = t^{\frac{p}{p-1}} (1-\xi^{\sigma_1}) \prod_{k=1}^\infty \left[ \xi^{\sigma_k} - \xi^{\sigma_{k+1}} \right]^{\frac{1}{p^k}}
	\]
	converges with nice estimates as $p \ra 1$. Taking the logarithm, this amounts to estimating from below the sum
	\[
	\sum_{k=1}^\infty \frac{1}{p^k} \log \left[ \xi^{\sigma_k} - \xi^{\sigma_{k+1}} \right]
	\]
by $\mathcal{O}(\frac{1}{p-1})$. For fixed $\tau>1$, we choose $\sigma_k$ inductively by taking
	\[
	\sigma_1 = 1, \qquad \sigma_{k+1} = \frac{\log\left( \xi^{\sigma_k} - \xi^\tau\right)}{\log \xi},
	\]
	so in particular,
	$$
	\xi^{\sigma_k} - \xi^{\sigma_{k+1}} = \xi^\tau, \qquad \text{hence} \qquad \sigma_{k+1} > \sigma_k > \ldots > \sigma_1 =1.
	$$
Note also that $\sigma_k \in (1, \tau)$ for every $k$, since $\xi^\tau < \xi^{\sigma_k}$ and therefore $t_k=t \xi^{\sigma_{k}}\geq t\xi^\tau$.
With such a choice,
	\[
	\sum_{k=1}^\infty \frac{1}{p^k} \log \left[ \xi^{\sigma_k} - \xi^{\sigma_{k+1}} \right] = \tau\log \xi \sum_{k=1}^\infty \frac{1}{p^k} = \frac{\tau \log \xi}{p-1}
	\]
	and thus 	
	\[
	P_1 = t^{\frac{p}{p-1}}(1-\xi^{\sigma_1}) \exp\left\{ \frac{\tau \log \xi}{p-1} \right\} = t^{\frac{p}{p-1}}(1-\xi) \xi^{\frac{\tau}{p-1}}.
	\]
Recalling that $\|\gr\|_{r}$, is a non-increasing function of $r$, $\|\gr\|_{t_{i+1}}\leq \|\gr\|_{t\xi^\tau}$ and therefore
$$
\lim_{i \ra \infty} \|\gr\|_{t_{i+1}}^{\frac{1}{p^{i+1}}} \le \lim_{i \ra \infty} \|\gr\|_{t\xi^{\tau}}^{\frac{1}{p^{i+1}}} = 1.
$$
Thus, letting $i \ra \infty$ and computing the sum at the exponent of $\hat C$, we deduce from \eqref{sobolev_iteriamola} the upper bound
	\[
\disp \|\gr\|_{t} \le \disp \hat C^{\frac{p}{p-1}} \left[ \prod_{k=0}^\infty (t_k \theta_k)^{p^{-k}}\right]^{-\frac{\nu-p}{p}} = \hat C^{\frac{p}{p-1}} t^{- \frac{\nu-p}{p-1}} \Big[(1-\xi) \xi^{\frac{\tau}{p-1}}\Big]^{-\frac{\nu-p}{p}}.
	\]
Finally, letting $\tau \ra 1$,  maximizing in $\xi \in (0,1)$ to estimate
	\[
	\max_{\xi \in (0,1)} (1-\xi) \xi^{\frac{1}{p-1}} = \frac{p-1}{p} p^{-\frac{1}{p-1}}
	\]
	and recalling the definition of $\hat C$ we eventually obtain
	\begin{equation}\label{sobolev_iteriamola}
\disp \|\gr\|_{t} \le \disp \left(\frac{p}{p-1}\right)^{\frac{\nu-p}{p}} \So_{p,\nu}^{\frac{\nu}{p(p-1)}} \bar C_{p, \nu}^{\frac{\nu-p}{p(p-1)}} \eta(2r(x))^{\frac{1}{p-1}} t^{- \frac{\nu-p}{p-1}} p^{\frac{\nu - p}{p(p-1)}}.
	\end{equation}
	Estimate \eqref{upper_consobolev} then follows from the definition of $\bar C_{p, \nu}$.
\end{proof}	
	
\begin{remark}	
\emph{It should be pointed out that the non-standard iteration carried out  in the above proof allowed to obtain a constant $C_{p,\nu}$ which remains bounded as $p\to 1$. Standard dyadic iterations, or variants thereof, would produce constants which diverge as  $p\to 1$.
}
\end{remark}

We conclude this section by describing a few relevant examples where \eqref{sobolev_weighted} holds on the entire $M$ with $\nu = m$ and no weight:
\begin{equation}\label{sobolev}
\left( \int |\psi|^{ \frac{mp}{m-p}}\right)^{\frac{m-p}{m}} \le \So_{p,m} \int |\nabla \psi|^p \qquad \forall \, \psi \in \lip_c(M),
\end{equation}
and consequently
\begin{equation}\label{decaygreen_examples}
\gr(x) \le C_{p,m}^{\frac{1}{p-1}} r(x)^{- \frac{m-p}{p-1}} \qquad \text{for } \, x \in M \backslash \{o\}.
\end{equation}
We stress that, by \cite[Prop. 2.5]{carron} and \cite{pigolasettitroyanov}, \eqref{sobolev} holds on $M$ possibly with a different constant $\So_{p,m}$ if and only if $M$ has infinite volume and \eqref{sobolev} holds outside some compact set of $M$. However, from the proof in \cite{pigolasettitroyanov} it is unclear whether the boundedness of the Sobolev constant outside of a compact set as $p \ra 1$ implies that of the global Sobolev constant.

\begin{remark}\label{rem_isoperimetric}
\emph{As in the end of Remark \ref{rem_localimportant}, if the $L^1$ Sobolev inequality
\begin{equation}\label{isoperimetric}
\left(\int |\psi|^{\frac{m}{m-1}}\right)^{\frac{m-1}{m}} \le \So_{1,m} \int |\nabla \psi| \qquad \forall \, \psi \in \lip_c(M)
\end{equation}
holds for some $\So_{1,m}>0$, then \eqref{sobolev} holds for every $p \in (1,m)$ with $\nu = m$ and constant
$$
\So_{p,m} = \left[ \frac{\So_{1,m} p(m-1)}{m-p}\right]^p \ra \So_{1,m} \qquad \text{as } \, p \ra 1.
$$
}
\end{remark}

\begin{example}\label{ex_minimal}
\emph{We recall that a Cartan-Hadamard space is a complete, simply-connected manifold with non-positive sectional curvature. Let $M^m \ra N^n$ be a complete, minimal immersion into a Cartan-Hadamard space. By \cite{hoffmanspruck}, the $L^1$ Sobolev inequality \eqref{isoperimetric} holds on $M$, and consequently, for each $p \in (1,m)$ the kernel $\gr$ of $\Delta_p$ satisfies \eqref{decaygreen_examples}.
}
\end{example}

\begin{example}
\emph{If $\Ricc \ge 0$ on $M$ and $m \ge 3$, then
	\[
	\text{$M$ enjoys \eqref{isoperimetric}} \qquad \Longleftrightarrow \qquad \lim_{t \ra \infty} \frac{|B_t|}{V_0(t)} \doteq \Theta >0,
	\]
	that is, $M$ has maximal volume growth. Indeed, referring to Remark \ref{rem_localimportant}, one observes that the constant in \eqref{eq_sobosempote} is uniform in $R$ provided that $\Theta>0$ (this result can also be found in \cite[Thm. 3.3.8]{saloff}), so implication $\Leftarrow$ holds. On the other hand, $\Rightarrow$ holds irrespectively of a bound on the Ricci tensor, see \cite{carron} and \cite[Lem. 7.15]{prs}.
}
\end{example}


\begin{example}\label{prop_hebey}
\emph{Let $M^m$ be a complete manifold satisfying
\begin{itemize}
\item[$(i)$] $\Ricc \ge -(m-1) \kappa^2$ for some $\kappa>0$, and
\begin{equation}\label{ipo_persobolev}
\inf_{x \in M} |B_1(x)| = \upsilon > 0;
\end{equation}
\item[$(ii)$] for some $p \in (1,m)$ and $\Po_p>0$, the Poincar\'e inequality
\begin{equation}\label{poincare}
\int |\psi|^p \le \Po_p \int |\nabla \psi|^p \qquad \forall \, \psi \in \lip_c(M).
\end{equation}
\end{itemize}
By work of N. Varopoulos (see \cite{hebey_book}, Thm. 3.2), because of $(i)$ $M$ enjoys the $L^1$ Sobolev inequality with potential
\begin{equation}\label{sobolevL1}
\left(\int |\psi|^{\frac{m}{m-1}}\right)^{\frac{m-1}{m}} \le \So_{1,m} \int \big[|\nabla \psi| + |\psi|\big] \qquad \forall \, \psi \in \lip_c(M),
\end{equation}
for some $\So_{1,m}$ depending on $(m,\kappa,\upsilon)$. Using again as a test function $|\psi|^{\frac{p(m-1)}{m-p}}$, by H\"older inequality and rearranging we get (see \cite[Lem. 2.1]{hebey_book})
\begin{equation}\label{sobolevLp}
\left(\int |\psi|^{\frac{mp}{m-p}}\right)^{\frac{m-p}{m}} \le \So_{p,m} \int \big[|\nabla \psi|^p + |\psi|^p\big] \qquad \forall \, \psi \in \lip_c(M),
\end{equation}
for some $\So_{p,m}$ depending on $(m,\kappa, \upsilon,p)$. Assumption $(ii)$ then guarantees \eqref{sobolev} with $\nu = m$.
}
\end{example}

\begin{example}\label{ex_roughisometric}
\emph{Two metric spaces $(M,\di_M)$ and $(N,\di_N)$ are said to be \emph{roughly isometric} if there exist $\varphi : M \ra N$ and constants $\eps > 0, C_1 \ge 1, C_2 \ge 0$ such that $B_\eps (\varphi(M)) = N$ and
$$
C_1^{-1} \di_M(x,y) -C_2 \le \di_N  \big( \varphi(x), \varphi(y) \big) \le C_1 \di_M(x,y) + C_2 \qquad \forall \, x,y \in M.
$$
The notion was introduced by M. Kanai, who proved in \cite[Thm. 4.1]{kanai} that if $M$ and $N$ are roughly isometric manifolds of the same dimension, both satisfying the condition
\begin{itemize}
\item[$(iii)$] $\Ricc \ge -(m-1) \kappa^2, \qquad \mathrm{inj}(M) >0$
\end{itemize}
for some constant $\kappa>0$, where $\mathrm{inj}(M)$ denotes the injectivity radius of $M$, then
	\[
	\text{\eqref{isoperimetric} holds on $M$} \qquad \Longleftrightarrow \qquad \text{\eqref{isoperimetric} holds on $N$}.
	\]
In particular, a manifold $M^m$ satisfying $(iii)$ and roughly isometric to $\R^m$ enjoys \eqref{isoperimetric}, and therefore \eqref{sobolev} with $\nu = m$. Under the same assumptions,  $\Delta_p$ is parabolic for each $p \ge m$, see \cite[Thm. 3.16]{holopainen3}, and thus \eqref{sobolev} is false for any $m \le p < \nu$. We remark in passing that $\mathrm{inj}(M)>0$ implies a lower bound for the volume as in \eqref{ipo_persobolev}, see \cite[Prop. 14]{croke}.
}
\end{example}

\begin{question}
\emph{If $\Ricc\geq -(m-1)\kappa^2$ one may wonder under which additional conditions the Green's kernel has an exponential decay of the type
$$
\gr(x) \le C_{p,m} r(x)^{-\frac{m-p}{p-1}} e^{-\lambda r(x)}
$$
for some constants $C_{p,m}, \lambda>0$ depending on $(\kappa,m,p)$.
Work of P. Li and J. Wang, \cite{liwang2}, suggests that this could be the case provided conditions $(i)$ and $(ii)$ in Example~\ref{prop_hebey} hold.
}
\end{question}


\subsection{Properness under volume doubling and Poincar\'e inequalities} \label{subsec_doupoinc}

In this section we investigate in detail the class of manifolds supporting global doubling and weak (Neumann) Poincar\'e inequalities.
\begin{definition}
Let $M^m$ be a complete Riemannian manifold.
\begin{itemize}
\item[$\dous$] we say that $M$ has the \emph{global doubling property} if there exists a constant $C_{\Dou} > 1$ such that
$$
|B_{2r}(x)| \le C_{\Dou} |B_r(x)| \qquad \text{for each } \ \  x \in M, \ \ r >0.
$$
\item[$\neuqp$] given $1 \le q \le p < \infty$, we say that the \emph{weak $(q,p)$-Poincar\'e inequality} holds on $M$ if there exists a constant $\Po_{q,p}$, depending on $(q,p,m)$, such that
\begin{equation}\label{WNP}
\left\{ \fint_{B_r(x)} |\psi - \bar \psi_{B_r(x)}|^q\right\}^{\frac{1}{q}} \le \Po_{q,p} r \left\{ \fint_{B_{2r}(x)} |\nabla \psi|^p \right\}^{\frac{1}{p}}
\end{equation}
for each $x \in M$, $r>0$ and $\psi \in \lip_c(M)$.
%
\end{itemize}
\end{definition}
%
%
\begin{remark}\label{rem_dou}
\emph{In view of Bishop-Gromov volume comparison and \eqref{buser}, both $\dous$ and $\neup$ (for every $p \ge 1$) hold if $\Ricc \ge 0$ on $M$. Furthermore, by \cite[Thm. 7.1 and Prop. 2.3]{CSC1}, if two manifolds $M_1,M_2$ have Ricci curvature bounded from below and are roughly isometric (see Example \ref{ex_roughisometric}) via a map $\varphi$ which also satisfies
	\[
	C^{-1} |B_1(x)| \le |B_1(\varphi(x))| \le C|B_1(x)| \qquad \forall \, x \in M_1
	\]
for some constant $C>1$, then $\dous$, $\neup$ hold on $M_1$ if and only if they hold on $M_2$. 
}
\end{remark}
\begin{remark}
\emph{An application of H\"older inequality shows that $\neuuno \Rightarrow \neuunop$ and $\neup \Rightarrow \neuqp \Rightarrow \neuunop$ for each $1 \le q \le p$.
}
\end{remark}
%
A standard iteration, see \cite[Lem. 8.1.13]{hkst}, shows that $\dous$ implies:
	\begin{equation}\label{rel_low_vol}
	\forall \, B_s' \subset B_t \ \text{balls}, \quad \frac{|B_s'|}{|B_t|} \ge C \left(\frac{s}{t}\right)^\nu \quad \text{with } \, \nu =  \log_2 C_\Dou \text{ and } C = C(C_\Dou).
\end{equation}
The constant $\nu$ is called the doubling dimension of $M$. 
\begin{remark}
\emph{From the asymptotic behavior of $|B_r|$ as $r\to 0$, it follows that $\nu\geq m$.
}
\end{remark}
Under an additional condition first introduced by V. Minerbe, \cite{minerbe}, which is in some sense the reverse of \eqref{rel_low_vol}, manifolds supporting $\dous$ and $\neup$ satisfy a weighted Sobolev inequality. The next theorem was obtained for $p=2$ in \cite{minerbe} and recently extended by D. Tewodrose in \cite{tewo2}. We quote the following simplified version of \cite[Thm 1.1]{tewo2}.
\begin{theorem} 
\label{theorem-tewo2}
Let $M^m$ be a complete manifold satisfying $\dous$ and $\neup$ for some $p\in[1,\nu)$, with doubling dimension $\nu=\log_2 C_\Dou$. Assume that there exist constants $C_{\RD}>0$, $b\in (p, \nu]$ and a point $o \in M$ such that
\begin{equation}\label{eq_reversevol}
\forall \, t \ge s > 0, \qquad \frac{|B_t(o)|}{|B_s(o)|} \ge C_{\RD} \left( \frac{t}{s}\right)^b.
\end{equation}
Then, there exists $\So_{p,\nu}$ depending only on $C_\Dou$, $p$, $\Po_{p,p}$, $b$ and $C_{\RD}$ such that
\begin{equation}
\label{sob_tewodrose}
\left(
\int_M \left[ \frac{r^\nu}{|B_r(o)|}
\right]^{-\frac{p}{\nu-p} } |\psi|^{\frac{\nu p}{\nu-p}}
\right)^{\frac{\nu-p}{\nu}}
\leq
\So_{p,\nu} \int_M |\nabla \psi|^p, \qquad \forall \,\psi\in \lip_c (M).
\end{equation}
\end{theorem}
\begin{remark}
\emph{If $\Ricc \ge - C(1+r)^{-2}$ and satisfies a few further conditions, a weighted Sobolev inequality with polynomial weights can be found in \cite{hein}.
}
\end{remark}
\begin{remark}\label{rem_reversevol}
\emph{Condition \eqref{eq_reversevol} holds, for instance, provided that there exist $\tilde C>1$ and $b \in (p,m]$ such that the balls $B_t$ centered at a fixed origin $o$ satisfy
$$
\tilde C^{-1} t^b \le |B_t| \le \tilde C t^b \qquad \forall \, t \ge 1.
$$
The constant $C_\RD$ then only depends on $\tilde C, m$ and on the lower bound $H$ for the Ricci curvature on, say, $B_{6}$. Indeed, if $1 \le s \le t$ then
\begin{equation}\label{eq_mine_1}
\frac{|B_t|}{|B_s|} \ge \frac{1}{\tilde C^2} \left( \frac{t}{s} \right)^b = C_1 \left( \frac{t}{s} \right)^b.
\end{equation}
On the other hand, if $0 < s \le t \le 1$, Remark \ref{rem_localimportant} guarantees a bound for the $L^1$-Sobolev (isoperimetric) constant $\So$ on $B_1$ in terms of $\max_{[0,6]}H$ and of $|B_1|$, hence of $H,\tilde C$. Integrating the inequality $|\partial B_\sigma| \ge \So^{-1} |B_\sigma|^{\frac{m-1}{m}}$ from $s$ to $t$ and using volume comparison we get
$$
\left( \frac{|B_t|}{|B_s|} \right)^{\frac{1}{m}} \ge 1 + \frac{t-s}{m \So |B_s|^{1/m}} \ge 1 + \frac{t-s}{m \So V_h(s)^{1/m}} \ge 1 + C_2 \left( \frac{t}{s} -1\right) \ge \min\{C_2,1\} \frac{t}{s},
$$
thus raising to the $m$-th power and using $b \le m$,
\begin{equation}\label{eq_mine_2}
\frac{|B_t|}{|B_s|} \ge C_3 \left(\frac{t}{s}\right)^m  \ge C_3 \left(\frac{t}{s}\right)^b.
\end{equation}
The case $t > 1 > s > 0$ follows by combining \eqref{eq_mine_1} and \eqref{eq_mine_2}, respectively with $s=1$ and $t=1$.
}
\end{remark}
If we define
	\begin{equation}\label{def_eta_tewo}
	\eta(t) = \sup_{s \in (0,t]} \frac{s^\nu}{|B_s(o)|},
	\end{equation}
and observe that $\eta$ is finite since $\nu \ge m$, a direct application of Theorems \ref{theorem-tewo2} and  \ref{teo_sobolev} yields a pointwise decay for $\gr$. However, in applications to the IMCF we need to  guarantee that the constant $C_{p,\nu}$ in \eqref{upper_consobolev}, hence $\So_{p,\nu}$, remains bounded on sets $(1,p_0] \subset (1,\nu)$. We shall prove it when $M$ supports $\dous$ and $\neuuno$. Note that, although Theorem \ref{theorem-tewo2} guarantees the validity of \eqref{sob_tewodrose} for $p=1$, it seems not immediate to deduce from it a weighted Sobolev inequality for $p \in (1,\nu)$. We therefore have to keep track of the constants in \cite{tewo2}, a procedure that requires some comments. First, by \cite[Thm. 9.1.15]{hkst} and in view of \eqref{rel_low_vol}, on a complete manifold satisfying $\dous$, $\neuuno$ there exists a constant $\So(C_\Dou,\Po_{1,1})$ such that
	\begin{equation}\label{neusob}
	\left( \fint_{B_r(y)} |\psi - \bar\psi_{B_r(y)}|^{\frac{\nu}{\nu-1}} \right)^{\frac{\nu-1}{\nu}} \le \So r \fint_{B_r(y)} |\nabla \psi| \qquad \forall \, y \in M, \psi \in \lip(B_r(y)),
	\end{equation}
In the next Lemma, we examine the behavior of the $L^p$ Sobolev constant for $p \in (1,\nu)$.
%
%
%
%
\begin{lemma}\label{lem_Lpsob}
Let $M$ be a complete manifold satisfying $\dous$ and $\neuuno$, with doubling dimension $\nu = \log_2 C_\Dou$. Then, for each $p_0 \in (1, \nu)$, there exists a constant $\bar \So$ depending on $C_\Dou, \Po_{1,1}, p_0$ such that, for each $p \in [1,p_0]$,
	\begin{equation}\label{neusob_p}
	\left( \fint_{B_r(y)} |\psi - \bar\psi_{B_r(y)}|^{\frac{\nu p}{\nu-p}} \right)^{\frac{\nu-p}{\nu p}} \le \bar \So r \left(\fint_{B_r(y)} |\nabla \psi|^p\right)^{\frac{1}{p}} \quad \forall \, y \in M, \psi \in \lip(B_r(y)),
	\end{equation}
\end{lemma}
\begin{proof}
Set for convenience $B = B_r(y)$ and $q^* = \frac{\nu q}{\nu - q}$. Define $g = |\psi - \bar \psi_B|^{p^*/1^*}$. Then, applying \eqref{neusob} we deduce
	\begin{equation}\label{eq_bonito_33}
	\begin{array}{lcl}
	\disp \left( \fint_B |\psi - \bar \psi_B|^{p^*}\right)^{\frac{1}{1^*}} & = & \disp \left( \fint_B g^{1^*} \right)^{\frac{1}{1^*}} \le \left( \fint_B |g- \bar g_B|^{1^*} \right)^{\frac{1}{1^*}} + \left( \fint_B \bar g_B^{1^*} \right)^{\frac{1}{1^*}} \\[0.5cm]
	 & \le & \disp \So r \fint_B |\nabla g| + \fint_B g.	
	\end{array}
	\end{equation}
However, by H\"older inequality
	\[
	\fint_B |\nabla g| = \frac{p^*}{1^*}\fint_B |\psi-\bar \psi_B|^{\frac{p^*}{1^*} - 1}|\nabla \psi| \le \frac{p^*}{1^*} \left(\fint_B g^{1^*}\right)^{\frac{p-1}{p}} \left( \fint_B |\nabla \psi|^p \right)^{\frac{1}{p}}.
	\]
We apply Young inequality $ab \le \frac{(\eps a)^{q'}}{q'} + \frac{(b/\eps)^q}{q}$ with the conjugate exponents
	\[
	q = \frac{p^*}{1^*}, \qquad q' = \frac{p(\nu-1)}{\nu(p-1)}
	\]
To obtain
	\[
	\So r \fint_B |\nabla g| \le \So r \frac{\eps^{\frac{p(\nu-1)}{\nu(p-1)}}\nu(p-1)}{\nu - p} \left(\fint_B g^{1^*}\right)^{\frac{1}{1^*}} + \So r \eps^{- \frac{p^*}{1^*}} \left( \fint_B |\nabla \psi|^p \right)^{\frac{p^*}{1^* p}}
	\]
We choose $\eps$ such that
	\[
	\eps^{\frac{p(\nu-1)}{\nu(p-1)}} = \frac{1}{\So r p^*}
	\]	
and plug into \eqref{eq_bonito_33} to deduce
	\begin{equation}\label{eq_buono_4}
\left( \fint_B g^{1^*} \right)^{\frac{1}{1^*}} \le p \left[ \So r \left(\So r p^* \right)^{\frac{\nu(p-1)}{\nu-p}} \left(\fint_B |\nabla \psi|^p \right)^{\frac{p^*}{1^* p}} + \fint_B g \right].
	\end{equation}
If $p^*/1^* \le 1^*$, then we can apply H\"older inequality and \eqref{neusob} to get
	\[
	\begin{array}{lcl}
	\disp \fint_B g & = & \disp \fint_B |\psi - \bar \psi_B|^{\frac{p^*}{1^*}} \le \left( \fint_B |\psi - \bar \psi_B|^{1^*} \right)^{\frac{p^*}{(1^*)^2}} \\[0.5cm]
	& \le & \disp (\So r)^{\frac{p^*}{1^*}} \left(\fint_B |\nabla \psi| \right)^{\frac{p^*}{1^*}} \le (\So r)^{\frac{p^*}{1^*}} \left(\fint_B |\nabla \psi|^p \right)^{\frac{p^*}{1^* p}}.
	\end{array}
	\]
Inserting into \eqref{eq_buono_4} and taking power $1^*/p^*$ we deduce \eqref{neusob_p} with constant
	\[
	\bar \So = \So p^\frac{1^*}{p^*} \left[ (p^*)^{\frac{\nu(p-1)}{\nu-p}} + 1\right]^{\frac{1^*}{p^*}},
	\]
which is uniformly bounded for $p \in [1,p_0]$. If $p^*/1^* > 1^*$ it is enough to iterate defining $p_1$ in such a way that $p^*/1^* = p_1^*$ and applying the above procedure with $p = p_1$ to estimate the last term in brackets in \eqref{eq_buono_4} by the integral of $g^{1/1^*}$. The iteration stops at step $k$ where  $p^*/(1^*)^k \le 1^*$.
\end{proof}
\begin{remark}
\emph{In particular, the above proposition and H\"older inequality guarantee that a manifold satisfying $\dous$ and $\neuuno$ also supports $\neup$ for each $p \in (1,\nu)$, with a constant $\Po_{p,p}$ that is uniformly bounded above on compact intervals $[1,p_0] \subset [1,\nu)$.
}
\end{remark}	
Next, by \cite[Prop. 2.8]{minerbe}, $\dous$, $\neup$ and property \eqref{eq_reversevol} guarantee the following \emph{relatively connected annuli} property:
\begin{itemize}
\item[] there exists $\kappa>1$ depending on $p, C_\Dou, \Po_{p,p}, C_{\RD}, b$ such that for each $R>0$, any two points in $\partial B_R(o)$ can be joined by a path that lies in $B_R(o) \backslash B_{\kappa^{-1} R}(o)$.
\end{itemize}
The constant $\kappa$, explicitly computed in \cite{minerbe}, remains bounded as $p \ra 1$ if so does $\Po_{p,p}$, and diverges as $p \ra b$. As a consequence, see \cite[Prop. 2.7]{tewo2}, from the decomposition of $M$ into annuli $A_j = B_{\kappa^{j+1}}(o) \backslash B_{\kappa^j}(o)$ one obtains a good covering $\{U_{i,a},U^*_{i,a}, U^\sharp_{i,a}\}_{(i,a) \in \Lambda}$ whose constants $Q_1,Q_2$ can be chosen to be independent of $p \in (1,p_0] \subset (1,b)$ for each pair of measures $\di x, \di \mu_{p,p^*}$, with $\di x$ the Riemannian volume and
	\[
	\di \mu_{p,p^*} = \left[ \frac{r^\nu}{|B_r(o)|}
\right]^{-\frac{p^*}{\nu}} \di x.
	\]
Thus, by Proposition 4.3 and Lemmas 4.1, 4.8 in \cite{tewo2}, discrete $p$-Poincar\'e inequalities hold with constants which are independent of $p \in (1,p_0]$. These, together with Lemma \ref{lem_Lpsob} above, ensure the validity of local Sobolev inequalities on $U_{i,a}$ and $U_{i,a}^*$, with Sobolev constants which are uniformly bounded for $p \in (1,p_0]$. The patching Theorem 2.3 in \cite{tewo2} yields the weighted Sobolev inequality \eqref{sob_tewodrose} with constant $\So_{p_0,\nu}$ only depending on $p_0, C_{\Dou}, \Po_{1,1}$, as claimed. Define $\eta$ as in \eqref{def_eta_tewo}. As mentioned above, $\eta$ is finite since $\nu \ge m$, and then Theorems \ref{theorem-tewo2} and \ref{teo_sobolev} imply the following
\begin{theorem} 
\label{theorem-decayGreen2}
Let $M^m$ be a complete manifold satisfying $\dous$ and $\neuuno$, with doubling dimension $\nu=\log_2 C_\Dou$. Assume that there exist constants $C_{\RD}>0$, $b\in (1, \nu]$ and a point $o \in M$ such that
\[
\forall \, t \ge s > 0, \qquad \frac{|B_t(o)|}{|B_s(o)|} \ge C_{\RD} \left( \frac{t}{s}\right)^b.
\]
Then, for each $p_0 \in (1,b)$, there exists $\So_{p_0,\nu}$ depending only on $C_\Dou$, $\Po_{1,1}$, $p_0$, $b$ and $C_{\RD}$ such that, for each $p \in (1,p_0]$, the Green kernel of $\Delta_p$ on an open set $\Omega \subset M$ with pole at $o \in \Omega$ satisfies
\begin{equation}\label{eq_decaygreen}
\gr(x) \le C_{p,\nu}^{\frac{1}{p-1}} \left[ \sup_{t \in (0,2r(x))} \frac{t^\nu}{|B_t(o)|}\right]^{\frac{1}{p-1}} r(x)^{- \frac{\nu-p}{p-1}}, \qquad \forall \, x \in  \Omega \backslash \{o\},
\end{equation}
where
$$
C_{p,\nu} = \So_{p_0,\nu}^{\frac{\nu}{p}} \left[ 2^{\nu}p(1+p)^p \left(\frac{p}{p-1}\right)^{p-1}\right]^{\frac{\nu-p}{p}}.
$$
\end{theorem}

\section{Convergence as $p\ra 1$}\label{sec_convergence}

We hereafter require the following\\[0.2cm]
\noindent \textbf{Assumption:} there exists $p_0 \in (1,m)$ such that $(\Hp)$ holds for every $p \in (1,p_0)$ with
$$
H(t) \ge 0, \qquad H(t) \ \ \text{ non-increasing on } \, \R^+.\vspace{0.2cm}
$$
where $(\Hp)$ is defined in Subsection \ref{subsec_fake}.\\[0.3cm]
%
%
Thus, we can  define $\varrho= \varrho_p$ as in \eqref{def_bxy} for each $p \in (1,p_0)$, and by the gradient estimates in Theorem \ref{teo_good}, $|\nabla \varrho_p| \le 1$. Up to passing to a subsequence, $\varrho_p \ra \varrho_1$ in the $C^{0,\alpha}_\loc$ topology on $M$, for some $\varrho_1$ that is $1$-Lipschitz. Here we investigate conditions to guarantee that $\varrho_1$ is positive on $M \backslash\{o\}$ and proper on $M$. The following observation will be repeatedly used.
%

\begin{lemma}\label{cor_measure}
Let $h \in C(\R^+_0)$ be positive and increasing on $\R^+$, fix $R \in (0,\infty]$ and let $\{t_j\} \subset (0,R)$ converging to some $t \in (0,R)$. Fix $p_0 > 1$ and let $\{p_j\} \subset (1, p_0)$ with $p_j\rightarrow 1$. If $R = \infty$, assume further that
	\begin{equation}\label{inte_p0}
	v_h^{-\frac{1}{p_0-1}} \in L^1(\infty).
	\end{equation}
Then,
\begin{equation}\label{liminf_rhop}
\disp \frac{1}{v_h(t)} = \disp \lim_{j \ra \infty} \left[\int_{t_j}^R v_h(s)^{-\frac{1}{p_j-1}}\di s\right]^{p_j-1}.
\end{equation}
\end{lemma}

\begin{proof}
Note first that \eqref{inte_p0} and the monotonicity of $v_h$ implies that $v_h^{-\frac{1}{p-1}} \in L^1(\infty)$ for each $p \in (1,p_0)$, in particular, the integrals in the RHS of \eqref{liminf_rhop} are finite. Let $c \in (0,t)$, and choose $j$ be large enough such that $t_j > c$. Then,
$$
\limsup_{j \ra \infty} \left[\int_{t_j}^R v_h(s)^{-\frac{1}{p_j-1}}\di s\right]^{p_j-1} \le \limsup_{j \ra \infty} \left\|\frac{1}{v_h} \right\|_{L^{\frac{1}{p_j-1}}([c,R))} = \left\|\frac{1}{v_h} \right\|_{L^{\infty}([c,R))} = \frac{1}{v_h(c)},
$$
where, in the last step, we used the monotonicity of $v_h$. Letting $c \uparrow t$ proves an inequality in \eqref{liminf_rhop}. The reverse inequality follows similarly.
\end{proof}
We first relate $\varrho_1$ to the solution to  the IMCF produced by R. Moser in \cite{moser}, and then comment about the inequality $|\nabla \varrho_1| \le 1$: setting
$$
w_p(x) = (1-p) \log \gr_p(x) = (1-p) \log \sgrh_p\big( \varrho_p(x)\big),
$$
then
\begin{equation}\label{eq_wp}
\Delta_p w_p = |\nabla w_p|^p
\end{equation}
pointwise on $M\setminus \{o\}$ and, because of Theorem \ref{teo_localsingular}, weakly on $M$.
Suppose that $\varrho_1 >0$ on $M \backslash \{o\}$. Then, passing to a subsequence and using \eqref{liminf_rhop},
$$
w_p  \to  \log v_h(\varrho_1)\doteq w_1 \qquad \text{locally uniformly on $M\backslash\{o\}$ as } \, p \ra 1,
$$
and by \cite{moser} $w_1$ is a weak solution to  the IMCF in the sense of Huisken-Ilmanen \cite{huiskenilmanen} on $M \backslash \{o\}$. It is convenient to consider the translated solution $w = w_1 -\log \omega_{m-1} = \log h(\varrho_1)^{m-1}$, in terms of which the inequality $|\nabla \varrho_1|\le 1$ can be rephrased as
$$
|\nabla w| \le (m-1)e^{-\frac{w}{m-1}} h'\Big( h^{-1}(e^\frac{w}{m-1})\Big).
$$
In particular, if $\Ricc \ge -(m-1)\kappa^2$ on $M$, an explicit computation gives
\begin{equation}\label{gradient_u_ex}
|\nabla w| \le (m-1)e^{-\frac{w}{m-1}} \sqrt{ \kappa^2 e^{\frac{2w}{m-1}} + 1}.
\end{equation}
These are the bounds described in Theorems \ref{teo_main_L1sobolev_intro} and \ref{teo_main_riccimagzero_intro} in the Introduction. We can interpret the following in terms of smooth IMCF. The term $|\nabla w|$ represents the unnormalized mean curvature $\mathcal{H}$ of the level set $\partial\{w<t\}$, which, along a smooth IMCF, is positive and varies according to
\begin{equation}\label{evol_MC}
\partial_t \mathcal{H} = - \Delta\left( \frac{1}{\mathcal{H}}\right) - \frac{|\nabla \nu|^2}{\mathcal{H}} - \frac{\Ricc(\nu,\nu)}{\mathcal{H}},
\end{equation}
with $\nabla \nu$ the second fundamental form of $\partial \{w<t\}$. Newton's inequality and $\Ricc \ge -(m-1)\kappa^2$ imply
$$
\frac{1}{2} \partial_t \mathcal{H}^2 \le  - \mathcal{H} \Delta\left( \frac{1}{\mathcal{H}}\right) - \frac{\mathcal{H}^2}{m-1} + (m-1)\kappa^2.
$$
\newline thus we obtain, in the sense of barriers,
$$
\partial_t \max\{\mathcal{H}^2/2\} \le - \frac{2}{m-1}\max\{\mathcal{H}^2/2\} + (m-1)\kappa^2.
$$
Integrating the ODE and taking square roots we get
\begin{equation}\label{particolare}
\max\{\mathcal{H}\}(t) \le (m-1)e^{-\frac{t}{m-1}} \sqrt{ \kappa^2 e^{\frac{2t}{m-1}} + \left[\frac{\max\{\mathcal{H}^2\}(0)}{(m-1)^2} -  \kappa^2 \right]}.
\end{equation}
This agrees with \eqref{gradient_u_ex} for a suitable translate of $w$. Estimate \eqref{gradient_u_ex} is therefore a version of \eqref{particolare} when the flow is not regular and the level set $\partial \{w<t\}$ is allowed to be non-compact. \\[0.2cm]
To establish the positivity of $\varrho_1$, we shall prove the following strong maximum principle which is a consequence of the sharp control on the constants in the Harnack inequality:

\begin{theorem}[\textbf{Strong maximum principle}]\label{teo_SMP}
Let $M$ be complete. Assume that $\Delta_p$ is non-parabolic on $M$ for $p \in (1,p_0)$, let $M_h$ be a model from below and define $\varrho_p$ on $M\backslash \{o\}$ according to \eqref{def_fake}. Assume that $\varrho_p \ra \varrho_1$ locally uniformly, for some sequence $p_j \ra 1$. Then, either $\varrho_1 \equiv 0$ on $M$ or $\varrho_1 >0$ on $M \backslash \{o\}$.
\end{theorem}

\begin{proof}
Suppose that $\varrho_1 \not \equiv 0$ on $M \backslash \{o\}$, and pick $x \in M \backslash \{o\}$ satisfying $\varrho_1(x)>0$. By \eqref{liminf_rhop}, this is equivalent to
$$
\lim_{j \ra \infty} (p_j-1)\log \gr_{p_j}(x) < \infty.
$$
Let $y \in M \backslash \{o\}$, and choose $R_1$ in such a way that $x,y$ belong to the same connected component of $M \backslash \overline{B}_{R_1}$. Set also $R_2 = \max\{ r(x),r(y)\}$ and $r = R_1/15$. Let $\{B_r(x_i)\}_{i=1}^N$,be a maximal collection of disjoint balls contained in the annulus $A(R_1,R_2) \doteq B_{R_2}(o)\backslash B_{R_1}(o)$. Choosing a constant $\kappa \ge 0$ such that $\Ricc \ge -(m-1)\kappa^2$ on $B_{6R_2}(o)$, Bishop-Gromov comparison ensures that
$$
|A(R_1,R_2)| \ge \sum_{i} |B_r(x_i)| \ge \sum_i |B_{2R_2}(x_i)| \frac{V_\kappa(r)}{V_\kappa(2R_2)} \ge N |A(R_1,R_2)| \frac{V_\kappa(r)}{V_\kappa(2R_2)},
$$
thus
$$
N \le \frac{V_\kappa(2R_2)}{V_\kappa(r)}.
$$
Since the family $\{B_{2r}(x_i)\}$ covers $A(R_1,R_2)$ and $B_{6r}(x_i) \Subset M \backslash \{o\}$, we can cover a path $\gamma \subset A(R_1,R_2)$ joining $y$ to $x$ via a chain of at most $N$ balls $\{B_l\}$ chosen in the family $\{B_{2r}(x_i)\}$ and with $B_l \cap B_{l+1} \neq \emptyset$. In view of Remark \ref{rem_localimportant}, both the $L^1$ Poincar\'e and the $L^1$ Sobolev holds on each $B_i$, with uniform constants depending on $\kappa$ and on a lower bound for $|B_{2R_2}(o)|$. Therefore, by Theorem \ref{teo_harnack} there exists a constant $C_1>1$ depending on $m,\kappa, R_2,|B_{2R_2}(o)|$ but independent of
\[
p_j \in \left(1, \min\left\{ p_0,\frac{2m-1}{2m-2}\right\}\right),
\]
such that
$$
\sup_{B_l} \gr_{p_j} \le C_1^{\frac{1}{p_j-1}} \inf_{B_l} \gr_{p_j} \qquad \text{for each $B_l$}.
$$
Iterating and taking logarithms,
$$
(p_j-1) \log \gr_{p_j}(y) \le (p_j-1) \log \gr_{p_j}(x) + N \log C_1,
$$
and therefore $(p_j-1) \log \gr_{p_j}(y)$ tends to a finite limit. Consequently, $\varrho_1(y)>0$.
\end{proof}
%

The next important Nondegeneracy Lemma ensures a control from below for $\varrho_1$ on balls containing the origin.

\begin{proposition}[\textbf{Nondegeneracy}]\label{prop_crucial}
Let $M$ be complete, $\Omega \subseteq M$ be an open set, and let $M_h$ be a model. Assume that $\Delta_p$ is non-parabolic both on $\Omega$ and on $M_h$ for every $p \in (1,p_0)$.  For 
every such $p$ let $\varrho_p$ be the fake distance associated by \eqref{def_bxy} to the Green kernel of $\Delta_p$ on $\Omega$ with pole at $o \in \Omega$. Suppose that $\varrho_p \ra \varrho_1$ locally uniformly  in $\Omega$ along some sequence $p_j \ra 1$, and that
$$
\varrho_1 > 0 \qquad \text{on } \, \Omega \backslash \{o\}.
$$
Then,
\begin{equation}\label{liminf_frontiera}
\forall K \Subset \overline{\Omega}\backslash\{o\}, \qquad  \inf_{K} \varrho_1(x) >0
\end{equation}
and, if $\overline{B}_R = \overline{B_R(o)} \subset \Omega$, 
\begin{equation}\label{good_22}
\varrho_1(x) \ge \min\left\{ v_h^{-1}\left( \frac{r(x)^{m-1}}{\So_{1,m,R}^m 2^{m^2-1}}\right), \min_{\partial B_R} \varrho_1 \right\} \qquad \text{on } \, B_R,
\end{equation}
where $\So_{1,m,R}$ is the $L^1$ Sobolev constant on $B_{R}$ for which
$$
\left(\int |\psi|^{\frac{m}{m-1}}\right)^{\frac{m-1}{m}} \le \So_{1,m,R} \int |\nabla \psi| \qquad \forall \, \psi \in \lip_c(B_{R}).
$$
\end{proposition}

\begin{proof}
To prove \eqref{liminf_frontiera}, note that for every $\eps$ such that $B_\eps = B_\eps(o) \Subset \Omega$ does not intersect $K$, the construction of the kernel $\gr_p$ on $\Omega$ guarantees that $\sup_{\partial B_\eps}\gr_p = \sup_{\Omega \backslash  B_\eps} \gr_p$. Rephrasing in terms of $\varrho_p$ and letting $p \ra 1$, this implies
$$
\inf_{\Omega \backslash B_\eps} \varrho_1 \ge \min_{\partial B_\eps} \varrho_1 > 0,
$$
and leads to \eqref{liminf_frontiera} and to the positivity of the lower bound in \eqref{liminf_frontiera}. To prove \eqref{good_22}, consider the kernel $\gr_p'$ of a relatively compact domain $\Omega' \Subset \Omega$. Set
$$
\tau = \min_{\partial B_R} \varrho_1,
$$
fix $0< \tau''< \tau' < \tau$ and let
$$
\Theta_p' \doteq \sgrh_p\big(\tau'\big) = \int_{\tau'}^{\infty} v_h(s)^{-\frac{1}{p-1}} \di s.
$$
By uniform convergence in $\partial B_R$, for $p \in \{p_j\}$ close enough to $1$ it holds $\varrho_p>\tau'$  on $\partial B_R$, that is, $\gr_p<\Theta'_p$ there, and since $\gr_p' \le \gr_p$ we have
\begin{equation}\label{grptheta}
\{\gr_p' > \Theta_p'\} \Subset B_{R}.
\end{equation}
By Remark \ref{rem_isoperimetric}, for each $p \in (1,m)$ there exists $\So_{p,m,R}$ such that
\begin{equation}\label{sobolev_local_p}
\left( \int |\psi|^{ \frac{m p}{m-p}}\right)^{\frac{m-p}{m}} \le \So_{p,m,R} \int |\nabla \psi|^p \qquad \forall \, \psi \in \lip_c\big(B_{R}\big),
\end{equation}
and $\So_{p,m,R} \ra \So_{1,m,R}$ as $p \ra 1$. To obtain the desired inequality in the limit $p \ra 1$, we apply Theorem \ref{teo_sobolev} to the function $(\gr'_p -\Theta_p')_+$, that is the Green kernel of the set $\{\gr'_p > \Theta_p'\}$ where \eqref{sobolev_local_p} holds, to obtain
$$
\gr'_p(x) \le \Theta_p' + C_{p,m,R}^{\frac{1}{p-1}} r(x)^{- \frac{m-p}{p-1}} \qquad \text{on } \, \{\gr_p' > \Theta_p'\}\backslash \{o\},
$$
with
\begin{equation}
C_{p,m,R} \ra C_{1,m,R} = \So_{1,m,R}^m 2^{m^2-1}.
\end{equation}
Letting $\Omega' \uparrow \Omega$ we deduce
\begin{equation}\label{eq_fortiori}
\gr_p(x) \le \Theta_p' + C_{p,m,R}^{\frac{1}{p-1}} r(x)^{- \frac{m-p}{p-1}} \qquad \text{on } \, \{\gr_p > \Theta_p'\}\backslash \{o\} = \{\varrho_p < \tau'\}\backslash \{o\}.
\end{equation}
A computation that uses \eqref{liminf_rhop} shows that
$$
\liminf_{p \ra 1} \left[\Theta_p' + C_{p,m,R}^{\frac{1}{p-1}} r(x)^{- \frac{m-p}{p-1}}\right]^{p-1} = \max \left\{ C_{1,m,R} r(x)^{1-m}, \frac{1}{v_h(\tau')}\right\}
$$
and thus, from $\{ \varrho_1 < \tau''\} \subset \{ \varrho_p < \tau'\}$ for $p$ small enough, raising \eqref{eq_fortiori} to power $p-1$, letting $p \ra 1$ and applying Lemma \ref{cor_measure} we infer
$$
\frac{1}{v_h(\varrho_1(x))} \le \max \left\{ C_{1,m,R} r(x)^{1-m}, \frac{1}{v_h(\tau')}\right\} = C_{1,m,R} r(x)^{1-m} \qquad \forall \, x \in \{ \varrho_1 < \tau''\},
$$
where the last equality follows since otherwise $\varrho_1(x) \ge \tau'$, contradicting $x \in \{ \varrho_1 < \tau''\}$. We therefore obtain \eqref{good_22} by letting $\tau'' \uparrow \tau$.
\end{proof}

We are now ready to prove our main theorems concerning proper solutions to  the IMCF. We first consider the case of a flow issuing from a point. The next two results prove, respectively, Theorems \ref{teo_main_L1sobolev_intro} and \ref{teo_main_riccimagzero_intro} in the Introduction.

\begin{theorem}\label{teo_main_L1sobolev}
Let $(M^m, \metric)$ be a connected, complete Riemannian manifold supporting the $L^1$ Sobolev inequality \eqref{isoperimetric} and satisfying, for some origin $o \in M$ and some $0 \le H \in C(\R^+_0)$ non-increasing,
$$
\Ricc \ge -(m-1)H(r),
$$
with $r(x) = \mathrm{dist}(x,o)$. Then, the fake distance $\varrho_1$ is positive, proper on $M \backslash \{o\}$ and there it satisfies
\begin{equation}\label{good_global}
\left\{ \begin{array}{l}
\disp v_h^{-1} \left( \frac{r(x)^{m-1}}{\So_{1,m}^{m} 2^{m^2-1}}\right) \le \varrho_1(x) \le r(x) \qquad \text{on } \, M, \\[0.5cm]
|\nabla \varrho_1| \le 1 \qquad \text{a.e. on } \, M.
\end{array}\right.
\end{equation}
Furthermore, $w = \log v_h(\varrho_1)$ is a solution to  the IMCF issuing from $o$.
\end{theorem}

\begin{proof}
By Remark \ref{rem_isoperimetric}, the $L^1$ Sobolev inequality implies the non-parabolicity of $\Delta_p$ for each $p \in (1,m)$. Hence, also $M_h$ is non-parabolic and $\varrho_p$ is well defined. Proposition \ref{prop_basiccomp} and Theorem \ref{teo_good} guarantee that $\varrho_p \ra \varrho_1$ up to a subsequence, for some $\varrho_1 \le r$ that is $1$-Lipschitz.
Next, by Theorem \ref{teo_sobolev} the bound \eqref{upper_consobolev} holds for each $p$ with $\eta \equiv 1$. Therefore, the sequence $\varrho_p$ is locally bounded away from zero and $\varrho_1 >0$ on $M \backslash \{o\}$. The lower bound in \eqref{good_global} follows as in the proof of Proposition \ref{prop_crucial}, applied with $\Omega = M$: in this case, the global Sobolev inequality guarantees that we can verbatim follow the proof by setting $R= \infty$, $\tau = \tau' = \tau'' = \infty$ and $\Theta'_p = 0$.
\end{proof}

%

\begin{remark}[\textbf{Asymptotically nonnegative Ricci curvature}]\label{rem_asiricci}
\emph{Particularly interesting is the case where $\Ricc \ge -(m-1)H(r)$ with
\begin{equation}\label{ipo_asiricci}
0 \le H(t) \ \ \text{ non-increasing on $\R^+$,} \qquad i_H \doteq \int_0^\infty t H(t) \di t < \infty.
\end{equation}
In this case, under the assumptions of Theorem \ref{teo_main_L1sobolev} the fake distance is of the same order of $r$: indeed, by work of Greene-Wu \cite[Lem. 4.5]{greenewu}, the volume $v_h$ of $M_h$ satisfies
$$
v_h(t) \le e^{(m-1)i_H} v_0(t) \qquad \text{on } \, \R^+,
$$
and thus, taking into account the value of the $L^1$ Sobolev constant $\So_{\R^m} = m^{-\frac{m-1}{m}} \omega_{m-1}^{-\frac{1}{m}}$ of $\R^m$, the first in \eqref{good_global} implies
$$
\left[ e^{-i_H} \left(\frac{\So_{\R^m}}{\So_{1,m}}\right)^{\frac{m}{m-1}} \frac{m}{2^{m+1}}\right] r(x) \le \varrho_1(x) \le r(x) \qquad \text{on } \, M.
$$
We conjecture that the constant $m/2^{m+1}$ should be replaced by $1$. Observe that if this were the case, one would be able to recover a rigidity result of M. Ledoux \cite{ledoux}, who showed that $\R^m$ is the only  manifold with $\Ricc \ge 0$ for which a Sobolev inequality holds with constant $\So_{1,m} = \So_{\R^m}$. See also \cite{pigovero} and the references therein for improvements.
}
\end{remark}

Our second main result is for manifolds satisfying the assumptions in Theorem \ref{theorem-decayGreen2}. Recall that a smooth manifold has doubling dimension $\nu \ge m$.
\begin{theorem}\label{teo_main_riccimagzero}
Let $M^m$ be a connected, complete non-compact manifold satisfying
$$
\Ricc \ge -(m-1)H(r),
$$
for some origin $o \in M$ and some $0 \le H \in C(\R^+_0)$ non-increasing. Assume further the global doubling and weak $(1,1)$-Poincar\'e properties $\dous, \neuuno$ with constants $C_{\Dou}, \Po_{1,1}$ and doubling dimension $\nu = \log_2 C_\Dou$, and that there exist $C_{\RD}$ and $b \in (1,\nu]$ such that
\begin{equation}\label{eq_reversevol_222}
\forall t \ge s >0, \qquad \frac{|B_t|}{|B_s|} \ge C_{\RD} \left( \frac{t}{s}\right)^b,
\end{equation}
where balls are centered at $o$. Then, the fake distance $\varrho_1$ is positive and proper on $M \backslash \{o\}$. Moreover, there exist constants $C, \bar C$ depending on $H(0), C_{\RD}, C_{\Dou}, \Po_{1,1},b,m$, with $\bar C$ also depending on a lower bound for the volume $|B_1|$, such that
\begin{equation}\label{good_global_2}
\left\{ \begin{array}{ll}
\disp v_h^{-1} \left( C r(x)^{\nu-1} \inf_{t \in (1, r(x))} \frac{|B_t|}{t^\nu} \right) \le \varrho_1(x) \le r(x) & \quad \text{on } \, M \backslash B_1, \\[0.5cm]
\bar C r(x) \le \varrho_1(x) \le r(x) & \quad \text{on } \, B_1,\\[0.3cm]
|\nabla \varrho_1| \le 1 \qquad \text{a.e. on } \, M.
\end{array}\right.
\end{equation}
Finally, $w = \log v_h(\varrho_1)$ is a solution to  the IMCF issuing from $o$.
\end{theorem}
\begin{proof}
Set $p_0 = (1+b)/2$. Theorem \ref{theorem-decayGreen2} guarantees that $\Delta_p$ is non-parabolic for each $p \in (1,p_0]
$, and that
\begin{equation}\label{esti_ll}
(1-p) \log \gr(x) \ge - \log \left\{ C_{p,\nu} \left[ \sup_{t \in (0,2r(x))} \frac{t^\nu}{|B_t|}\right] r(x)^{p-\nu} \right\}, \qquad \forall \, x \in  M\backslash \{o\},
\end{equation}
with $C_{p,\nu}$ uniformly bounded for $p \in (1,p_0]$. Proceeding as in Theorem \ref{teo_main_L1sobolev} shows that $\varrho_1$ exists and is $1$-Lipschitz. Letting $p \ra 1$ in \eqref{esti_ll} we therefore obtain
$$
v_h\big(\varrho_1(x)\big) \ge C_{1,\nu}^{-1} \left[ \inf_{t \in (0,2r(x))} \frac{|B_t|}{t^\nu}\right] r(x)^{\nu-1} \qquad \text{on } \, M.
$$
By volume comparison on $B_1$, we deduce $|B_t|/t^m \ge C_1|B_1|$ for $t \in [0,1]$, where the constant $C_1$ depends on $H(0)$, and thus $|B_t|/t^\nu \ge C_1|B_1|t^{m-\nu} \ge C_1|B_1|$. Therefore, by the doubling condition, for $r(x) \ge 1$
	\[
\inf_{t \in (0,2r(x))} \frac{|B_t|}{t^\nu} \ge C_2 \inf_{t \in (0,r(x))} \frac{|B_t|}{t^\nu} \ge C_2 \min \left\{ C_1|B_1|, \inf_{t \in (1,r(x))} \frac{|B_t|}{t^\nu} \right\} \ge C_3 \inf_{t \in (1,r(x))} \frac{|B_t|}{t^\nu}.
	\]
Thus, the first inequality in \eqref{good_global_2} holds. Next, \eqref{eq_reversevol_222} implies $|B_t|/t^\nu \ge C_{\RD}|B_1| t^{b-\nu}$ for $t \ge 1$, hence
\begin{equation}\label{eq_lo_bound}
v_h\big(\varrho_1(x)\big) \ge C_4 |B_1| r(x)^{b-1} \qquad \text{on } \, M\backslash B_1.
\end{equation}
In particular, $\varrho_1$ is proper and thus strictly positive on $M \backslash \{o\}$ by Theorem \ref{teo_SMP}. To prove the lower bound in $B_1$, we use the Nondegeneracy Lemma with $\Omega = M$ and $R=1$ taking into account \eqref{eq_lo_bound} to deduce
	\begin{equation}\label{eq_lowb}
	\varrho_1(x) \ge \disp \min\left\{ v_h^{-1} \left(\frac{r(x)^{m-1}}{\So^{m} 2^{m^2-1}}\right), v_h^{-1} (C_4|B_1|)\right\} \qquad \text{on } \, B_1,
	\end{equation}
where $\So$ is the Sobolev constant of $B_1$. By Remark \ref{rem_localimportant}, $\So$ can be estimated in terms of a lower bound for $|B_1|$ and a lower bound for $\Ricc$ on $B_6$ (that is, on $H(0)$). Hence, the right hand side of \eqref{eq_lowb} can be estimated from below by $\bar C r(x)$, for some constant $\bar C$ depending on the same parameters as $C$ and also on a lower bound for $|B_1|$. This concludes the proof.
\end{proof}
\begin{remark}
\emph{If $M$ has asymptotically non-negative Ricci curvature, by Remark \ref{rem_asiricci} the first in \eqref{good_global_2} can be rewritten as
	\[
	C' r(x)^{\frac{b-1}{m-1}} \le \varrho_1(x) \le r(x) \qquad \text{on } \, M \backslash B_1,
	\]
where $C'$ depends on the same parameters as $C$, on a lower bound for $|B_1|$ and also on $i_H$ in \eqref{ipo_asiricci}.
}
\end{remark}
\begin{proof}[Proof of Theorem \ref{teo_main_riccimagzero_intro}] It is enough to apply Theorem \ref{teo_main_riccimagzero} with the choices $H = 0$, $h(t) = t$. By Bishop-Gromov comparison, the doubling constant can be chosen to be $C_\Dou = 2^m$, so $\nu = m$. Note that the translated function $w = (m-1) \log \varrho_1$ satisfies \eqref{gradient_u_ex} with $\kappa = 0$.
\end{proof}

We next consider the flow starting from a relatively compact set $\Omega$ and we give the

\begin{proof}[Proof of Theorem \ref{teo_main_relcompact}]
We recall that given the $p$-capacity potential $u_p$ of $(\Omega,M)$, the (proper) solution to  the IMCF is obtained as the limit
$$
w(x) = \lim_{p \ra 1} (1-p) \log u_p(x).
$$
For $p \in [1, p_0)$, define the fake inner and outer $p$-radii:
$$
R^{(p)}_i = \sup \Big\{ t : \{ \varrho_p < t\} \subset \Omega \Big\}, \qquad R^{(p)}_o = \inf \Big\{ t : \Omega \subset \{\varrho_p < t\} \Big\}.
$$
Note that, by the uniform convergence of $\varrho_p$,
$$
R_i = \liminf_{p \ra 1} R_i^{(p)}, \qquad R_o = \limsup_{p \ra 1} R_o^{(p)}.
$$
By comparison, $u_p$ satisfies
\begin{equation}\label{eq_dopp}
\left[ \sgrh_p\big( R^{(p)}_i\big) \right]^{-1} \gr_p(x) \le u_p(x) \le \left[ \sgrh_p\big(R^{(p)}_o\big) \right]^{-1} \gr_p(x),
\end{equation}
hence taking logarithms, multiplying by $1-p$, letting $p \ra 1$ and using Lemma \ref{cor_measure} we infer
\begin{equation}\label{doublebo}
\log v_h\big(\varrho_1(x)\big) - \log v_h\big(R_o\big) \ \le \ w(x) \ \le \ \log v_h\big(\varrho_1(x)\big) - \log v_h\big(R_i\big).
\end{equation}
Also, since $\varrho_p \le r$, the definition of $R_i^{(p)}$ and the monotonicity of $H$ yield
$$
\Ricc \ge - (m-1) H(r) \ge -(m-1) H\big(R_i^{(p)}\big) \qquad \text{on } \, M \backslash \Omega,
$$
and we can use Theorem \ref{teo_bellagradiente} since $u_p$ is bounded to deduce the inequality
\begin{equation}\label{const_gradup-1}
(p-1) |\nabla \log u_p| \le \max \left\{ (m-1) \sqrt{H\big(R_i^{(p)}\big)}, \ (p-1) \max_{\partial \Omega} |\nabla \log u_p| \right\}.
\end{equation}
Next, by the boundary gradient estimate in \cite[Prop. 3.1]{kotschwarni} (cf. also \cite{huiskenilmanen}), for fixed $\eps>0$ there exists $p_\eps \in (1, p_0)$ depending on $\eps$ and on the geometry of a neighbourhood of $\partial \Omega$ such that
\begin{equation}\label{bge}
(p-1) \max_{\partial \Omega} |\nabla \log u_p| \le \max_{\partial \Omega} \mathcal{H}_+ + \eps \qquad \forall \, p \in (1, p_\eps),
\end{equation}
with $\mathcal{H}_+(x) = \max\{ \mathcal{H}(x),0\}$. Taking limits in \eqref{const_gradup-1} in $p$ and eventually letting $\eps \ra 0$ we get
\begin{equation}\label{const_gradup}
|\nabla w| \le \max \left\{ (m-1) \sqrt{H(R_i)}, \ \max_{\partial \Omega} \mathcal{H}_+ \right\},
\end{equation}
which concludes the proof.
\end{proof}

\begin{remark}
\emph{If $\Ricc \ge 0$, \eqref{const_gradup} gives the estimate
$$
|\nabla w| \le \max_{\partial \Omega} \mathcal{H}_+,
$$
which forces $\max_{\partial \Omega} \mathcal{H}_+ >0$. This is consistent with a result of A. Kasue (see \cite[Thm. C (2)]{kasue2}, cf. also \cite[Thm. 1.6]{agofogamazzie}), stating that if a complete, non-compact manifold with $\Ricc \ge 0$ contains a relatively compact subset $\Omega$ with $\mathcal{H} \le 0$, then $M \backslash \Omega$ splits isometrically as $\partial \Omega \times \R^+_0$. Clearly, none of these manifolds satisfy the volume growth condition \eqref{eq_reversevol_222}.
}
\end{remark}

\begin{remark}\label{rem_concluding}
\emph{The gradient estimate in \eqref{const_gradup} can be improved, under the same assumptions, to include a decay in terms of $\varrho_1$. The procedure goes as follows: define a fake distance $\bar \varrho_p$ via the identity
$$
\sgrh_p\big(R^{(p)}_i\big) u_p(x) = \sgrh_p\left( \bar \varrho_p(x)\right).
$$
Because of \eqref{eq_dopp}, we deduce
\begin{equation}\label{twobound_345}
\hat \varrho_p \doteq (\sgrh_p)^{-1} \left( \sgrh_p(\varrho_p) \frac{\sgrh_p\big(R^{(p)}_i\big)}{\sgrh_p\big(R^{(p)}_o\big)} \right) \le \bar \varrho_p \le \varrho_p.
\end{equation}
Hence, the inequality $\varrho_p \le r$ that follows by Proposition \ref{prop_basiccomp} implies $\Ricc \ge - H(\bar \varrho_p)$, thus we can apply Lemma \ref{lem_key} to get, since $u_p$ is bounded,
$$
|\nabla \bar \varrho_p| \le \max \left\{ 1, \ \max_{\partial \Omega} | \nabla \bar \varrho_p| \right\}.
$$
Rephrasing in terms of $u_p$ and recalling that $u_p =1$ on $\partial \Omega$ implies $\bar \varrho_p = R_i^{(p)}$ on $\partial \Omega$,
$$
|\nabla \log u_p| \le \big| (\log \sgrh_p)'(\bar \varrho_p)\big| \cdot \max \left\{ 1, \ \frac{\max_{\partial \Omega} |\nabla \log u_p|}{\big| (\log \sgrh_p)'(R_i^{(p)})\big| } \right\}.
$$
Because of the monotonicity of $(\log \sgrh_p)'$ in Lemma \ref{lem_ODE}, and because of the boundary gradient estimate \eqref{bge}, if $p$ is close enough to $1$ then
\begin{equation}\label{esplicita!}
(p-1)|\nabla \log u_p| \le \big| (\log \sgrh_p)'(\hat \varrho_p)\big| \cdot \max \left\{ p-1, \ \ \frac{\eps + \max_{\partial \Omega} \mathcal{H}_+ }{\big| (\log \sgrh_p)'(R_i^{(p)})\big| } \right\}.
\end{equation}
Explicit computations can be performed in relevant cases, notably when $H(r) = \kappa^2/r^2$ for some constant $\kappa \ge 0$. A solution to  $h'' = Hh$ is
$$
h(t) = t^{\kappa'} \qquad \text{with} \qquad \kappa' = \frac{1 + \sqrt{1+4\kappa^2}}{2} \ge 1
$$
(technically, $h$ does not solve the initial condition for the derivative in \eqref{eq_h_uguale} when $\kappa> 0$, nor $H$ is continuous in zero, but this can be handled via Sturm comparison, cf. Remark \ref{rem_compacritical} and \cite{prs, bmr2, bmr5}).
Computing $\sgrh_p$ we deduce that $\hat \varrho_p = \varrho_p R_i^{(p)}/ R_o^{(p)}$ in \eqref{twobound_345}, so we can rewrite \eqref{esplicita!} as follows:
\begin{equation}
(p-1)|\nabla \log u_p| \le \frac{\kappa'(m-1)-p+1}{\varrho_p} \frac{R_o^{(p)}}{R_i^{(p)}} \cdot \max \left\{ 1, \ \frac{R_i^{(p)}\left( \eps + \max_{\partial \Omega} \mathcal{H}_+\right) }{\kappa'(m-1) -p+1} \right\},
\end{equation}
Letting $p \ra 1$ and then $\eps \ra 0$ we eventually infer
\begin{equation}
|\nabla w| \le \frac{R_o}{\varrho_1} \max \left\{ \frac{\kappa'(m-1)}{R_i}, \ \max_{\partial \Omega} \mathcal{H}_+ \right\},
\end{equation}
to be compared to \eqref{HI_gradesti}.
}
\end{remark}

\section{Basic isoperimetric properties of the sets $\{\varrho_1 < t\}$}\label{sec_isoperimetry}

Let $\varrho_1$ be the fake distance constructed in either Theorem \ref{teo_main_L1sobolev} or Theorem \ref{teo_main_riccimagzero}, and let $w = \log v_h(\varrho_1)$ be the associated solution to  the IMCF.  The purpose of this section is to estimate the isoperimetric profile of the sets $\{ \varrho_1 < t\}$.
We recall that, if $\tau_M$ denotes the family of subsets of $M$ with finite perimeter, the isoperimetric profile of $M$ is the function
$$
I_M \ \ : \ \ (0,|M|) \to \R, \qquad I_M(\upsilon) = \inf\Big\{ \haus^{m-1}(\partial^*\Omega) \ : \ \Omega \in \tau_M, |\Omega|= \upsilon \Big\},
$$
where $\haus^{m-1}$ is the $(m-1)$-dimensional Hausdorff measure and $\partial^*\Omega$ is the reduced boundary of $\Omega$. A subset realizing $I_M(\upsilon)$ is called an isoperimetric subset. By an application of Bishop-Gromov comparison theorem, it is known (cf. \cite[Thms. 3.4-3.5]{morganjohnson} and \cite[Prop. 3.2]{mondinonardulli}) that under the condition
\begin{equation}\label{lower_ricci_H}
\Ricc \ge -(m-1)H(r) \qquad \text{on } \, M
\end{equation}
for some $0 \le H \in C(\R^+_0)$ non-increasing, then the area of $\partial B_r(o)$ is no bigger than the surface area of the  ball $\BB_r \subset M_h$ centered at the origin and having the same volume as $B_r(o)$. Moreover, rigidity holds in case of equality, namely, $B_r(o)$ and $\BB_r$ are isometric.
The theorem, stated in \cite{morganjohnson, mondinonardulli} for constant $H$, also holds for each $H \ge 0$ non-increasing and is, in fact, a consequence of the concavity of the area of $\partial \BB_r$ as a function of the volume:
\begin{equation}\label{conca}
s \mapsto v_h \big( V_h^{-1}(s)\big) \qquad \text{is strictly concave on } \, \R^+.
\end{equation}
Property \eqref{conca} follows from Lemma \ref{lem_ODE}, since a first differentiation shows that it is equivalent to the decreasing monotonicity of $v_h'/v_h$. As a consequence, $I_M$ does not exceed the isoperimetric profile of \emph{geodesic balls centered at the origin in $M_h$}.

\begin{remark}
\emph{If $H \ge 0$ is non-increasing, then $\BB_r$ is never an isoperimetric set in $M_h$ unless $H$ is constant on $[0,r]$. Indeed, $\partial \BB_r$ is even unstable: to see this we use the Riccati equation \eqref{eq_riccati} to write its stability operator $J$ as
$$
J = \disp -\Delta_{\partial \BB_r} - \Big( \Ricc_h(\nabla r, \nabla r) + |\II_{\partial \BB_r}|^2 \Big) = -\frac{\Delta_{\Sph}}{h^2(r)} + (m-1)\left(\frac{h'(r)}{h(r)}\right)',
$$
with $\Delta_\Sph$ the Laplacian on the unit sphere $\Sph^{m-1} \subset \R^{m}$. Since the first nonzero eigenvalue of $\Sph^{m-1}$ is $m-1$, $J$ is non-negative for variations $\phi$ with $\int_{M_h} \phi =0$ if and only if
\begin{equation}\label{eq234}
\frac{1}{h^2(r)} \ge - \left( \frac{h'(r)}{h(r)}\right)', \qquad \text{that is,} \qquad \big(h'(r)\big)^2 - h(r) h''(r) \le 1.
\end{equation}
However, $((h')^2 - h h'')(0^+) = 1$ and  $((h')^2 - h h'')'=-h^2H'$. 
Therefore, if $H$ is non-increasing then \eqref{eq234} holds if and only if $H$ is constant on $[0,r]$. To our knowledge, the problem of deciding which conditions on $H$ guarantee that balls $\BB_r \subset M_h$ are isoperimetric sets is still partly open, and in this respect see \cite{brendle}. However, more can be said for surfaces, see Theorems 2.8 and 2.16 in \cite{ritore} as well as Theorem 3.1 in \cite{HowHutMor}.
}
\end{remark}

Our purpose is to show that, similarly, the perimeter of the subsets $\{ \varrho_1 < t\}$ is smaller than the one of balls in $M_h$ centered at the origin and having the same volume. We shall first describe in more detail the behavior of $\varrho_1$ near the origin.

\begin{proposition}\label{teo_inzero_rho1}
Let $\varrho_1$ be the fake distance with origin  at  $o$ constructed in either
Theorem \ref{teo_main_L1sobolev} or Theorem \ref{teo_main_riccimagzero}. Then,
%
\begin{equation}\label{asin_rho1}
\begin{array}{lll}
(i) & \disp \varrho_1(x) \sim r(x) & \qquad \text{as } \, x \ra o, \\[0.2cm]
(ii) & \disp \frac{\haus^{m-1}( \partial \{ \varrho_1 < t\})}{v_h(t)} \ra 1 & \qquad \text{as } \, t \ra 0.
\end{array}
\end{equation}
\end{proposition}

\begin{proof}
Let $B_{R_0}(o) \subset \mathscr{D}_o$ and choose $\bar \kappa$ to satisfy $\Sect_\rad \le \bar \kappa^2$ on $B_{R_0}(o)$. By reducing $R_0$ we may assume that
$R_0 < \min\{1,\pi/(2\bar \kappa)\}$. For $p \in (1, 3/2)$, let $v_{\bar \kappa}$ and $\sgr_{R_0}^{\bar \kappa}$ be the volume of spheres and the kernel of $\Delta_p$ for the model of curvature $\bar \kappa^2$. By  Corollary \ref{teo_confronto_conmodel},
\begin{equation}\label{doppiobound}
\sgrh_{R_0}\big(r(x)\big) \le \gr(x) \le \sgr_{R_0}^{\bar \kappa}\big( r(x)\big) + \sup_{\partial B_{R_0}} \gr.
\end{equation}
Using Theorem \ref{teo_sobolev}, there exists a constant $C_p$ bounded as $p \ra 1$ such that
$$
\gr(x) \le C_p^{\frac{1}{p-1}} R_0^{-\frac{m-p}{p-1}} \qquad \text{on } \, \partial B_{R_0}.
$$
Plugging into \eqref{doppiobound}, raising to power $(p-1)$, letting $p \ra 1$, taking logarithms and using Lemma \ref{cor_measure}, we get
\begin{equation}\label{moltosharp}
\begin{array}{rl}
\disp - \log v_h\big(r(x)\big) \le - \log v_h\big(\varrho_1(x)\big) \le & \disp \log \ \limsup_{p \ra 1} \left[\sgr_{R_0}^{\bar \kappa}\big( r(x)\big) + C_p^{\frac{1}{p-1}} R_0^{-\frac{m-p}{p-1}}\right]^{p-1} \\[0.5cm]
\le & \disp \log \max \left\{ C_1 R_0^{1-m}, \frac{1}{v_{\bar \kappa}(r(x))}\right\} \\[0.4cm]
\le & \disp - \log v_{\bar \kappa}\big(r(x)\big),
\end{array}
\end{equation}
where the last inequality follows provided that we choose $x \in B_R$ with $v_{\bar \kappa}(R) \le R_0^{m-1}/C_1$, and $(i)$ follows immediately.
We next use $(i)$ to show $(ii)$ via blow-up: consider the exponential chart $\BB_{R_0} \subset \R^m \ra B_{R_0}(o)$ with polar coordinates $(s,\theta)$, and let $\metric$ be the pull-back metric. For $\lambda>0$ we define the dilation
$$
T_\lambda : \BB_{\frac{R_0}{\lambda}}^* \ra \BB_{R_0}^*, \qquad T_\lambda(s,\theta) = (\lambda s, \theta),
$$
and set $g_\lambda = \lambda^{-2} T_\lambda^* \metric$. Then, $g_\lambda$ converges locally smoothly on $\R^m$ to the Euclidean metric $g_0$ as $\lambda \ra 0$, and by rescaling $w_\lambda = w \circ T_\lambda$ solves the IMCF on $(\BB_{R_0/\lambda}^*, g_\lambda)$. To pass to limits in $\lambda$ we shall normalize $w_\lambda$, so for fixed $(s_0,\theta_0)$ and for $\lambda < R_0/s_0$ define
$$
\bar w_\lambda(s,\theta) = w_\lambda(s,\theta) - w_\lambda(s_0,\theta_0) \qquad \text{on } \, \BB_{\frac{R_0}{\lambda}}^*.
$$
In the next computation, crucial for us are the gradient bound $|\nabla \varrho_1| \le 1$ and \eqref{moltosharp}, that guarantees $\varrho_1 \ge Cr $ on $\BB_{R_0}^*$. Indeed, if $\nabla^\lambda, |\cdot |_\lambda$ are the gradient and norm in the metric $g_\lambda$,
$$
\begin{array}{lcl}
|\nabla^\lambda \bar w_\lambda(s,\theta)|^2_\lambda = \lambda^2 |\nabla w(\lambda s,\theta)|^2 = \lambda^2 \left[\frac{v_h'}{v_h}(\varrho_1)\right]^2|\nabla \varrho_1|^2 \le \lambda^2 \frac{C_1}{\varrho_1^2(\lambda s, \theta)} \le \frac{C_2}{s^2}
\end{array}
$$
for uniform constants $C_1$ and $C_2$. This and $\bar w_\lambda(s_0,\theta_0) = 0$ guarantee that $\bar w_\lambda \ra \bar w_0$ locally in $C^\alpha$ on $\R^m \backslash \{0\}$, with $\bar w_0(s_0,\theta_0) = 0$, and by the compactness Theorem 2.1 in \cite{huiskenilmanen} (tweaked to include the case of variable metrics) $\bar w_0$ is a solution to  the IMCF on $\R^m \backslash \{0\}$. By \cite[Prop. 7.1]{huiskenilmanen}, $\bar w_0$ is necessarily a flow of spheres, thus $\bar w_0(s,\theta) = (m-1) \log (s/s_0)$. The regularity result in \cite[Thm. 1.3]{huiskenilmanen} ensures that for fixed $\sigma$ the sets $\partial \{\bar w_\lambda < \sigma \}$ possess uniform $C^{1,\alpha}$ bounds in $\lambda$, outside of a set of Hausdorff dimension $m-8$. Therefore,
\begin{equation}\label{kk}
\forall \, \sigma \in \R, \qquad \haus_{g_\lambda}^{m-1}\big( \partial \{\bar w_\lambda < \sigma\} \big) \ra \haus_{g_0}^{m-1}\big( \partial \{\bar w_0 < \sigma\} \big) = \omega_{m-1} s_0^{m-1} e^{\sigma}
\end{equation}
and by rescaling
\begin{equation}\label{kkk}
\haus_{g_\lambda}^{m-1}\big( \partial \{(s, \theta) : \bar w_\lambda(s, \theta) < \sigma\} \big) = \disp \lambda^{1-m} \haus^{m-1}_{\metric} \big( \partial \{(r,\theta) : w(r, \theta) < \sigma + w(\lambda s_0,\theta_0) \} \big).
\end{equation}
Rephrasing \eqref{moltosharp} in terms of $w$, for fixed $\eps>0$ there exists $R_\eps$ such that
\begin{equation}\label{bella_asi_u}
w(x) = (m-1) \log r(x) + \log \omega_{m-1} + o_\eps(1) \qquad \text{on } \, B_{R_\eps}^*(o),
\end{equation}
where $o_\eps(1)$ is a function that vanishes as $\eps \ra 0$, uniformly on $B_{R_\eps}(o)$. Having defined $t$ according to
$$
\log v_h(t) = \sigma + w(\lambda s_0,\theta_0), \quad \text{so that, by \eqref{bella_asi_u},} \quad  t^{m-1} = e^{\sigma} (\lambda s_0)^{m-1} (1+ o_\eps(1)),
$$
from \eqref{kk} and \eqref{kkk} we deduce
$$
1 = \lim_{\lambda \ra 0} \frac{\haus^{m-1}( \partial\{ \varrho_1 < t\})}{\omega_{m-1}\lambda^{m-1}s_0^{m-1}e^\sigma} = (1+o_\eps(1)) \lim_{t \ra 0} \frac{\haus^{m-1}( \partial\{ \varrho_1 < t\})}{\omega_{m-1}t^{m-1}},
$$
and the conclusion follows by letting $\eps \ra 0$.
\end{proof}

With the above preparation, we are ready to investigate the sets $\{\varrho_1< t\}$. Let $p_0>1$ be close enough to $1$ in such a way that $\Delta_p$ is non-parabolic on $M$ for $p \in (1, p_0)$ ($p_0=m$ for Theorem \ref{teo_main_L1sobolev}, $p_0=b$ for Theorem \ref{teo_main_riccimagzero}). For $t>0$ and $u \in C(M)$ define
$$
\Ar_u(t) = \frac{1}{v_h(t)}\int_{\vutp} u|\nabla \varrho_p|^{p-1}, \qquad \VV_u(t) = \frac{1}{V_h(t)} \int_{\vmt} u|\nabla \varrho_p|^p.
$$


\begin{lemma}\label{lemma_diffeAuVu}
Suppose that $\varrho= \varrho_p$ is proper on $M$. If $u \in C(M) \cap W^{1,1}_\loc(M \backslash \{o\})$ then $\VV_u$ is absolutely continuous on $\R^+$ and $\Ar_u$ is a.e. equivalent to an absolutely continuous function $\hat{\Ar}_u$. Moreover,
\begin{equation}\label{eq_deriAuVu}
\begin{array}{ll}
(i) & \quad \disp \VV_u'(t) = \frac{v_h(t)}{V_h(t)} \Big[ \Ar_u(t)- \VV_u(t)\Big] \\[0.5cm]
(ii) &\quad \disp \hat{\Ar}_u'(t) = \frac{1}{v_h(t)} \int_{\vut} |\nabla \varrho|^{p-2} \langle \nabla u,\nu\rangle,
\end{array}
\end{equation}
a.e. on $\R^+$, with $\nu = \nabla \varrho/|\nabla \varrho|$.
\end{lemma}

\begin{proof}
Identity $(i)$ and the absolute continuity of $\VV$ is a simple consequence of the coarea's formula: for a.e. $t$, 
$$
\VV_u(t)' = \left(\frac{1}{V_h}\right)' \int_{\vmt} u|\nabla \varrho|^p + \frac{1}{V_h} \int_{\vut}u|\nabla \varrho|^{p-1} = \frac{1}{V_h^2} \left[ -v_h \cdot (V_h \VV_u) + V_h (v_h\Ar_u)\right]
$$
To prove $(ii)$, define $\ell = \ell(t)$ such that $\{\rho < t\} = \{\gr > \ell\}$. Rewriting \eqref{ide_Gr} in terms of $\varrho$ and $t$ we get, for a.e. $t$, 
\begin{equation}\label{ide_Au}
\begin{array}{lcl}
u(o) & = & \disp \int_{\{ \varrho< t\}} |\nabla \gr|^{p-2} \langle \nabla \gr, \nabla u \rangle - \int_{\{ \varrho = t\}} |\nabla \gr|^{p-2} u\langle \nabla \gr, \frac{\nabla \varrho}{|\nabla \varrho|} \rangle \\[0.5cm]
& = & \disp \int_{\{ \varrho< t\}} \frac{|\nabla \varrho|^{p-2}}{v_h(\varrho)} \langle \nabla \varrho, \nabla u \rangle - \Ar_u(t).
\end{array}
\end{equation}
Hence, $\Ar_u(t)$ coincides a.e with an absolutely continuous function $\hat{\Ar}_u(t)$. Item $(ii)$ follows by differentiating $\hat{\Ar}_u$ and using the coarea's formula.
\end{proof}


If $p<m$, we note from \eqref{ide_Au} and Proposition  \ref{prop_nearminm} that
	\[
	u(o) = \lim_{t \ra 0} \Ar_u(t) \qquad \forall \, u \in C^1(M).
	\]
Thus, if $u \equiv 1$, from Lemma \ref{lemma_diffeAuVu} we obtain
\begin{equation}\label{deri_V1}
\Ar_1(t) \equiv 1, \qquad \VV_1(t) = \frac{1}{V_h(t)} \int_0^t v_h(s) \Ar_1(s)\di s \equiv 1.
\end{equation}

\begin{theorem}\label{teo_inzero_rho2}
Let $\varrho_1$ be the fake distance constructed in either Theorem \ref{teo_main_L1sobolev} or Theorem \ref{teo_main_riccimagzero}, with pole at $o$. Then,
\begin{equation}\label{asin_rho2}
\begin{array}{lll}
(i) & \disp \haus^{m-1}\big( \partial \{ \varrho_1 < t\} \big) = \haus^{m-1}\big( \partial \mathrm{int} \{\varrho_1 \le t\}\big) = v_h(t) & \qquad \forall \, t \in \R^+, \\[0.2cm]
(ii) & | \{ \varrho_1 < t\} | \ge V_h(t) & \qquad \forall \, t \in \R^+.
\end{array}
\end{equation}
\end{theorem}

\begin{proof}
Let $w = \log v_h(\varrho_1)$ be the associated solution to  the IMCF. Because of the Minimizing Hull Property 1.4 and Lemma 1.6 in \cite{huiskenilmanen},
$$
e^{-s} \haus^{m-1}\big(\partial \{ w < s\}\big) = e^{-s} \haus^{m-1} \big( \partial \mathrm{int}(\{w \le s\}) \big) \qquad \text{is constant for } \, s \in \R.
$$
Changing variables according to $s = \log v_h(t)$, and taking into account $(ii)$ in Proposition \ref{teo_inzero_rho1}, we immediately deduce $(i)$. To prove $(ii)$, for $t \in [0,\infty)$, let $p_j\downarrow 1$ and fix a sequence $t_i \uparrow t$. Using $|\nabla \varrho_{p_j}| \le 1$ and the uniform convergence of $\varrho_{p_j}$ we get, for each $i$,
\begin{equation}\label{array_buono}
\disp |\{ \varrho_1 < t\}| \ge \disp \lim_{j \ra \infty} \left| \bigcup_{l=j}^\infty \{ \varrho_{p_l} \le t_i\} \right|  \ge \liminf_{j \ra \infty} |\{ \varrho_{p_j} \le t_i\}| \ge \liminf_{j \ra \infty } \int_{\{\varrho_{p_j} < t_i\}} |\nabla \varrho_{p_j}|^{p_j}.
\end{equation}
Using
$$
\int_{\{\varrho_{p_j} < t_i\}} |\nabla \varrho_{p_j}|^{p_j} = V_h(t_i) \VV_1(t_i) = V_h(t_i)
$$
and letting $i \to \infty$ we infer $|\{ \varrho_1 < t\}| \ge V_h(t)$ for every $t$.
\end{proof}

\begin{remark}
\emph{A different way of using $p$-Laplace type equations to investigate the isoperimetric properties of  Riemannian manifolds can also be found in \cite{druet}.
}
\end{remark}

\bibliographystyle{amsplain}

\begin{thebibliography}{99.}



\bibitem{agofogamazzie} V. Agostiniani, M. Fogagnolo and L. Mazzieri, \emph{Sharp geometric inequalities for closed hypersurfaces in manifolds with nonnegative Ricci curvature.} Invent. Math. (2020), online first.

\bibitem{agofogamazzie_2} V. Agostiniani, M. Fogagnolo and L. Mazzieri, \emph{Minkowski Inequalities via Nonlinear Potential Theory.} Available at arXiv:1906.00322.

\bibitem{agomazzie} V. Agostiniani and L. Mazzieri, \emph{On the geometry of the level sets of bounded static potentials.} Comm. Math. Phys. 355 (2017), no. 1, 261-301.


\bibitem{bmpr} B. Bianchini, L. Mari, P. Pucci and M. Rigoli, \emph{Geometric Analysis of Quasilinear Equations and Inequalities on Noncompact Manifolds.} Frontiers in Mathematics, Birkh\"auser, 2021.

\bibitem{bmr2} B. Bianchini, L. Mari and M. Rigoli, \emph{On some aspects of Oscillation Theory and Geometry.} Mem. Amer. Math. Soc. 225 (2013), no.1056.

\bibitem{bmr4} B. Bianchini, L. Mari and M. Rigoli, \emph{Yamabe type equations with sign-changing nonlinearities on the Heisenberg group, and the role of Green functions.} Recent trends in Nonlinear Partial Differential Equations I. Evolution problems "Workshop in honour of Patrizia Pucci's 60th birthday. Contemp. Math. 594, 115-136, Amer. Math. Soc., Providence, RI, 2013.

\bibitem{bmr3} B. Bianchini, L. Mari and M. Rigoli, \emph{Yamabe type equations with sign-changing nonlinearities on non-compact Riemannian manifolds.} J. Funct. Anal. 268 (2015), no.1, 1-72.

\bibitem{bmr5} B. Bianchini, L. Mari and M. Rigoli, \emph{Yamabe type equations with a sign-changing nonlinearity, and the prescribed curvature problem.} J. Differential Equations 260 (2016), no.10, 7416-7497.

\bibitem{bishop} R.L. Bishop, Decomposition of cut loci. Proc. Amer. Math. Soc. 65 (1977), 133-136.

\bibitem{bombierigiusti} E. Bombieri and E. Giusti, \emph{Harnack's inequality for elliptic differential equations on minimal surfaces.} Invent. Math. 15 (1972), 24-46.

\bibitem{borghinimazzieri} S. Borghini and L. Mazzieri, \emph{On the mass of static metrics with positive cosmological constant: I.} Class. Quant. Grav. 35 (2018), no. 12, 43 pp.

\bibitem{brendle} S. Brendle, \emph{Constant mean curvature surfaces in warped product manifolds.} Publ. Math. Inst. Hautes \'Etudes Sci. 117 (2013), 247-269.



\bibitem{carron} G. Carron, \emph{Une suite exacte en $L^2$-cohomologie.} Duke Math. J. 95 (1998), no.2,  343-372.

\bibitem{cheegercolding} J. Cheeger and T. Colding, \emph{Lower bounds on Ricci curvature and the almost rigidity of warped products.} Ann. of Math. (2) 144 (1996), no. 1, 189-237.


\bibitem{chengyau} S. Y. Cheng and S.-T. Yau, \emph{Differential equations on Riemannian manifolds and their geometric applications.} Comm. Pure Appl. Math. 28 (1975), 333-354.

\bibitem{chrusciel} P.T. Chru\'sciel, \emph{Sur les Feuilletages ``Conformement Minimaux" des Variet\'es Riemanniennes de Dimension Trois.} C.R. Acad. Sci. Paris Ser. I Math. 301 (1985), 609-612.

\bibitem{chrusciel2} P.T. Chru\'sciel, \emph{Sur les Coordonn\'ees p-Harmoniques en Relativit\'e G\'en\'erale.} C.R. Acad. Sci. Paris Ser. I Math. 305 (1987) 797-800.

\bibitem{colding} T.H. Colding, \emph{New monotonicity formulas for Ricci curvature and applications, I.} Acta Math. 209 (2012), 229-263.

\bibitem{coldingmini} T.H. Colding and W.P. Minicozzi II, \emph{Harmonic functions with polynomial growth.} J. Differential Geom. 45 (1997), 1-77.

\bibitem{coldingmini_cv} T.H. Colding and W.P. Minicozzi II, \emph{Ricci curvature and monotonicity for harmonic functions.} Calc. Var. Partial Differential Equations 49 (2014), no. 3-4, 1045-1059.

\bibitem{CSC1} T. Coulhon and L. Saloff-Coste, \emph{Vari\'et\'es riemanniennes isom\'etriques \`a  l'infini.} Rev. Mat. Iberoamericana 11 (1995), 687-726.

\bibitem{croke} C.B.Croke, \emph{Some isoperimetric inequalities and eigenvalues estimates.} Ann. Sci. \'Ecole Norm. Sup. 13 (1980), 419-435.

\bibitem{druet} O. Druet, \emph{Sharp local isoperimetric inequalities involving the scalar curvature.} Proc. Amer. Math. Soc. 130 (2002), no. 8, 2351-2361.


\bibitem{fogamazziepina} M. Fogagnolo, L. Mazzieri and A. Pinamonti, \emph{Geometric aspects of $p$-capacitary potentials.} Ann. Inst. H. Poincar\'e Anal. Non Lin\'eaire 36 (2019), no. 4, 1151-1179.

\bibitem{pinchofraas} M. Fraas and Y. Pinchover, \emph{Positive Liouville theorems and asymptotic behavior for p-Laplacian type elliptic equations with a Fuchsian potential.} Confluentes Math. 3 (2011), no. 2, 291-323.

\bibitem{greenewu} R.E. Greene and H. Wu, \emph{Function theory on manifolds which possess a pole.} Lecture Notes in Mathematics, 699. Springer, Berlin, 1979. ii+215 pp.



\bibitem{hebey_book} E. Hebey, \emph{Nonlinear analysis on manifolds: Sobolev spaces and inequalities.} Courant Lecture Notes in Mathematics, 5. New York University, Courant Institute of Mathematical Sciences, New York; American Mathematical Society, Providence, RI, 1999. x+309 pp.

\bibitem{hein} H.-Hein, \emph{Weighted Sobolev inequalities under lower Ricci curvature bounds.} Proc. Amer. Math. Soc. 139 (2011), no. 8, 2943-2955.


\bibitem{HKM} J. Heinonen, T. Kilpel\"ainen and O. Martio, \emph{Nonlinear potential theory of degenerate elliptic equations.} Unabridged republication of the 1993 original, Dover Publications Inc., Mineola, NY (2006), xii+404.

\bibitem{hkst} J. Heinonen, P. Koskela, N. Shanmugalingam and J.T. Tyson, \emph{Sobolev spaces on metric measure spaces. An approach based on upper gradients.} New Mathematical Monographs, 27. Cambridge University Press, Cambridge, 2015. xii+434 pp.

\bibitem{hoffmanspruck} D. Hoffman and J. Spruck, \emph{Sobolev and isoperimetric inequalities for Riemannian submanifolds.} Comm. Pure Appl. Math. 27 (1974), 715-727.

\bibitem{holopainen} I. Holopainen, \emph{Nonlinear potential theory and quasiregular mappings on Riemannian manifolds.} Ann. Acad. Sci. Fenn. Ser. A I Math. Dissertationes 74 (1990), 45pp.

\bibitem{holopainen3} I. Holopainen, \emph{Positive solutions of quasilinear elliptic equations on Riemannian manifolds.} Proc. London Math. Soc. 65 (1992), 651-672.

\bibitem{holopainen2} I. Holopainen, \emph{Volume growth, Green's functions, and parabolicity of ends.} Duke Math. J. 97 (1999), no. 2, 319-346.

\bibitem{HowHutMor} H. Howards, M. Hutchings and F. Morgan, \emph{The isoperimetric problem on surfaces of revolution of decreasing Gauss curvature.} Trans. Amer. Math. Soc., 352(11) (2000), 4889-4909.

\bibitem{huiskenilmanen} G. Huisken and T. Ilmanen, \emph{The inverse mean curvature flow and the Riemannian Penrose inequality.} J. Differential Geom. 59 (2001), no. 3, 353-437.

\bibitem{huiskenilmanen_2} G. Huisken and T. Ilmanen, \emph{Higher regularity of the inverse mean curvature flow.} J. Differential Geom. 80 (2008), no. 3, 433-451

\bibitem{imperimovero} D. Impera, M. Rimoldi and G. Veronelli, \emph{Density problems for second order Sobolev spaces and cut-off functions on manifolds with unbounded geometry.} Int. Math. Res. Not. (2019), online first.


\bibitem{jezierski} J. Jezierski, \emph{Positivity of Mass for Certain Spacetimes with Horizons.} Class.
Quant. Grav. 6 (1989) 1535-1539.

\bibitem{jeziekijo} J. Jezierski and J. Kijowski, \emph{A Simple Argument for Positivity of Gravitational
Energy.} Proceedings of the XV International Conference on Differential Geometric Methods in Theoretical Physics (Clausthal, 1986), World Sci. Pub., Teaneck NJ (1987), 187-194.

\bibitem{jeziekijo2} J. Jezierski and J. Kijowski, \emph{Positivity of Total Energy in General Relativity.}  Phys. Rev. D 36 (1987) 1041-104.

\bibitem{morganjohnson}  D.L. Johnson and F. Morgan \emph{Some sharp isoperimetric theorems for Riemannian manifolds.} Indiana Univ. Math. J. 49 (2000), no. 3, 1017-1041.



\bibitem{kanai} M. Kanai, \emph{Rough isometries and combinatorial approximations of geometries of noncompact
Riemannian manifolds.} J. Math. Soc. Japan 37 (1985), 391-413.


\bibitem{kasue2} A. Kasue, \emph{Ricci curvature, geodesics and some geometric properties of Riemannian manifolds with boundary.} J. Math. Soc. Japan 35 (1983), no. 1, 117-131.

\bibitem{kichenveron} S. Kichenassamy and L. V{\'e}ron, \emph{Singular solutions of the $p$-Laplace equation}, Math. Ann 275 (1986), no. 4, 599-615.

\bibitem{koreschoen} N. Korevaar and R. Schoen, \emph{Global existence theorems for harmonic maps to non-locally compact spaces.} Comm. Anal. Geom. 5 (1997), no. 2, 333-387.

\bibitem{kotschwarni} B Kotschwar and L. Ni, \emph{Local gradient estimates of $p$-harmonic functions, $1/H$-flow, and an entropy formula.} Ann. Sci. \'Ecole Norm. Sup\'er. (4) 42 (2009), no. 1, 1-36.

\bibitem{kura} T. Kura, \emph{On the Green function of the $p$-Laplace equation for Riemannian manifolds.} Proc. Japan Acad. Ser. A Math. Sci. 75 (1999), no. 3, 37-38.

\bibitem{ledoux} M. Ledoux, \emph{On manifolds with non-negative Ricci curvature and Sobolev inequalities.} Comm. Anal. Geom. 7 (1999), 347-353.

\bibitem{lischoen} P. Li, R. Schoen, \emph{$L^p$ and mean value properties of subharmonic functions
on Riemannian manifolds.} Acta Math. 153 (1984), no. 3-4, 279-301.


\bibitem{litam_harm} P. Li and L.F. Tam, \emph{Green's functions, harmonic functions, and volume comparison.} J. Differential Geom. 41 (1995), no. 2, 277-318.


 \bibitem{liwang2} P. Li and J. Wang, \emph{Complete manifolds with positive spectrum.} J. Differential Geom. 58 (2001), no. 3, 501-534.

\bibitem{liwang} P. Li and J. Wang, \emph{Complete manifolds with positive spectrum. II.} J. Differential Geom. 62 (2002), 143-162.



\bibitem{mrs} L. Mari, M. Rigoli and A.G. Setti, \emph{Some remarks on the Green kernel of the $p$-Laplacian on manifolds.} In preparation.

\bibitem{minerbe} V. Minerbe, \emph{Weighted Sobolev inequalities and Ricci flat manifolds.} Geom. Funct. Anal. 18 (2009), no. 5, 1696-1749.

\bibitem{mondinonardulli} A. Mondino and S. Nardulli, \emph{Existence of isoperimetric regions in non-compact Riemannian manifolds under Ricci or scalar curvature conditions.} Comm. Anal. Geom. 24 (2016), no. 1, 115-138.

\bibitem{moser} R. Moser, \emph{The inverse mean curvature flow and p-harmonic functions.} J. Eur. Math. Soc. (JEMS) 9 (2007), no. 1, 77-83.

\bibitem{moser_3} R. Moser, \emph{The inverse mean curvature flow as an obstacle problem.} Indiana Univ. Math. J. 57 (2008), no. 5, 2235-2256.

\bibitem{moser_2} R. Moser, \emph{Geroch monotonicity and the construction of weak solutions of the inverse mean curvature flow.} Asian J. Math. 19 (2015), no. 2, 357-376.

\bibitem{nabervalto} A. Naber and D. Valtorta, \emph{Sharp estimates on the first eigenvalue of the $p$-Laplacian with negative Ricci lower bound.} Math. Z. 277 (2014), no. 3-4, 867-891.

\bibitem{ni} L. Ni, \emph{Mean value theorems on manifolds.} Asian J. Math. 11 (2007), no.2, 277-304.



\bibitem{pinchovertintarev} Y. Pinchover and K. Tintarev, \emph{Ground state alternative for p-Laplacian with potential term.} Calc. Var. Partial Differential Equations 28 (2007) 179-201.

\bibitem{prsmemoirs} S. Pigola, M. Rigoli and A.G. Setti, \emph{Maximum principles on Riemannian manifolds and applications.} Mem. Amer. Math. Soc. 174 (2005), no. 822, x+99 pp.

\bibitem{prs} S. Pigola, M. Rigoli and A.G. Setti, \emph{Vanishing and finiteness results in Geometric Analysis. A generalization of the Bochner technique.} Progress in Mathematics 266, Birk\"auser, 2008, xiv+282 pp.

\bibitem{pigolasettitroyanov} S. Pigola, A.G. Setti and M. Troyanov, \emph{The topology at infinity of a manifold supporting an ${L^{q,p}}$-Sobolev inequality.} Expo. Math. 32 (2014), no.4, 365-383.

\bibitem{pigovero} S. Pigola and G. Veronelli, \emph{Lower volume estimates and Sobolev inequalities.} Proc. Amer. Math. Soc. 138 (2010), no. 12, 4479-4486.


\bibitem{pucciserrin} P. Pucci and J. Serrin, \emph{The maximum principle.} Progress in Nonlinear Differential Equations and their Applications, 73, Birkh\"auser Verlag, Basel, 2007, x+235 pp.


\bibitem{rigolisalvatorivignati} M. Rigoli, M. Salvatori and M. Vignati, \emph{A note on $p$-subharmonic functions on complete manifolds.} Manuscripta Math. 92 (1997), no. 3, 339-359.

\bibitem{rimovero} M. Rimoldi and G. Veronelli, \emph{Extremals of Log Sobolev inequality on non-compact manifolds and Ricci soliton structures.} Calc. Var. Partial Differential Equations 58 (2019), online first.

\bibitem{ritore} Manuel Ritor\'e, \emph{Constant geodesic curvature curves and isoperimetric domains in rotationally symmetric surfaces}. Comm. Anal. Geom. 9 (2001), no. 5, 1093-1138.

\bibitem{saloff} L. Saloff-Coste, \emph{Aspects of Sobolev-type inequalities.} (English summary)
London Mathematical Society Lecture Note Series, 289. Cambridge University Press, Cambridge, 2002. x+190 pp.

\bibitem{Serrin_1} J. Serrin, \emph{Local behaviour of solutions of quasilinear equations.} Acta Math. 111 (1964), 247-302.

\bibitem{Serrin_2} J. Serrin, \emph{Isolated singularities of solutions of quasilinear equations.} Acta Math. 113 (1965), 219-240.

\bibitem{sungwang} C.-J. A. Sung and J. Wang, \emph{Sharp gradient estimate and spectral rigidity for $p$-Laplacian.} Math. Res. Lett. 21 (2014), no. 4, 885-904.


\bibitem{tewo2} D. Tewodrose, \emph{Adimensional weighted Sobolev inequalities in PI spaces.} To appear in Ann. Acad. Sci. Fenn. Math., available at arXiv:2006.10493.

\bibitem{tolksdorff} P. Tolksdorf, \emph{Regularity of a more general class of quasilinear elliptic equations.} J. Differential Equations 51 (1984) 126-150.


\bibitem{troyanov2} M. Troyanov, \emph{Parabolicity of manifolds.} Siberian Adv. Math. 9 (1999), no.4, 125-150.

\bibitem{troyanov} M. Troyanov, \emph{Solving the $p$-Laplacian on manifolds.} Proc. Amer. Math. Soc. 128 (2000), no.2, 541-545.


\bibitem{veron} L. V{\'e}ron, \emph{Singularities of solutions of second order quasilinear equations}, Pitman Research Notes in Mathematics Series 353 (1996), viii+377.

\bibitem{wangzhang} X. Wang and L. Zhang, \emph{Local gradient estimate for $p$-harmonic functions on Riemannian manifolds.} Comm. Anal. Geom. 19 (2011), no.4, 759-771.

\bibitem{wolter} F.E. Wolter, \emph{Distance function and cut loci on a complete Riemannian manifold.} Arch. Math. (Basel) 32 (1979), 92-96.


\bibitem{yau} S.-T. Yau, \emph{Harmonic functions on complete Riemannian manifolds.} Comm. Pure Appl. Math. 28 (1975), 201-228.



\end{thebibliography}

\end{document}